%% file: ModSubcat.tex
\documentclass[11pt]{amsart}

\input{ModSubCatMacros}

\begin{document}

\title{Classifying submodules over monoidal categories}

\author{Hadi Salmasian}
\address[H.S.]{
    Department of Mathematics and Statistics \\
    University of Ottawa \\
    Ottawa, ON, K1N 6N5, Canada
}
\urladdr{\href{https://mysite.science.uottawa.ca/hsalmasi/}{mysite.science.uottawa.ca/hsalmasi/}, \textrm{\textit{ORCiD}:} \href{https://orcid.org/0000-0002-1073-7183}{orcid.org/0000-0002-1073-7183}}
\email{hadi.salmasian@uottawa.ca}

\author{Alistair Savage}
\address[A.S.]{
    Department of Mathematics and Statistics \\
    University of Ottawa \\
    Ottawa, ON, K1N 6N5, Canada
}
\urladdr{\href{https://alistairsavage.ca}{alistairsavage.ca}, \textrm{\textit{ORCiD}:} \href{https://orcid.org/0000-0002-2859-0239}{orcid.org/0000-0002-2859-0239}}
\email{alistair.savage@uottawa.ca}

\author{Yaolong Shen}
\address[Y.S.]{
    Department of Mathematics and Statistics \\
    University of Ottawa \\
    Ottawa, ON, K1N 6N5, Canada
}
\urladdr{\href{https://sites.google.com/virginia.edu/yaolongshen}{sites.google.com/virginia.edu/yaolongshen}, \textrm{\textit{ORCiD}:} \href{https://orcid.org/0000-0002-8840-3394}{orcid.org/0000-0002-8840-3394}}
\email{yshen5@uottawa.ca}

\begin{abstract}
    We study the classification of submodules of module categories over monoidal categories, extending ideas of Coulembier on the classification of tensor ideals in monoidal categories.  We develop a framework that applies to module categories equipped with a twisted cylinder twist, a structure closely related to the twisted reflection equation and quantum symmetric pairs.  Under mild assumptions, we establish an order-preserving bijection between submodules  of a module category $\mathcal{M}$ and submodules of the path-algebra module $\mathcal{M}(\mathds{1},-)$. We show that this correspondence is compatible with idempotent completion and analyze its behavior under decategorification to the split Grothendieck group, giving criteria for classification in terms of indecomposable objects.  As an application, we study the disoriented skein category as a module category over the oriented skein category, describe its indecomposable objects, and obtain a complete classification of its submodules.
\end{abstract}



\ifboolexpr{togl{comments} or togl{details}}{%
  {\color{magenta}DETAILS OR COMMENTS ON}
}{%
}

\maketitle
\thispagestyle{empty}

\tableofcontents

\section{Introduction}

Module categories over monoidal categories play a central role in modern representation theory, categorification, and low-dimensional topology. They arise naturally in a wide range of contexts, including cyclotomic quotients of diagrammatic categories such as Brauer categories, Kauffman categories, Heisenberg categories, and their quantum analogues, as well as in categories of modules over quantum symmetric pairs. From a structural perspective, module categories provide a flexible
framework for encoding boundary conditions, reflection phenomena, and symmetry breaking in tensor-categorical settings.

A fundamental problem in this area is the classification of submodules of a given module category.  Such questions are closely related to the theory of tensor ideals in monoidal categories, to the structure theory of diagrammatic categories, and to the analysis of decategorified invariants. In the monoidal setting, this problem was studied systematically by Coulembier~\cite{Cou18}, who developed a powerful correspondence between tensor ideals of a rigid monoidal category and submodules of a
certain module over the path algebra of the monoidal category.  This correspondence reduces a categorical classification problem to a more tractable ring-theoretic one and has since found numerous applications.

The present paper is inspired by the approach of~\cite{Cou18}, but our focus is fundamentally different. Instead of working in a monoidal category, we study \emph{module categories} over a monoidal category. This shift introduces substantial new difficulties, since module categories lack an internal tensor product, and the techniques of~\cite{Cou18} do not apply directly. In order to adapt these ideas, we introduce and systematically study module categories equipped with a
\emph{twisted cylinder twist}. This structure, which generalizes the notion of a twist in a monoidal category, provides the precise amount of additional compatibility needed to recover a useful diagrammatic calculus and to extend Coulembier's methods to the module category setting.  The notion is closely related to the twisted reflection equation and to structures arising in the theory of quantum symmetric pairs~\cite{BK19}.

\Cref{sec:twists,sec:submod,sec:Kar,sec:decat} contain the most general results of the paper. In \cref{sec:twists} we recall the definition of a twisted cylinder twist and develop general tools for verifying and constructing such twists, with particular attention to the rigid and ribbon settings. In \cref{sec:submod} we establish our main structural result: under mild assumptions, there is an order-preserving bijection between submodules of a module category $\cM$ and submodules of the path-algebra module $\cM(\one,-)$. This correspondence generalizes the main result of~\cite{Cou18} to the module category context. \Cref{sec:Kar} shows that this correspondence is compatible with idempotent completion, allowing one to work in Krull--Schmidt settings. In \cref{sec:decat} we study the decategorification map to the split Grothendieck group and give general criteria under which submodules can be classified purely in terms of indecomposable objects.

In \cref{sec:pagodas} we introduce the notion of a \emph{pagoda}, a technical structure that packages a module category together with auxiliary data controlling its tensor-theoretic behavior. While somewhat technical, pagodas provide concrete and verifiable conditions that ensure the applicability of the general theory and lead to strong finiteness and rigidity results for submodules. 

Finally, in \cref{sec:DSindec,sec:DSapp} we apply the general framework developed in this paper to the disoriented skein category $\DScat = \DScat(q,t)$.  This category, introduced in \cite{SSS25}, is a module category over the oriented skein category $\OScat = \OScat(q,t)$ (also known as the framed HOMFLYPT skein category), which underpins the HOMFLYPT link invariant.  For $m,n \in \N$, $d=m-2n$, we have a commutative diagram of functors
\[
    \begin{tikzcd}
        \DScat(q,q^d) \times \OScat(q,q^d) \arrow[r,"\otimes"] \arrow[d,"\bR_\DScat \times \bR_\OScat"'] & \DScat(q,q^d) \arrow[d,"\bR_\DScat"]
        \\
        \Uis\tmod \times \Us\tmod \arrow[r,"\otimes"] & \Uis\tmod
    \end{tikzcd}
\]
where $\Us$ is $U_q(\fgl(m|2n))$, together with a diagram automorphism $\sigma$, and $\Uis = \Uis(m|2n)$ is a coideal subalgebra (sometimes called an iquantum superalgebra), so that $(\Us,\Uis)$ is a quantum supersymmetric pair of type AI-II.  The horizontal arrows are the module category actions, and `$\tmod$' denotes the category of tensor modules.  It was shown in \cite[Th.~8.1]{SSS25} that the functor
\begin{equation} \label{wizard}
    \bR_\DScat \colon \DScat(q,q^d) \to \Uis(m|2n)\tmod
\end{equation}
is full.  The results of the current paper allow us to say more.  We describe the indecomposable objects of $\DScat$, construct an explicit pagoda structure, and obtain a complete classification of its $\OScat$-submodules.  Precisely, in \cref{ginger}, we show that \emph{every} $\OScat$-submodule of $\DScat$ is the kernel of \cref{wizard} for some $m,n$, and these kernels form a strictly decreasing chain.  In addition, each kernel is generated by a single idempotent.  This allows us, in \cref{presentation}, to give a complete diagrammatic presentation of the category of tensor modules of $\Uis(m|2n)$.

We expect that the results of the current paper will have broad applications to the classification of submodules over tensor categories in other settings of interest in categorical representation theory.  In particular, in current work we are developing analogues of the results of \cite{SSS25} for quantum symmetric pairs of other types and we hope that we will be able to classify the module subcategories appearing in that theory using the results of the current paper.

\subsection*{Acknowledgments}

The research of A.S.\ and H.S.\ was supported by Discovery Grants RGPIN-2023-03842 and RGPIN-2024-04030 from the Natural Sciences and Engineering Research Council of Canada.  Y.S.\ was supported by these grants and the Fields Institute for Research in Mathematical Sciences.  The authors would like to thank Linliang Song for helpful conversations.

\section{Cylinder twists\label{sec:twists}}

In the papers \cite{Har01,tom98,tH98}, tom Dieck and Häring-Oldenburg developed the notion of a monoidal category with a cylinder twist, which is closely related to the concept of a braided module category \cite{Enr07,Bro13}.  In this section, we recall a twisted analogue introduced by Balagovi\'c and Kolb \cite{BK19}.  We then develop some techniques for defining and verifying such a structure, before considering an example of interest to us: the disoriented skein category defined in \cite{SSS25}.  Throughout this section $\kk$ denotes an arbitrary commutative ring.  All categories are assumed to be $\kk$-linear.  For a category $\cC$, we let $\cC(X,Y)$ denote the class of morphisms from $X$ to $Y$, and $1_X$ denote the identity endomorphism of $X$.  We will use the usual string diagram calculus for strict monoidal categories and module categories.  We let $\N$ denote the set of nonnegative integers.

\subsection{Definition and first properties}

Let $\cC$ be a $\kk$-linear strict monoidal category, and let $\cM$ be a $\kk$-linear category.  Let $\cEnd(\cM)$ denote the strict monoidal category of $\kk$-linear endofunctors of $\cM$.  A \emph{right action} of $\cC$ on $\cM$ is a monoidal functor $\bA \colon \cC \to \cEnd(\cM)^\rev$, where $\cD^\rev$ denotes the \emph{reverse} of a monoidal category $\cD$, i.e., we reverse the tensor product.  Given such a right action, we also say that $\cM$ is a \emph{right $\cC$-module category}, or \emph{$\cC$-module} for short; see \cite[\S 7.1]{ENGO15}.  For objects $C \in \Ob(\cC)$ and $M \in \Ob(\cM)$, we set
\begin{equation} \label{beaver}
    M \otimes C := \bA(C)(M).
\end{equation}
We say that the $\cC$-module $\cM$ is \emph{strict} if $\bA$ is a strict monoidal functor.  Equivalently, $\cM$ is strict if
\[
    M \otimes (C_1 \otimes C_2) = (M \otimes C_1) \otimes C_2,\qquad
    \text{for all } M \in \Ob(\cM),\ C_1,C_2 \in \Ob(\cC),
\]
and similarly for morphisms.

Suppose that $(\cM,\cC)$ is a \emph{strict tensor pair}.  This means that
\begin{itemize}
    \item $\cC$ is a braided strict monoidal category,
    \item $\cM$ is a strict $\cC$-module,
    \item $\cC$ is a subcategory of $\cM$ with $\Ob(\cC) = \Ob(\cM)$,
    \item when restricted to $\cC \times \cC$, the module category action is given by the tensor product of $\cC$.
\end{itemize}
An \emph{involution} of $\cC$ is an isomorphism of braided monoidal categories
\[
    \natural \colon \cC \to \cC,\qquad
    X \mapsto X^\natural,\qquad
    f \mapsto f^\natural,\qquad
    X \in \Ob(\cC),\ f \in \Mor(\cC).
\]
such that
\[
    \natural^2 = \id_\cC.
\]
(Our $\natural$ corresponds to the braided tensor equivalence $tw$ in \cite[\S4.2]{BK19}.  We make the slightly stronger assumption that $\natural$ is an isomorphism that squares to the identity.)

A family $\tau = (\tau_X)_{X \in \Ob(\cC)}$ of morphisms
\[
    \tau_X
    =
    \begin{tikzpicture}[centerzero]
        \draw (0,-0.3) \botlabel{X} -- (0,0.3) \toplabel{X^\natural};
        \coupon{0,0}{\tau};
    \end{tikzpicture}
    \in \cM(X,X^\natural)
\]
is called an \emph{$(\cM,\natural)$-endomorphism} of $\cC$ if
\begin{equation} \label{chifrijo}
    \begin{tikzpicture}[centerzero]
        \draw (0,-0.6) \botlabel{X} -- (0,0.6) \toplabel{Y^\natural};
        \coupon{0,-0.2}{f};
        \coupon{0,0.23}{\tau};
    \end{tikzpicture}
    =
    \begin{tikzpicture}[centerzero]
        \draw (0,-0.6) \botlabel{X} -- (0,0.6) \toplabel{Y^\natural};
        \coupon{0,0.2}{f^\natural};
        \coupon{0,-0.3}{\tau};
    \end{tikzpicture}
\end{equation}
for all $f \in \cC(X,Y)$, $X,Y \in \Ob(\cC)$.  In other words, $\tau$ is an $(\cM,\natural)$-endomorphism of $\cC$ if it is a natural transformation $\operatorname{Inc}^\cM_\cC \Rightarrow \operatorname{Inc}^\cM_\cC \circ \natural$, where $\operatorname{Inc}^\cM_\cC$ is the inclusion functor from $\cC$ to $\cM$.  If $\tau_X$ is an isomorphism for all $X \in \Ob(\cC)$, then we say that $\tau$ is an \emph{$(\cM,\natural)$-automorphism} of $\cC$.

\begin{defin}[{\cite[Def.~4.6]{BK19}}]
    Suppose that $(\cM,\cC)$ is a strict tensor pair, and that $\natural$ is an involution of $\cC$.  A \emph{$\natural$-cylinder twist} for $(\cM,\cC)$ is an $(\cM,\natural)$-automorphism $\tau$ of $\cC$ such that
    \begin{equation} \label{taudub}
        \begin{tikzpicture}[centerzero,scale=1.5]
            \draw (-0.15,-0.6) \botlabel{X} -- (-0.15,0.6) \toplabel{X^\natural};
            \draw (0.15,-0.6) \botlabel{Y} -- (0.15,0.6) \toplabel{Y^\natural};
            \genbox{-0.25,-0.1}{0.25,0.1}{\tau};
        \end{tikzpicture}
        =
        \begin{tikzpicture}[centerzero,scale=1.5]
            \draw (0.2,-0.6) \botlabel{Y} -- (0.2,-0.4) \braidup (-0.2,0);
            \draw[wipe] (-0.2,-0.6) -- (-0.2,-0.4) \braidup (0.2,0);
            \draw (-0.2,-0.6) \botlabel{X} -- (-0.2,-0.4) \braidup (0.2,0);
            \draw (0.2,0) \braidup (-0.2,0.4) -- (-0.2,0.6) \toplabel{X^\natural};
            \draw[wipe] (-0.2,0) \braidup (0.2,0.4) -- (0.2,0.6);
            \draw (-0.2,0) \braidup (0.2,0.4) -- (0.2,0.6) \toplabel{Y^\natural};
            \coupon{-0.2,-0.4}{\tau};
            \coupon{-0.2,0}{\tau};
        \end{tikzpicture}
        \qquad \text{for all } X,Y \in \Ob(\cC).
    \end{equation}
\end{defin}

As shown in \cite[(4.6)]{BK19}, it follows from \cref{taudub}, together with the fact that $\tau$ is an $(\cM,\natural)$-automorphism that
\begin{equation} \label{reflection}
    \begin{tikzpicture}[centerzero,scale=1.5]
        \draw (0.2,-0.6) \botlabel{Y} -- (0.2,-0.4) \braidup (-0.2,0);
        \draw[wipe] (-0.2,-0.6) -- (-0.2,-0.4) \braidup (0.2,0);
        \draw (-0.2,-0.6) \botlabel{X} -- (-0.2,-0.4) \braidup (0.2,0);
        \draw (0.2,0) \braidup (-0.2,0.4) -- (-0.2,0.6) \toplabel{X^\natural};
        \draw[wipe] (-0.2,0) \braidup (0.2,0.4) -- (0.2,0.6);
        \draw (-0.2,0) \braidup (0.2,0.4) -- (0.2,0.6) \toplabel{Y^\natural};
        \coupon{-0.2,-0.4}{\tau};
        \coupon{-0.2,0}{\tau};
    \end{tikzpicture}
    =
    \begin{tikzpicture}[centerzero,scale=1.5]
        \draw (0.2,-0.6) \botlabel{Y} -- (0.2,-0.4) \braidup (-0.2,0);
        \draw[wipe] (-0.2,-0.6) -- (-0.2,-0.4) \braidup (0.2,0);
        \draw (-0.2,-0.6) \botlabel{X} -- (-0.2,-0.4) \braidup (0.2,0);
        \draw (0.2,0) \braidup (-0.2,0.4) -- (-0.2,0.6) \toplabel{X^\natural};
        \draw[wipe] (-0.2,0) \braidup (0.2,0.4) -- (0.2,0.6);
        \draw (-0.2,0) \braidup (0.2,0.4) -- (0.2,0.6) \toplabel{Y^\natural};
        \coupon{-0.2,0.4}{\tau};
        \coupon{-0.2,0}{\tau};
    \end{tikzpicture}
\end{equation}
for all $X,Y \in \Ob(\cC)$.  (See also \cref{pulpo}.)  The equation \cref{reflection} is a twisted version of the \emph{reflection equation} for braided module categories.

\begin{rem}
    Taking $X=Y=\one$ in \cref{taudub} shows that $\tau_\one^2 = \tau_\one$.  Since $\tau_\one$ is invertible by assumption, this implies that $\tau_\one = 1_\one$.
\end{rem}

The following example shows how the notion of a cylinder twist generalizes the concept of a balanced monoidal category.

\begin{eg} \label{candy}
    Recall that a \emph{twist} on a monoidal category $\cC$ is a natural family of isomorphisms
    \[
        \theta_X
        =
        \begin{tikzpicture}[centerzero]
            \draw (0,-0.4) \botlabel{X} -- (0,0.4) \toplabel{X};
            \coupon{0,0}{\theta};
        \end{tikzpicture}
        \in \cC(X,X)
    \]
    such that $\theta_\one = 1_\one$ and
    \[
        \begin{tikzpicture}[centerzero,scale=1.5]
            \draw (-0.1,-0.6) \botlabel{X} -- (-0.1,0.6);
            \draw (0.1,-0.6) \botlabel{Y} -- (0.1,0.6);
            \genbox{-0.2,-0.1}{0.2,0.1}{\theta};
        \end{tikzpicture}
        =
        \begin{tikzpicture}[centerzero,scale=1.5]
            \draw (0.2,-0.6) \botlabel{Y} -- (0.2,-0.4) \braidup (-0.2,0);
            \draw[wipe] (-0.2,-0.6) -- (-0.2,-0.4) \braidup (0.2,0);
            \draw (-0.2,-0.6) \botlabel{X} -- (-0.2,-0.4) \braidup (0.2,0);
            \draw (0.2,0) \braidup (-0.2,0.4) -- (-0.2,0.6);
            \draw[wipe] (-0.2,0) \braidup (0.2,0.4) -- (0.2,0.6);
            \draw (-0.2,0) \braidup (0.2,0.4) -- (0.2,0.6);
            \coupon{0.2,0}{\theta};
            \coupon{-0.2,0}{\theta};
        \end{tikzpicture}
        \quad
        \left(
            =
            \begin{tikzpicture}[centerzero,scale=1.5]
                \draw (0.2,-0.6) \botlabel{Y} -- (0.2,-0.4) \braidup (-0.2,0);
                \draw[wipe] (-0.2,-0.6) -- (-0.2,-0.4) \braidup (0.2,0);
                \draw (-0.2,-0.6) \botlabel{X} -- (-0.2,-0.4) \braidup (0.2,0);
                \draw (0.2,0) \braidup (-0.2,0.4) -- (-0.2,0.6);
                \draw[wipe] (-0.2,0) \braidup (0.2,0.4) -- (0.2,0.6);
                \draw (-0.2,0) \braidup (0.2,0.4) -- (0.2,0.6);
                \coupon{-0.2,-0.4}{\theta};
                \coupon{-0.2,0}{\theta};
            \end{tikzpicture}
            =
            \begin{tikzpicture}[centerzero,scale=1.5]
                \draw (0.2,-0.6) \botlabel{Y} -- (0.2,-0.4) \braidup (-0.2,0);
                \draw[wipe] (-0.2,-0.6) -- (-0.2,-0.4) \braidup (0.2,0);
                \draw (-0.2,-0.6) \botlabel{X} -- (-0.2,-0.4) \braidup (0.2,0);
                \draw (0.2,0) \braidup (-0.2,0.4) -- (-0.2,0.6);
                \draw[wipe] (-0.2,0) \braidup (0.2,0.4) -- (0.2,0.6);
                \draw (-0.2,0) \braidup (0.2,0.4) -- (0.2,0.6);
                \coupon{-0.2,0.4}{\theta};
                \coupon{-0.2,0}{\theta};
            \end{tikzpicture}
        \right)
        \qquad \text{for all } X,Y \in \Ob(\cC),
    \]
    where the last two equalities follow from the naturality of the braiding.  A \emph{balanced monoidal category} is a braided monoidal category equipped with a twist.  If $\cC$ is a braided strict monoidal category, then $\cC$ is a strict $\cC$-module with action given by the tensor product, and $(\cC,\cC)$ is a strict tensor pair.  A twist on $\cC$ is the same as an $\id_\cC$-cylinder twist for $(\cC,\cC)$. 
\end{eg}

For the remainder of this subsection, we suppose that $(\cM,\cC)$ is a strict tensor pair, that $\natural$ is an involution of $\cC$, and that $\tau$ is a $\natural$-cylinder twist for $(\cM,\cC)$.  For $X,Y \in \Ob(\cC)$, define
\begin{equation} \label{patacones}
    \begin{tikzpicture}[centerzero]
        \draw (0,-0.5) \botlabel{X} -- (0,0.5);
        \draw (0.4,-0.5) \botlabel{Y} -- (0.4,0.5) \toplabel{Y^\natural};
        \coupon{0.4,0}{\tau};
    \end{tikzpicture}
    :=
    \begin{tikzpicture}[centerzero]
        \draw (0.4,-0.5) \botlabel{Y} \braidup (0,0);
        \draw[wipe] (0,-0.5) \braidup (0.4,0);
        \draw (0,-0.5) \botlabel{X} \braidup (0.4,0) \braidup (0,0.5);
        \draw[wipe] (0,0) \braidup (0.4,0.5);
        \draw (0,0) \braidup (0.4,0.5) \toplabel{Y^\natural};
        \coupon{0,0}{\tau};
    \end{tikzpicture}
    \ .
\end{equation}
Then, by \cref{taudub,reflection}, we have
\[
    \begin{tikzpicture}[centerzero]
        \draw (0,-0.5) \botlabel{X} -- (0,0.5) \toplabel{X^\natural};
        \draw (0.4,-0.5) \botlabel{Y} -- (0.4,0.5) \toplabel{Y^\natural};
        \genbox{-0.2,-0.2}{0.6,0.2}{\tau};
    \end{tikzpicture}
    =
    \begin{tikzpicture}[centerzero]
        \draw (0,-0.5) \botlabel{X} -- (0,0.5);
        \draw (0.4,-0.5) \botlabel{Y} -- (0.4,0.5);
        \coupon{0,0}{\tau};
        \coupon{0.4,0}{\tau};
    \end{tikzpicture}
    =
    \begin{tikzpicture}[centerzero]
        \draw (0,-0.5) \botlabel{X} -- (0,0.5);
        \draw (0.4,-0.5) \botlabel{Y} -- (0.4,0.5);
        \coupon{0,-0.15}{\tau};
        \coupon{0.4,0.15}{\tau};
    \end{tikzpicture}
    \qquad \text{for all } X,Y \in \Ob(\cC).
\]

It follows from \cref{patacones} and the naturality of the braiding on $\cC$ that
\begin{equation} \label{tauinterchange}
    \begin{tikzpicture}[centerzero]
        \draw (0,-0.5) \botlabel{X} -- (0,0.5) \toplabel{Y};
        \draw (0.4,-0.5) \botlabel{Z} -- (0.4,0.5);
        \coupon{0,0.15}{f};
        \coupon{0.4,-0.15}{\tau};
    \end{tikzpicture}
    =
    \begin{tikzpicture}[centerzero]
        \draw (0,-0.5) \botlabel{X} -- (0,0.5) \toplabel{Y};
        \draw (0.4,-0.5) \botlabel{Z} -- (0.4,0.5);
        \coupon{0,-0.15}{f};
        \coupon{0.4,0.15}{\tau};
    \end{tikzpicture}
    \qquad \text{for all } X,Y,Z \in \Ob(\cC),\ f \in \cC(X,Y).
\end{equation}
It also follows from \cref{patacones}, \cref{chifrijo}, and the naturality of the braiding on $\cC$ that
\begin{equation}
    \begin{tikzpicture}[centerzero]
        \draw (0,-0.6) \botlabel{X} -- (0,0.6);
        \draw (0.5,-0.6) \botlabel{Y} -- (0.5,0.6) \toplabel{Z^\natural};
        \coupon{0.5,-0.25}{f};
        \coupon{0.5,0.3}{\tau};
    \end{tikzpicture}
    =
    \begin{tikzpicture}[centerzero]
        \draw (0,-0.6) \botlabel{X} -- (0,0.6);
        \draw (0.5,-0.6) \botlabel{Y} -- (0.5,0.6) \toplabel{Z^\natural};
        \coupon{0.5,0.25}{f^\natural};
        \coupon{0.5,-0.3}{\tau};
    \end{tikzpicture}
    \qquad \text{for all } X,Y,Z \in \Ob(\cC),\ f \in \cC(Y,Z).
\end{equation}

\subsection{Verifying cylinder twists}

Throughout this subsection, we suppose that $(\cM,\cC)$ is a strict tensor pair and that $\natural$ is an involution of $\cC$.  Furthermore, we fix a collection
\[
    \tau = (\tau_X)_{X \in \Ob(C)},\qquad
    \tau_X \text{ is an isomorphism in } \cM(X,X^\natural) \text{ for all } X \in \Ob(\cC),
\]
such that \cref{taudub} is satisfied.  Define the category $\cC_\tau$ of $\cC$ by $\Ob(\cC_\tau) = \Ob(\cC)$ and
\[
    \Mor(\cC_\tau)
    := \{f \in \Mor(\cC) : \cref{chifrijo} \text{ is satisfied}\}.
\]
(It is clear that $\Mor(\cC_\tau)$ is closed under linear combinations and composition.  It is also clear that $1_X \in \Mor(\cC_\tau)$ for all $X \in \Ob(\cC)$.)
Then $\tau$ is a $\natural$-cylinder twist if and only if $\cC_\tau = \cC$.  Our goal in this subsection is to develop some tools for verifying, in special cases, that $\tau$ is a $\natural$-cylinder twist.

\begin{lem} \label{platano}
    The category $\cC_\tau$ is a monoidal subcategory of $\cC$.
\end{lem}

\begin{proof}
    We must show that $\Mor(\cC_\tau)$ is closed under tensor product.  Suppose $f,g \in \Mor(\cC_\tau)$.  Then,
    \[
        \begin{tikzpicture}[centerzero]
            \draw (-0.3,-0.8) -- (-0.3,0.8);
            \draw (0.3,-0.8) -- (0.3,0.8);
            \coupon{-0.3,-0.3}{f};
            \coupon{0.3,-0.3}{g};
            \genbox{-0.5,0.1}{0.5,0.4}{\tau};
        \end{tikzpicture}
        \overset{\cref{taudub}}{=}
        \begin{tikzpicture}[anchorbase,scale=1.5]
            \draw (0.2,-1) -- (0.2,-0.4) \braidup (-0.2,0);
            \draw[wipe] (-0.2,-0.6) -- (-0.2,-0.4) \braidup (0.2,0);
            \draw (-0.2,-1) -- (-0.2,-0.4) \braidup (0.2,0);
            \draw (0.2,0) \braidup (-0.2,0.4);
            \draw[wipe] (-0.2,0) \braidup (0.2,0.4);
            \draw (-0.2,0) \braidup (0.2,0.4);
            \coupon{-0.2,-0.4}{\tau};
            \coupon{-0.2,0}{\tau};
            \coupon{-0.2,-0.7}{f};
            \coupon{0.2,-0.7}{g};
        \end{tikzpicture}
        =
        \begin{tikzpicture}[anchorbase,scale=1.5]
            \draw (0.2,-0.6) -- (0.2,-0.4) \braidup (-0.2,0);
            \draw[wipe] (-0.2,-0.6) -- (-0.2,-0.4) \braidup (0.2,0);
            \draw (-0.2,-0.6) -- (-0.2,-0.4) \braidup (0.2,0);
            \draw (0.2,0) \braidup (-0.2,0.4) -- (-0.2,0.9);
            \draw[wipe] (-0.2,0) \braidup (0.2,0.4);
            \draw (-0.2,0) \braidup (0.2,0.4) -- (0.2,0.9);
            \coupon{-0.2,-0.4}{\tau};
            \coupon{-0.2,0}{\tau};
            \coupon{-0.2,0.6}{f^\natural};
            \coupon{0.2,0.6}{g^\natural};
        \end{tikzpicture}
        \overset{\cref{taudub}}{=}
        \begin{tikzpicture}[anchorbase]
            \draw (-0.3,-0.8) -- (-0.3,0.8);
            \draw (0.3,-0.8) -- (0.3,0.8);
            \coupon{-0.3,0.3}{f^\natural};
            \coupon{0.3,0.3}{g^\natural};
            \genbox{-0.5,-0.1}{0.5,-0.4}{\tau};
        \end{tikzpicture},
    \]
    where the second equality follows from the naturality of the braiding and the assumption that $f,g \in \Mor(\cC_\tau)$.
\end{proof}

The usefulness of \cref{platano} is that, to verify that $\tau$ is a $\natural$-cylinder twist, we only need to check that \cref{chifrijo} holds for all $f$ in a set of morphisms that generate $\cC$ as a monoidal category.

\begin{lem} \label{pulpo}
    For $X,Y \in \Ob(\cC)$, the subcategory $\cC_\tau$ contains the braiding morphism
    \[
        \begin{tikzpicture}[centerzero]
            \draw (0.2,-0.3) \botlabel{Y} \braidup (-0.2,0.3);
            \draw[wipe] (-0.2,-0.3) \braidup (0.2,0.3);
            \draw (-0.2,-0.3) \botlabel{X} \braidup (0.2,0.3);
        \end{tikzpicture}
    \]
    if and only if $\tau$ satisfies the twisted reflection equation \cref{reflection}.  In particular, $\cC_\tau$ is a braided monoidal subcategory of $\cC$ if and only if $\tau$ satisfies the twisted reflection equation \cref{reflection} for all $X,Y \in \Ob(\cC)$.
\end{lem}

\begin{proof}
    The subcategory $\cC_\tau$ is a braided monoidal subcategory of $\cC$ if and only if it contains all the braiding morphisms.  Since $\natural$ is an isomorphism of braided monoidal categories, we have
    \[
        \left(
            \begin{tikzpicture}[centerzero]
                \draw (0.2,-0.3) \botlabel{Y} \braidup (-0.2,0.3);
                \draw[wipe] (-0.2,-0.3) \braidup (0.2,0.3);
                \draw (-0.2,-0.3) \botlabel{X} \braidup (0.2,0.3);
            \end{tikzpicture}
        \right)^\natural
        =
        \begin{tikzpicture}[centerzero]
            \draw (0.2,-0.3) \botlabel{Y^\natural} \braidup (-0.2,0.3);
            \draw[wipe] (-0.2,-0.3) \braidup (0.2,0.3);
            \draw (-0.2,-0.3) \botlabel{X^\natural} \braidup (0.2,0.3);
        \end{tikzpicture}.
    \]
    Then the result follows from the fact that 
    \[
        \begin{tikzpicture}[centerzero,scale=1.5]
            \draw (0.15,-0.6) \botlabel{Y} \braidup (-0.15,0) -- (-0.15,0.6) \toplabel{Y^\natural};
            \draw[wipe] (-0.15,-0.6) \braidup (0.15,0);
            \draw (-0.15,-0.6) \botlabel{X} \braidup (0.15,0) -- (0.15,0.6) \toplabel{X^\natural};
            \genbox{-0.25,-0.1}{0.25,0.1}{\tau};
        \end{tikzpicture}
        \overset{\cref{taudub}}{=}
        \begin{tikzpicture}[anchorbase,scale=1.5]
            \draw (0.2,-0.4) \botlabel{Y} \braidup (-0.2,0);
            \draw[wipe] (-0.2,-0.4) \braidup (0.2,0);
            \draw (-0.2,-0.4) \botlabel{X} \braidup (0.2,0);
            \draw (0.2,0) \braidup (-0.2,0.4);
            \draw[wipe] (-0.2,0) \braidup (0.2,0.4);
            \draw (-0.2,0) \braidup (0.2,0.4) \braidup (-0.2,0.8) \toplabel{Y^\natural};
            \draw[wipe] (-0.2,0.4) \braidup (0.2,0.8);
            \draw (-0.2,0.4) \braidup (0.2,0.8) \toplabel{X^\natural};
            \coupon{-0.2,0.4}{\tau};
            \coupon{-0.2,0}{\tau};
        \end{tikzpicture}
        \qquad \text{and} \qquad
        \begin{tikzpicture}[centerzero,scale=1.5]
            \draw (0.15,-0.6) \botlabel{Y} -- (0.15,0) \braidup (-0.15,0.6) \toplabel{Y^\natural};
            \draw[wipe] (-0.15,0) \braidup (0.15,0.6);
            \draw (-0.15,-0.6) \botlabel{X} -- (-0.15,0) \braidup (0.15,0.6) \toplabel{X^\natural};
            \genbox{-0.25,-0.1}{0.25,0.1}{\tau};
        \end{tikzpicture}
        =
        \begin{tikzpicture}[anchorbase,scale=1.5]
            \draw (0.2,-0.6) \botlabel{Y} -- (0.2,-0.4) \braidup (-0.2,0);
            \draw[wipe] (-0.2,-0.6) -- (-0.2,-0.4) \braidup (0.2,0);
            \draw (-0.2,-0.6) \botlabel{X} -- (-0.2,-0.4) \braidup (0.2,0);
            \draw (0.2,0) \braidup (-0.2,0.4);
            \draw[wipe] (-0.2,0) \braidup (0.2,0.4);
            \draw (-0.2,0) \braidup (0.2,0.4) \braidup (-0.2,0.8) \toplabel{Y^\natural};
            \draw[wipe] (-0.2,0.4) \braidup (0.2,0.8);
            \draw (-0.2,0.4) \braidup (0.2,0.8) \toplabel{X^\natural};
            \coupon{-0.2,-0.4}{\tau};
            \coupon{-0.2,0}{\tau};
        \end{tikzpicture}
        \ .\qedhere
    \]
\end{proof}

\subsection{Constructing cylinder twists\label{subsec:constiwst}}

Throughout this subsection, we suppose that $(\cM,\cC)$ is a strict tensor pair and that $\natural$ is an involution of $\cC$.  In this subsection, we discuss some general strategies for defining a $\natural$-cylinder twist for $(\cM,\cC)$.

Suppose that $\Ob(\cC)$ is freely generated, under tensor product, by a set $S$ of objects.  For $X_1,\dotsc,X_r \in S$, we say that the object $X = X_1 \otimes \dotsb \otimes X_r$ has \emph{length} $r$, which we denote by $\ell(X)$.  Define $\tau_\one = 1_\one$ and suppose that we have isomorphisms
\[
    \tau_X \in \cM(X,X^\natural),\qquad X \in S,
\]
such that \cref{reflection} is satisfied for all $X,Y \in S$.  We then extend the definition of $\tau$ recursively as follows: assuming $\tau_X$ has been defined for objects of length $r$, we define $\tau_{X \otimes Y}$ by \eqref{taudub} for $Y \in S$, $X \in \Ob(\cC)$, $\ell(X) = r$.

\begin{lem}
    For $\tau$ defined as above, the twisted reflection equation \cref{reflection} holds for all $X,Y \in \Ob(\cC)$.
\end{lem}

\begin{proof}
    By assumption and \cref{pulpo}, $\cC_\tau$ contains the braidings for objects in $S$.  By \cref{platano}, $\cC_\tau$ is a monoidal subcategory of $\cC$.  Since arbitrary braidings are built from the braidings of the generating objects via tensor product and composition, it follows that $\cC_\tau$ contains all the braidings.  Then the result follows by another application of \cref{pulpo}.
    \details{
        We give an alternative proof by induction on the lengths of $X$ and $Y$.  It is trivial for $X = \one$ or $Y = \one$, and it holds for $\ell(X) = \ell(Y)=1$ by assumption.  Now suppose that \cref{reflection} holds for all $X,Y \in \Ob(\cC)$ with $\ell(X), \ell(Y) \le r$.  Then, for $Z \in S$ and $X,Y \in \Ob(\cC)$ with $\ell(X) \le r$ and $\ell(Y)=r$, we have
        \[
            \begin{tikzpicture}[centerzero,scale=2]
                \draw (0.2,-0.6) \botlabel{Y \otimes Z} -- (0.2,-0.4) \braidup (-0.2,0);
                \draw[wipe] (-0.2,-0.6) -- (-0.2,-0.4) \braidup (0.2,0);
                \draw (-0.2,-0.6) \botlabel{X} -- (-0.2,-0.4) \braidup (0.2,0);
                \draw (0.2,0) \braidup (-0.2,0.4) -- (-0.2,0.6);
                \draw[wipe] (-0.2,0) \braidup (0.2,0.4) -- (0.2,0.6);
                \draw (-0.2,0) \braidup (0.2,0.4) -- (0.2,0.6);
                \coupon{-0.2,0.4}{\tau};
                \coupon{-0.2,0}{\tau};
            \end{tikzpicture}
            =
            \begin{tikzpicture}[centerzero]
                \draw (0.4,-1.6) \botlabel{Z} -- (0.4,-1.2) \braidup (0,-0.8) \braidup (-0.4,-0.4) -- (-0.4,0);
                \draw[wipe] (-0.4,-0.8) \braidup (0,-0.4);
                \draw (0,-1.6) \botlabel{Y} \braidup (-0.4,-1.2) -- (-0.4,-0.8) \braidup (0,-0.4) -- (0,0) \braidup (-0.4,0.4) -- (-0.4,0.8);
                \draw[wipe] (-0.4,-1.6) \braidup (0,-1.2) \braidup (0.4,-0.8);
                \draw (-0.4,-1.6) \botlabel{X} \braidup (0,-1.2) \braidup (0.4,-0.8) -- (0.4,0.4) \braidup (0,0.8) \braidup (-0.4,1.2) -- (-0.4,1.6);
                \draw[wipe] (-0.4,0) \braidup (0,0.4) \braidup (0.4,0.8);
                \draw (-0.4,0) \braidup (0,0.4) \braidup (0.4,0.8) -- (0.4,1.6);
                \draw[wipe] (-0.4,0.8) \braidup (0,1.2);
                \draw (-0.4,0.8) \braidup (0,1.2) -- (0,1.6);
                \coupon{-0.4,-1}{\tau};
                \coupon{-0.4,-0.2}{\tau};
                \coupon{-0.4,1.4}{\tau};
            \end{tikzpicture}
            \overset{\textup{BR}}{=}
            \begin{tikzpicture}[centerzero]
                \draw (0.4,-1.6) \botlabel{Z} -- (0.4,-1.2) \braidup (0,-0.8) \braidup (-0.4,-0.4) -- (-0.4,0);
                \draw[wipe] (-0.4,-0.8) \braidup (0,-0.4);
                \draw (0,-1.6) \botlabel{Y} \braidup (-0.4,-1.2) -- (-0.4,-0.8) \braidup (0,-0.4) -- (0,0);
                \draw[wipe] (-0.4,-1.6) \braidup (0,-1.2) \braidup (0.4,-0.8);
                \draw (-0.4,-1.6) \botlabel{X} \braidup (0,-1.2) \braidup (0.4,-0.8) -- (0.4,0) \braidup (0,0.4) \braidup (-0.4,0.8) -- (-0.4,1.6);
                \draw[wipe] (0,0) \braidup (0.4,0.4);
                \draw (0,0) \braidup (0.4,0.4) -- (0.4,0.8) \braidup (0,1.2) -- (0,1.6);
                \draw[wipe] (-0.4,0.4) \braidup (0,0.8) \braidup (0.4,1.2);
                \draw (-0.4,0) -- (-0.4,0.4) \braidup (0,0.8) \braidup (0.4,1.2) -- (0.4,1.6);
                \coupon{-0.4,-1}{\tau};
                \coupon{-0.4,-0.2}{\tau};
                \coupon{-0.4,1.4}{\tau};
            \end{tikzpicture}
            \overset{\textup{BR}}{=}
            \begin{tikzpicture}[centerzero]
                \draw (0.4,-1.6) \botlabel{Z} -- (0.4,-0.4) \braidup (0,0) \braidup (-0.4,0.4) -- (-0.4,0.8);
                \draw (0,-1.6) \botlabel{Y} \braidup (-0.4,-1.2) -- (-0.4,-0.8);
                \draw[wipe] (-0.4,-1.6) \braidup (0,-1.2);
                \draw (-0.4,-1.6) \botlabel{X} \braidup (0,-1.2) -- (0,-0.8) \braidup (-0.4,-0.4) -- (-0.4,0);
                \draw[wipe] (-0.4,-0.8) \braidup (0,-0.4) \braidup (0.4,0);
                \draw (-0.4,-0.8) \braidup (0,-0.4) \braidup (0.4,0) -- (0.4,1.2) \braidup (0,1.6);
                \draw[wipe] (-0.4,0) \braidup (0,0.4);
                \draw (-0.4,0) \braidup (0,0.4) -- (0,0.8) \braidup (-0.4,1.2) -- (-0.4,1.6);
                \draw[wipe] (-0.4,0.8) \braidup (0,1.2) \braidup (0.4,1.6);
                \draw (-0.4,0.8) \braidup (0,1.2) \braidup (0.4,1.6);
                \coupon{-0.4,1.4}{\tau};
                \coupon{-0.4,0.6}{\tau};
                \coupon{-0.4,-1}{\tau};
            \end{tikzpicture}
            \overset{\textup{IH}}{=}
            \begin{tikzpicture}[centerzero]
                \draw (0.4,-1.6) \botlabel{Z} -- (0.4,-0.4) \braidup (0,0) \braidup (-0.4,0.4) -- (-0.4,0.8);
                \draw (0,-1.6) \botlabel{Y} \braidup (-0.4,-1.2) -- (-0.4,-0.8);
                \draw[wipe] (-0.4,-1.6) \braidup (0,-1.2);
                \draw (-0.4,-1.6) \botlabel{X} \braidup (0,-1.2) -- (0,-0.8) \braidup (-0.4,-0.4) -- (-0.4,0);
                \draw[wipe] (-0.4,-0.8) \braidup (0,-0.4) \braidup (0.4,0);
                \draw (-0.4,-0.8) \braidup (0,-0.4) \braidup (0.4,0) -- (0.4,1.2) \braidup (0,1.6);
                \draw[wipe] (-0.4,0) \braidup (0,0.4);
                \draw (-0.4,0) \braidup (0,0.4) -- (0,0.8) \braidup (-0.4,1.2) -- (-0.4,1.6);
                \draw[wipe] (-0.4,0.8) \braidup (0,1.2) \braidup (0.4,1.6);
                \draw (-0.4,0.8) \braidup (0,1.2) \braidup (0.4,1.6);
                \coupon{-0.4,-0.2}{\tau};
                \coupon{-0.4,0.6}{\tau};
                \coupon{-0.4,-1}{\tau};
            \end{tikzpicture}
            \overset{\textup{IH}}{=}
            \begin{tikzpicture}[centerzero]
                \draw (0.4,-1.8) \botlabel{Z} -- (0.4,-0.2) \braidup (0,0.2) \braidup (-0.4,0.6) -- (-0.4,1);
                \draw (0,-1.8) \botlabel{Y} -- (0,-1.4) \braidup (-0.4,-1) -- (-0.4,-0.6);
                \draw[wipe] (-0.4,-1.4) \braidup (0,-1);
                \draw (-0.4,-1.8) \botlabel{X} -- (-0.4,-1.4) \braidup (0,-1) -- (0,-0.6) \braidup (-0.4,-0.2) -- (-0.4,0.2);
                \draw[wipe] (-0.4,-0.6) \braidup (0,-0.2) \braidup (0.4,0.2);
                \draw (-0.4,-0.6) \braidup (0,-0.2) \braidup (0.4,0.2) -- (0.4,1.4) \braidup (0,1.8);
                \draw[wipe] (-0.4,0.2) \braidup (0,0.6);
                \draw (-0.4,0.2) \braidup (0,0.6) -- (0,1) \braidup (-0.4,1.4) -- (-0.4,1.8);
                \draw[wipe] (-0.4,1) \braidup (0,1.4) \braidup (0.4,1.8);
                \draw (-0.4,1) \braidup (0,1.4) \braidup (0.4,1.8);
                \coupon{-0.4,0.8}{\tau};
                \coupon{-0.4,-0.8}{\tau};
                \coupon{-0.4,-1.6}{\tau};
            \end{tikzpicture}
            \overset{\textup{BR}}{=}
            \begin{tikzpicture}[centerzero]
                \draw (0.4,-1.6) \botlabel{Z} -- (0.4,-0.8) \braidup (0,-0.4) \braidup (-0.4,0) -- (-0.4,0.4);
                \draw (0,-1.6) \botlabel{Y} -- (0,-1.2) \braidup (-0.4,-0.8) -- (-0.4,-0.4);
                \draw[wipe] (-0.4,-1.2) \braidup (0,-0.8) \braidup (0.4,-0.4);
                \draw (-0.4,-1.6) \botlabel{X} -- (-0.4,-1.2) \braidup (0,-0.8) \braidup (0.4,-0.4) -- (0.4,0.8) \braidup (0,1.2) \braidup (-0.4,1.6);
                \draw[wipe] (-0.4,-0.4) \braidup (0,0);
                \draw[wipe] (-0.4,1.2) \braidup (0,1.6);
                \draw (-0.4,-0.4) \braidup (0,0) -- (0,0.4) \braidup (-0.4,0.8) -- (-0.4,1.2) \braidup (0,1.6);
                \draw[wipe] (-0.4,0.4) \braidup (0,0.8) \braidup (0.4,1.2);
                \draw (-0.4,0.4) \braidup (0,0.8) \braidup (0.4,1.2) -- (0.4,1.6);
                \coupon{-0.4,0.2}{\tau};
                \coupon{-0.4,-0.6}{\tau};
                \coupon{-0.4,-1.4}{\tau};
            \end{tikzpicture}
            =
            \begin{tikzpicture}[centerzero,scale=2]
                \draw (0.2,-0.6) \botlabel{Y \otimes Z} -- (0.2,-0.4) \braidup (-0.2,0);
                \draw[wipe] (-0.2,-0.6) -- (-0.2,-0.4) \braidup (0.2,0);
                \draw (-0.2,-0.6) \botlabel{X} -- (-0.2,-0.4) \braidup (0.2,0);
                \draw (0.2,0) \braidup (-0.2,0.4) -- (-0.2,0.6);
                \draw[wipe] (-0.2,0) \braidup (0.2,0.4) -- (0.2,0.6);
                \draw (-0.2,0) \braidup (0.2,0.4) -- (0.2,0.6);
                \coupon{-0.2,-0.4}{\tau};
                \coupon{-0.2,0}{\tau};
            \end{tikzpicture},
        \]
        where 
        the first and last equalities follow from the recursive definition of $\tau$, the equalities labeled `BR' follow from the braid relation, and the equalities labeled `IH' follow from the induction hypothesis.  Thus, \cref{reflection} holds for $\ell(X) \le r$, $\ell(Y) \le r+1$.  A similar argument now shows that it holds for $\ell(X), \ell(Y) \le r+1$, completing the proof of the inductive step.
    }
\end{proof}

\subsection{The ribbon setting\label{subsec:ribbon}}

Throughout this subsection, we suppose that $(\cM,\cC)$ is a strict tensor pair and that $\cC$ is \emph{rigid}.  We also assume that $\natural$ is an involution of $\cC$ that is an isomorphism of \emph{rigid} braided monoidal categories.

We adopt the common convention of orienting strands in rigid monoidal categories.  For $X \in \cC$, we denote its left dual by $X^\vee$.  This means that we have morphisms
\begin{equation} \label{lcps}
    \begin{tikzpicture}[centerzero={0,-0.1}]
        \draw[<-] (-0.2,0.1) \toplabel{X} -- (-0.2,0) to[out=down,in=down,looseness=2] (0.2,0) -- (0.2,0.1);
    \end{tikzpicture}
    \ \colon \one \to X \otimes X^\vee,\qquad
    \begin{tikzpicture}[centerzero={0,0.1}]
        \draw[<-] (-0.2,-0.1) -- (-0.2,0) to[out=up,in=up,looseness=2] (0.2,0) -- (0.2,-0.1) \botlabel{X};
    \end{tikzpicture}
    \ \colon X^\vee \otimes X \to \one,
\end{equation}
such that
\begin{equation} \label{zigleft}
    \begin{tikzpicture}[centerzero]
        \draw[->] (0.3,-0.5) \botlabel{X} to (0.3,0) to[out=up,in=up,looseness=2] (0,0) to[out=down,in=down,looseness=2] (-0.3,0) to (-0.3,0.5);
    \end{tikzpicture}
    =
    \begin{tikzpicture}[centerzero]
        \draw[->] (0,-0.5) \botlabel{X} to (0,0.5);
    \end{tikzpicture}
    \ ,\qquad
    \begin{tikzpicture}[centerzero]
        \draw[->] (0.3,0.5) to (0.3,0) to[out=down,in=down,looseness=2] (0,0) to[out=up,in=up,looseness=2] (-0.3,0) to (-0.3,-0.5) \botlabel{X};
    \end{tikzpicture}
    \ =
    \begin{tikzpicture}[centerzero]
        \draw[<-] (0,-0.5) \botlabel{X} to (0,0.5);
    \end{tikzpicture}
    \ ,
\end{equation}
where the final string diagram denotes the identity morphism of $X^\vee$.  Our assumption on $\natural$ implies that
\[
    (X^\vee)^\natural = (X^\natural)^\vee,
    \qquad
    \left(
        \begin{tikzpicture}[centerzero]
            \draw[<-] (-0.2,0.1) \toplabel{X} -- (-0.2,0) to[out=down,in=down,looseness=2] (0.2,0) -- (0.2,0.1);
        \end{tikzpicture}
    \right)^\natural
    =\,
    \begin{tikzpicture}[centerzero={0,-0.1}]
        \draw[<-] (-0.2,0.1) \toplabel{X^\natural} -- (-0.2,0) to[out=down,in=down,looseness=2] (0.2,0) -- (0.2,0.1);
    \end{tikzpicture}
    ,\qquad \text{and} \qquad
    \left(
        \begin{tikzpicture}[centerzero]
            \draw[<-] (-0.2,-0.1) -- (-0.2,0) to[out=up,in=up,looseness=2] (0.2,0) -- (0.2,-0.1) \botlabel{X};
        \end{tikzpicture}
    \right)^\natural
    =
    \begin{tikzpicture}[centerzero]
        \draw[<-] (-0.2,-0.1) -- (-0.2,0) to[out=up,in=up,looseness=2] (0.2,0) -- (0.2,-0.1) \botlabel{X^\natural};
    \end{tikzpicture}
    \ ,
\]
for all $X \in \Ob(\cC)$.

Recall that a pivotal braided monoidal category is the same as a balanced rigid monoidal category; see \cite[Cor.4.21]{Sel11}.  If $\cC$ is a pivotal braided monoidal category, then the balanced structure of $\cC$ is given by the twist
\[
    \begin{tikzpicture}[centerzero]
        \draw[->] (0,-0.4) \botlabel{X} -- (0,0.4) \toplabel{X};
        \coupon{0,0}{\theta};
    \end{tikzpicture}
    :=
    \begin{tikzpicture}[centerzero,xscale=-1]
        \draw (0,-0.4) \botlabel{X} to[out=up,in=180] (0.25,0.15) to[out=0,in=up] (0.4,0);
        \draw[wipe] (0.25,-0.15) to[out=180,in=down] (0,0.4);
        \draw[->] (0.4,0) to[out=down,in=0] (0.25,-0.15) to[out=180,in=down] (0,0.4);
    \end{tikzpicture}
    \ ,\qquad X \in \Ob(\cC),
\]
We assume for the remainder of this subsection that $\cC$ is a \emph{ribbon category}, which means that $\cC$ is a pivotal braided monoidal category and
\begin{equation} \label{ribbon}
    \begin{tikzpicture}[centerzero,xscale=-1]
        \draw (0,-0.4) \botlabel{X} to[out=up,in=180] (0.25,0.15) to[out=0,in=up] (0.4,0);
        \draw[wipe] (0.25,-0.15) to[out=180,in=down] (0,0.4);
        \draw[->] (0.4,0) to[out=down,in=0] (0.25,-0.15) to[out=180,in=down] (0,0.4);
    \end{tikzpicture}
    =
    \begin{tikzpicture}[centerzero]
        \draw[->] (0.4,0) to[out=down,in=0] (0.25,-0.15) to[out=180,in=down] (0,0.4);
        \draw[wipe] (0,-0.4) to[out=up,in=180] (0.25,0.15);
        \draw (0,-0.4) \botlabel{X} to[out=up,in=180] (0.25,0.15) to[out=0,in=up] (0.4,0);
    \end{tikzpicture}
    \qquad \text{for all } X \in \Ob(\cC).
\end{equation}
We will assume that the pivotal structure is \emph{strict}, which implies that left and right duals coincide.  (Recall that every pivotal category is equivalent, as a pivotal category, to a strict pivotal category.)  Thus, we also have morphisms
\begin{equation} \label{rcps}
    \begin{tikzpicture}[centerzero={0,-0.1}]
        \draw[->] (-0.2,0.1) -- (-0.2,0) to[out=down,in=down,looseness=2] (0.2,0) -- (0.2,0.1) \toplabel{X};
    \end{tikzpicture}
    \ \colon \one \to X^\vee \otimes X,\qquad
    \begin{tikzpicture}[centerzero={0,0.1}]
        \draw[->] (-0.2,-0.1) \botlabel{X} -- (-0.2,0) to[out=up,in=up,looseness=2] (0.2,0) -- (0.2,-0.1);
    \end{tikzpicture}
    \ \colon X \otimes X^\vee \to \one,
\end{equation}
such that
\begin{equation} \label{zigright}
    \begin{tikzpicture}[centerzero,xscale=-1]
        \draw[->] (0.3,-0.5) \botlabel{X} to (0.3,0) to[out=up,in=up,looseness=2] (0,0) to[out=down,in=down,looseness=2] (-0.3,0) to (-0.3,0.5);
    \end{tikzpicture}
    =
    \begin{tikzpicture}[centerzero]
        \draw[->] (0,-0.5) \botlabel{X} to (0,0.5);
    \end{tikzpicture}
    \ ,\qquad
    \begin{tikzpicture}[centerzero,xscale=-1]
        \draw[->] (0.3,0.5) to (0.3,0) to[out=down,in=down,looseness=2] (0,0) to[out=up,in=up,looseness=2] (-0.3,0) to (-0.3,-0.5) \botlabel{X};
    \end{tikzpicture}
    \ =
    \begin{tikzpicture}[centerzero]
        \draw[<-] (0,-0.5) \botlabel{X} to (0,0.5);
    \end{tikzpicture}
    \ .
\end{equation}
We also assume, as in \cref{subsec:constiwst}, that $\Ob(\cC)$ is freely generated, under tensor product, by a set $S$ of objects.

Define $\tau_\one = 1_\one$ and suppose that we have isomorphisms
\[
    \tau_X \in \cM(X,X^\natural),\qquad X \in S,
\]
such that \cref{reflection} is satisfied for all $X,Y \in S$.  We then define $\tau = (\tau_X)_{X \in \Ob(\cC)}$ as explained in \cref{subsec:constiwst}.  Furthermore, we assume that $X^\vee,X^\natural \in S$ for all $X \in S$.

\begin{assumption} \label{donut}
    Suppose that there exist $t_X, c_X \in \kk^\times$, $X \in S$, such that, for all $X \in S$,
    \begin{gather} \label{mora}
        \begin{tikzpicture}[centerzero,xscale=-1]
            \draw (0,-0.4) \botlabel{X} to[out=up,in=180] (0.25,0.15) to[out=0,in=up] (0.4,0);
            \draw[wipe] (0.25,-0.15) to[out=180,in=down] (0,0.4);
            \draw[->] (0.4,0) to[out=down,in=0] (0.25,-0.15) to[out=180,in=down] (0,0.4);
        \end{tikzpicture}
        = t_X
        \begin{tikzpicture}[centerzero]
            \draw[->] (0,-0.4) \botlabel{X} -- (0,0.4);
        \end{tikzpicture}
        \ ,\qquad
        \begin{tikzpicture}[centerzero]
            \draw (0,-0.5) \botlabel{X} -- (0,-0.3) to[out=up,in=0] (-0.35,0.25) to[out=180,in=up] (-0.55,0.1) -- (-0.55,-0.1);
            \draw[wipe] (-0.55,-0.1) to[out=down,in=180] (-0.35,-0.25) to[out=0,in=down] (0,0.3);
            \draw[->] (-0.55,-0.1) to[out=down,in=180] (-0.35,-0.25) to[out=0,in=down] (0,0.3) -- (0,0.5);
            \coupon{-0.55,0}{\tau};
        \end{tikzpicture}
        = c_X
        \begin{tikzpicture}[centerzero]
            \draw[->] (0,-0.5) \botlabel{X} -- (0,0.5);
            \coupon{0,0}{\tau};
        \end{tikzpicture}
        \ ,\qquad
        \begin{tikzpicture}[centerzero]
            \draw[->] (0,-0.5) \botlabel{X} -- (0,0.5);
            \coupon{0,-0.2}{\tau};
            \coupon{0,0.2}{\tau};
        \end{tikzpicture}
        = c_X t_X
        \begin{tikzpicture}[centerzero]
            \draw[->] (0,-0.5) \botlabel{X} -- (0,0.5);
        \end{tikzpicture}
        \ ,
        \\ \label{sleep}
        c_X = c_{X^\vee} = c_{X^\natural}.
    \end{gather}
\end{assumption}
    
It follows from \cref{donut} that
\begin{equation} \label{fresa}
    \begin{tikzpicture}[centerzero={0,0.2}]
        \draw (0.2,0.6) \toplabel{X} \braiddown (-0.2,0) to[out=down,in=down,looseness=2] (0.2,0);
        \draw[wipe] (0.2,0) \braidup (-0.2,0.6);
        \draw[->] (0.2,0) \braidup (-0.2,0.6);
        \coupon{-0.2,0}{\tau};
    \end{tikzpicture}
    = c_X\,
    \begin{tikzpicture}[centerzero={0,0.2}]
        \draw[<-] (-0.2,0.6) -- (-0.2,0) to[out=down,in=down,looseness=2] (0.2,0) -- (0.2,0.6) \toplabel{X};
        \coupon{-0.2,0.2}{\tau};
    \end{tikzpicture}
    ,\qquad
    \begin{tikzpicture}[centerzero={0,0.2}]
        \draw (0.2,0) \braidup (-0.2,0.6) \toplabel{X};
        \draw[wipe] (0.2,0.6) \braiddown (-0.2,0);
        \draw[<-] (0.2,0.6) \braiddown (-0.2,0) to[out=down,in=down,looseness=2] (0.2,0);
        \coupon{-0.2,0}{\tau};
    \end{tikzpicture}
    = c_X^{-1}\,
    \begin{tikzpicture}[centerzero={0,0.2}]
        \draw[->] (-0.2,0.6) \toplabel{X} -- (-0.2,0) to[out=down,in=down,looseness=2] (0.2,0) -- (0.2,0.6);
        \coupon{-0.2,0.2}{\tau};
    \end{tikzpicture}
    \ ,\qquad
    \begin{tikzpicture}[centerzero={0,-0.2}]
        \draw (0.2,0) \braiddown (-0.2,-0.6) \botlabel{X};
        \draw[wipe] (0.2,-0.6) \braidup (-0.2,0);
        \draw[<-] (0.2,-0.6) \braidup (-0.2,0) to[out=up,in=up,looseness=2] (0.2,0);
        \coupon{-0.2,0}{\tau};
    \end{tikzpicture}
    = c_X\,
    \begin{tikzpicture}[centerzero={0,-0.2}]
        \draw[->] (-0.2,-0.6) \botlabel{X} -- (-0.2,0) to[out=up,in=up,looseness=2] (0.2,0) -- (0.2,-0.6);
        \coupon{-0.2,-0.2}{\tau};
    \end{tikzpicture}
    \ ,\qquad
    \begin{tikzpicture}[centerzero={0,-0.2}]
        \draw (0.2,-0.6) \botlabel{X} \braidup (-0.2,0) to[out=up,in=up,looseness=2] (0.2,0);
        \draw[wipe] (0.2,0) \braiddown (-0.2,-0.6);
        \draw[->] (0.2,0) \braiddown (-0.2,-0.6);
        \coupon{-0.2,0}{\tau};
    \end{tikzpicture}
    = c_X^{-1}\,
    \begin{tikzpicture}[centerzero={0,-0.2}]
        \draw[<-] (-0.2,-0.6) -- (-0.2,0) to[out=up,in=up,looseness=2] (0.2,0) -- (0.2,-0.6) \botlabel{X};
        \coupon{-0.2,-0.2}{\tau};
    \end{tikzpicture}
    \ .
\end{equation}

\begin{lem}
    If \cref{donut} holds, then
    \begin{equation} \label{tamarindo}
        t_X = t_{X^\vee} = t_{X^\natural}
        \qquad \text{for all } X \in S.
    \end{equation}
\end{lem}

\begin{proof}
    Rotating the first equation in \cref{mora} by $180\degree$, then applying \cref{ribbon}, gives the first equality in \cref{tamarindo}.  Applying $\natural$ to the first equation in \cref{mora} gives the second equality in \cref{tamarindo}.
\end{proof}

\begin{lem} \label{yungang}
    If \cref{donut} holds, then
    \[
        \begin{tikzpicture}[centerzero={0,0.1}]
            \draw[<-] (-0.2,-0.1) -- (-0.2,0) to[out=up,in=up,looseness=2] (0.2,0) -- (0.2,-0.1) \botlabel{X};
        \end{tikzpicture}
        ,
        \begin{tikzpicture}[centerzero={0,-0.1}]
            \draw[<-] (-0.2,0.1) \toplabel{X} -- (-0.2,0) to[out=down,in=down,looseness=2] (0.2,0) -- (0.2,0.1);
        \end{tikzpicture}
        \, \in \cC_\tau
        \qquad \text{for all } X \in S.
    \]
\end{lem}

\begin{proof}
    For $X \in \Ob(\cC)$, we have
    \[
        \begin{tikzpicture}[centerzero]
            \draw[<-] (-0.3,-0.4) \botlabel{X} -- (-0.3,-0.15);
            \draw[<-] (-0.3,0.15) node[anchor=south east] {\strandlabel{X^\natural}} to[out=up,in=up,looseness=2] (0.3,0.15);
            \draw[<-] (0.3,-0.15) -- (0.3,-0.4) \botlabel{X};
            \genbox{-0.45,-0.15}{0.45,0.15}{\tau};
        \end{tikzpicture}
        \overset{\cref{taudub}}{=}
        \begin{tikzpicture}[centerzero]
            \draw (-0.2,0) \braiddown (0.2,-0.5) -- (0.2,-0.8) \botlabel{X};
            \draw[wipe] (-0.2,0.5) \braiddown (0.2,0) \braiddown (-0.2,-0.5) -- (-0.2,-0.8);
            \draw[->] (-0.2,0.5) \braiddown (0.2,0) \braiddown (-0.2,-0.5) -- (-0.2,-0.8);
            \draw[wipe] (-0.2,0.5) to[out=up,in=up,looseness=2] (0.2,0.5) \braiddown (-0.2,0);
            \draw (-0.2,0.5) to[out=up,in=up,looseness=2] (0.2,0.5) \braiddown (-0.2,0);
            \coupon{-0.2,0}{\tau};
            \coupon{-0.2,-0.5}{\tau};
        \end{tikzpicture}
        \overset{\cref{mora}}{=} t_X^{-1}\,
        \begin{tikzpicture}[centerzero]
            \draw (-0.2,0.25) \braiddown (0.2,-0.25) -- (0.2,-0.55) \botlabel{X};
            \draw[wipe] (-0.2,-0.55) -- (-0.2,-0.25) \braidup (0.2,0.25) to[out=up,in=up,looseness=2] (-0.2,0.25);
            \draw[<-] (-0.2,-0.55) -- (-0.2,-0.25) \braidup (0.2,0.25) to[out=up,in=up,looseness=2] (-0.2,0.25);
            \coupon{-0.2,0.25}{\tau};
            \coupon{-0.2,-0.25}{\tau};
        \end{tikzpicture}
        \overset{\cref{fresa}}{=} t_X^{-1} c_X^{-1}\,
        \begin{tikzpicture}[centerzero]
            \draw[<-] (-0.2,-0.55) -- (-0.2,0.25) to[out=up,in=up,looseness=2] (0.2,0.25) -- (0.2,-0.55) \botlabel{X};
            \coupon{-0.2,0.25}{\tau};
            \coupon{-0.2,-0.25}{\tau};
        \end{tikzpicture}
        \overset{\cref{mora}}{\underset{\cref{sleep}}{=}}\,
        \begin{tikzpicture}[centerzero]
            \draw[<-] (-0.2,-0.2) -- (-0.2,0) to[out=up,in=up,looseness=2] (0.2,0) -- (0.2,-0.2) \botlabel{X};
        \end{tikzpicture}
        \ .
    \]
    Since $\tau_\one = 1_\one$, it follows that 
    \(
        \begin{tikzpicture}[centerzero]
            \draw[<-] (-0.2,-0.1) -- (-0.2,0) to[out=up,in=up,looseness=2] (0.2,0) -- (0.2,-0.1) \botlabel{X};
        \end{tikzpicture}
        \in \cC_\tau
    \)
    The proof that
    \(
        \begin{tikzpicture}[centerzero={0,-0.1}]
            \draw[<-] (-0.2,0.1) \toplabel{X} -- (-0.2,0) to[out=down,in=down,looseness=2] (0.2,0) -- (0.2,0.1);
        \end{tikzpicture}
        \in \cC_\tau
    \)
    is analogous.
    \details{
        \[
            \begin{tikzpicture}[centerzero]
                \draw[<-] (-0.3,0.4) \toplabel{X^\natural} -- (-0.3,0.15);
                \draw[<-] (-0.3,-0.15) to[out=down,in=down,looseness=2] (0.3,-0.15) node[anchor=north west] {\strandlabel{X}};
                \draw[<-] (0.3,0.15) -- (0.3,0.4) \toplabel{X^\natural};
                \genbox{-0.45,-0.15}{0.45,0.15}{\tau};
            \end{tikzpicture}
            \overset{\cref{taudub}}{\underset{\cref{reflection}}{=}}
            \begin{tikzpicture}[centerzero]
                \draw (-0.2,-0.5) to[out=down,in=down,looseness=2] (0.2,-0.5) \braidup (-0.2,0);
                \draw[wipe] (-0.2,-0.5) \braidup (0.2,0);
                \draw[->] (-0.2,-0.5) \braidup (0.2,0) \braidup (-0.2,0.5) -- (-0.2,0.8) \toplabel{X^\natural};
                \draw[wipe] (-0.2,0) \braidup (0.2,0.5);
                \draw (-0.2,0) \braidup (0.2,0.5) -- (0.2,0.8);
                \coupon{-0.2,0}{\tau};
                \coupon{-0.2,0.5}{\tau};
            \end{tikzpicture}
            \overset{\cref{mora}}{\underset{\cref{ribbon}}{=}} t_X^{-1}\,
            \begin{tikzpicture}[centerzero]
                \draw[<-] (-0.2,0.55) \toplabel{X^\natural} -- (-0.2,0.25) \braiddown (0.2,-0.25) to[out=down,in=down,looseness=2] (-0.2,-0.25);
                \draw[wipe] (-0.2,-0.25) \braidup (0.2,0.25);
                \draw (-0.2,-0.25) \braidup (0.2,0.25) -- (0.2,0.55);
                \coupon{-0.2,-0.25}{\tau};
                \coupon{-0.2,0.25}{\tau};
            \end{tikzpicture}
            \overset{\cref{mora}}{=} t_X^{-1} c_X^{-1}\,
            \begin{tikzpicture}[centerzero]
                \draw[<-] (-0.2,0.55) \toplabel{X^\natural} -- (-0.2,-0.25) to[out=down,in=down,looseness=2] (0.2,-0.25) -- (0.2,0.55);
                \coupon{-0.2,0.25}{\tau};
                \coupon{-0.2,-0.25}{\tau};
            \end{tikzpicture}
            \overset{\substack{\cref{mora} \\ \cref{sleep}}}{\underset{\cref{tamarindo}}{=}}\,
            \begin{tikzpicture}[centerzero]
                \draw[<-] (-0.2,0.2) \toplabel{X^\natural} -- (-0.2,0) to[out=down,in=down,looseness=2] (0.2,0) -- (0.2,0.2);
            \end{tikzpicture}
        \]
    }
\end{proof}

\subsection{The disoriented skein category\label{DStwist}}

We fix elements $q,t \in \kk^\times$, $\delta \in \kk$ such that
\begin{equation} \label{delta}
    \delta (q-q^{-1}) = t-t^{-1}.
\end{equation}
Let $\OScat = \OScat(q,t)$ be the oriented skein category as defined in \cite[\S 2.1]{SSS25}, except that we replace the last relation in \cite[(2.3)]{SSS25} by the relation
\begin{equation} \label{cookie}
    \rightbub = \delta 1_\one.
\end{equation}
This category was first introduced in \cite[\S5.2]{Tur89}, where it was called the \emph{Hecke category}. In \cite[Definition 2.1]{QS19} it is called the \emph{quantized oriented Brauer category}.  It was studied in depth in \cite{Bru17}.  Let $\DScat = \DScat(q,t)$ be the disoriented skein category of \cite[Def.~2.1]{SSS25}.  This is the $\OScat(q,t)$-module generated by the morphisms
\[
    \togupdown \colon \downobj \to \upobj,\qquad
    \togdownup \colon \upobj \to \downobj,
\]
which we call \emph{toggles}, subject to the relations
\begin{gather} \label{DStoggles}
    \begin{tikzpicture}[centerzero]
        \draw[->] (0,-0.5) -- (0,-0.2);
        \draw[<-] (0,-0.2) -- (0,0.2);
        \draw[->] (0,0.2) -- (0,0.5);
        \opendot{0,-0.2};
        \opendot{0,0.2};
    \end{tikzpicture}
    \ =\
    \begin{tikzpicture}[centerzero]
        \draw[->] (0,-0.5)--(0,0.5);
    \end{tikzpicture}
    \ ,\qquad
    \begin{tikzpicture}[centerzero]
        \draw[<-] (0,-0.5) -- (0,-0.2);
        \draw[->] (0,-0.2) -- (0,0.2);
        \draw[<-] (0,0.2) -- (0,0.5);
        \opendot{0,-0.2};
        \opendot{0,0.2};
    \end{tikzpicture}
    \ =\
    \begin{tikzpicture}[centerzero]
        \draw[<-] (0,-0.5)--(0,0.5);
    \end{tikzpicture}
    \ ,\qquad
    \begin{tikzpicture}[centerzero]
        \draw[<-] (0.2,-0.6) -- (0.2,-0.4) \braidup (-0.2,0);
        \draw[wipe] (-0.2,-0.6) -- (-0.2,-0.4) \braidup (0.2,0);
        \draw[<-] (-0.2,-0.6) -- (-0.2,-0.4) \braidup (0.2,0);
        \draw[->] (0.2,0) \braidup (-0.2,0.4) -- (-0.2,0.6);
        \draw[wipe] (-0.2,0) \braidup (0.2,0.4) -- (0.2,0.6);
        \draw[->] (-0.2,0) \braidup (0.2,0.4) -- (0.2,0.6);
        \opendot{-0.2,-0.4};
        \opendot{-0.2,0};
    \end{tikzpicture}
    =
    \begin{tikzpicture}[centerzero]
        \draw[<-] (0.2,-0.6) -- (0.2,-0.4) \braidup (-0.2,0);
        \draw[wipe] (-0.2,-0.6) -- (-0.2,-0.4) \braidup (0.2,0);
        \draw[<-] (-0.2,-0.6) -- (-0.2,-0.4) \braidup (0.2,0);
        \draw[->] (0.2,0) \braidup (-0.2,0.4) -- (-0.2,0.6);
        \draw[wipe] (-0.2,0) \braidup (0.2,0.4) -- (0.2,0.6);
        \draw[->] (-0.2,0) \braidup (0.2,0.4) -- (0.2,0.6);
        \opendot{-0.2,0.4};
        \opendot{-0.2,0};
    \end{tikzpicture}
    \ ,
    \\ \label{DScurls}
    \begin{tikzpicture}[anchorbase]
        \draw[<-] (-0.2,0.2) to[out=down,in=135] (0.2,-0.3);
        \draw[wipe] (-0.2,-0.3) to[out=45,in=down] (0.2,0.2);
        \draw[->] (-0.2,-0.3) to[out=45,in=down] (0.2,0.2) arc(0:180:0.2);
        \opendot{-0.2,0.2};
    \end{tikzpicture}
    \ = q\
    \begin{tikzpicture}[anchorbase]
        \draw[->] (0.2,-0.3) -- (0.2,0.1) arc(0:180:0.2) -- (-0.2,0);
        \draw[<-] (-0.2,0) -- (-0.2,-0.3);
        \opendot{-0.2,0};
    \end{tikzpicture}
    \ ,\qquad
   \begin{tikzpicture}[anchorbase]
        \draw[<-] (-0.2,0.3) to[out=-45,in=up] (0.2,-0.2) arc(360:180:0.2);
        \draw[wipe] (-0.2,-0.2) to[out=up,in=-135] (0.2,0.3);
        \draw[->] (-0.2,-0.2) to[out=up,in=-135] (0.2,0.3);
        \opendot{-0.2,-0.2};
    \end{tikzpicture}
    \ = q\
    \begin{tikzpicture}[anchorbase]
        \draw[<-] (0.2,0.3) -- (0.2,-0.1) arc(360:180:0.2) -- (-0.2,0);
        \draw[->] (-0.2,0) -- (-0.2,0.3);
        \opendot{-0.2,0};
    \end{tikzpicture}
    \ .
\end{gather}
We refer the reader to \cite[\S2.2]{SSS25} for a detailed discussion of this category.

\begin{rem} \label{medusa}
    If $q-q^{-1}$ is invertible, then $\delta$ is determined by \cref{delta}, and \cref{cookie} is equivalent to \cite[(2.3)]{SSS25}.  It is for this reason that we do not explicitly include $\delta$ in the notation $\OScat(q,t)$.  If $q = \pm 1$, then \cref{delta} forces $t = \pm 1$.  It is straightforward to verify that we have an isomorphism of $\kk$-linear categories $\DScat(q,t) \xrightarrow{\cong} \DScat(-q,-t)$ given by multiplying each diagram by $(-1)^c$, where $c$ is the number of crossings in the diagram.  It is also straightforward to verify that we have an isomorphism of $\kk$-linear categories $\DScat(q,t) \xrightarrow{\cong} \DScat(q,-t)$ given by multiplying rightward caps and cups by $-1$.  Thus, we lose no generality in assuming $q=t=1$.  Then $\delta \in \kk$ is arbitrary, and the oriented skein category is equivalent to the oriented Brauer category considered in \cite{BCNR17}.  One easily verifies that the results of \cite[\S\S 2--4]{SSS25} hold in the slightly more general setting considered here.
\end{rem}  

The pair $(\DScat,\OScat)$ is a strict tensor pair.  Let $\natural$ be the braided monoidal isomorphism $\Theta$ of \cite[(2.12)]{SSS25}.  In particular, $\upobj^\natural = \downobj$ and $\downobj^\natural = \upobj$.  We define
\[
    \tau_\one = 1_\one,\qquad
    \tau_\upobj = \togdownup,\qquad
    \tau_\downobj = q^{-1} t \togupdown,
\]
and extend $\tau$ to arbitrary objects as described in \cref{subsec:constiwst}.  Comparing \cite[(2.3),(2.20)]{SSS25} to \cref{mora,fresa}, we see that \cref{donut} holds with $S = \{\upobj,\downobj\}$,
\[
    t_\upobj = t_\downobj = t
    \qquad \text{and} \qquad
    c_\upobj = c_\downobj = q^{-1}.
\]
Thus, by \cref{yungang},
\[
    \leftcap\, ,\ \leftcup\, ,\ \rightcap\, ,\ \rightcup \in \OScat_\tau.
\]

By \cref{pulpo}, the third relation in \cref{DStoggles} implies that $\posupcross \in \OScat_\tau$.  Hence, $\negupcross = (\posupcross)^{-1} \in \OScat_\tau$.  Thus, by \cref{platano}, $\cC_\tau = \cC$, and so $\tau$ is a $\natural$-cylinder twist on $\OScat$.

\begin{rem}
    Note that the convention \cref{patacones} does \emph{not} match our convention \cite[(2.22)]{SSS25} for the toggle $\togupdown$ on arbitrary strands:
    \begin{equation} \label{toggy}
        \begin{tikzpicture}[centerzero]
            \draw[multi] (-0.2,-0.5) \botlabel{\lambda} -- (-0.2,0.5);
            \draw[->] (0.2,0) -- (0.2,0.5);
            \draw[->] (0.2,0) -- (0.2,-0.5);
            \opendot{0.2,0};
        \end{tikzpicture}
        :=
        \begin{tikzpicture}[centerzero]
            \draw[->] (-0.15,0) to[out=up,in=-135] (0.2,0.5);
            \draw[wipe] (0.15,0) to[out=up,in=-45] (-0.2,0.5);
            \draw[multi] (-0.2,-0.5) \botlabel{\lambda} to[out=45,in=down] (0.15,0) to[out=up,in=-45] (-0.2,0.5);
            \draw[wipe] (-0.15,0) to[out=down,in=135] (0.2,-0.5);
            \opendot{-0.15,0};
            \draw[->] (-0.15,0) to[out=down,in=135] (0.2,-0.5);
            \opendot{-0.15,0};
        \end{tikzpicture}
        \colon \lambda \downobj\to \lambda \upobj,
        \qquad
        \begin{tikzpicture}[centerzero]
            \draw[multi] (-0.2,-0.5) \botlabel{\lambda} -- (-0.2,0.5);
            \draw[<-] (0.2,0) -- (0.2,0.5);
            \draw[<-] (0.2,0) -- (0.2,-0.5);
            \opendot{0.2,0};
        \end{tikzpicture}
        :=
        \begin{tikzpicture}[centerzero]
            \draw[<-] (-0.15,0) to[out=down,in=135] (0.2,-0.5);
            \draw[wipe] (-0.2,-0.5) to[out=45,in=down] (0.15,0);
            \draw[multi] (-0.2,-0.5) \botlabel{\lambda} to[out=45,in=down] (0.15,0) to[out=up,in=-45] (-0.2,0.5);
            \draw[wipe] (-0.15,0) to[out=up,in=-135] (0.2,0.5);
            \draw[<-] (-0.15,0) to[out=up,in=-135] (0.2,0.5);
            \opendot{-0.15,0};
            \opendot{-0.15,0};
        \end{tikzpicture}
        \colon \lambda \upobj\to \lambda \downobj,
        \qquad \lambda \in \word.
    \end{equation}
    The convention in \cite[(2.22)]{SSS25} was chosen so that $\togupdown$ is inverse to $\togdownup$ on arbitrary strands.  On the other hand, \cref{patacones} is forced by the axioms of a $\natural$-cylinder braiding.  Thus, while
    \[
        \begin{tikzpicture}[centerzero]
            \draw (0,-0.5) \botlabel{X} -- (0,0.5);
            \draw[->] (0.4,-0.5) -- (0.4,-0.1);
            \draw[<-] (0.4,0.1) -- (0.4,0.5);
            \coupon{0.4,0}{\tau};
        \end{tikzpicture}
        =
        \begin{tikzpicture}[centerzero]
            \draw (0,-0.5) \botlabel{X} -- (0,0.5);
            \draw[->] (0.4,-0.5) -- (0.4,0);
            \draw[<-] (0.4,0) -- (0.4,0.5);
            \opendot{0.4,0};
        \end{tikzpicture}
        \qquad \text{for all } X \in \Ob(\OScat),
    \]
    the morphisms
    \[
        \begin{tikzpicture}[centerzero]
            \draw (0,-0.5) \botlabel{X} -- (0,0.5);
            \draw[<-] (0.4,-0.5) -- (0.4,0);
            \draw[->] (0.4,0) -- (0.4,0.5);
            \coupon{0.4,0}{\tau};
        \end{tikzpicture}
        \qquad \text{and} \qquad
        \begin{tikzpicture}[centerzero]
            \draw (0,-0.5) \botlabel{X} -- (0,0.5);
            \draw[<-] (0.4,-0.5) -- (0.4,0);
            \draw[->] (0.4,0) -- (0.4,0.5);
            \opendot{0.4,0};
        \end{tikzpicture}
    \]
    are different in general.  (They are equal when $X = \one$.)
\end{rem}

\section{Submodule correspondence\label{sec:submod}}

Throughout this section $\kk$ is an arbitrary commutative ring.  If $\cM$ is a right $\cC$-module category, then a \emph{$\cC$-submodule of $\cM$} is a collection
\[
    \cN = (\cN(X,Y))_{X,Y \in \Ob(\cM)}
\]
of $\kk$-submodules $\cN(X,Y) \subseteq \cM(X,Y)$ that is closed under composition with morphisms in $\cM$ and the action of morphisms in $\cC$.  In other words, it is a collection such that
\begin{itemize}
    \item $f \circ g \in \cN(W,Y)$ for all $f \in \cN(X,Y)$, $g \in \cM(W,X)$;
    \item $g \circ f \in \cN(X,Z)$ for all $f \in \cN(X,Y)$, $g \in \cM(Y,Z)$;
    \item $f \otimes 1_Z \in \cN(X \otimes Z, Y \otimes Z)$ for all $f \in \cN(X,Y)$, $Z \in \Ob(\cC)$.
\end{itemize}
\details{
    In the third bullet point above, it is enough to consider closure under the action of $1_Z$, $Z \in \Ob(\cC)$, because of our assumption that $\cN$ is closed under composition with morphisms in $\cM$.  This implies that
    \[
        f \otimes g
        = (f \otimes 1_Z) \circ (1_X \otimes g)
        \in \cN(X \otimes W, Y \otimes Z),\qquad
        f \in \cN(X,Y),\ g \in \cC(W,Z).
    \]
}
In this section, we translate the problem of classifying $\cC$-submodules of $\cM$ to the problem of classifying submodules of a certain module over the path algebra of $\cM$.  Throughout this section we fix the following data and assumptions.

\begin{assumption} \label{jetlag}
    We assume that
    \begin{enumerate}
        \item $(\cM,\cC)$ is a strict tensor pair, $\natural$ is an involution of $\cC$, and $\tau$ is a $\natural$-cylinder twist for $(\cM,\cC)$;
        \item $\cC$ is rigid;
        \item \label{jetlag3} $\cM$ is generated, as a $\cC$-module, by the morphisms $\tau_X$, $X \in \Ob(\cC)$.
    \end{enumerate}
\end{assumption}

\begin{rem} \label{gold}
    Since the category $\cM$ is not monoidal, one cannot form arbitrary tensor products of morphisms in $\cM$.  However, we can do the following.  Suppose $f$ is a morphism in $\cM$.  By \cref{jetlag}\cref{jetlag3}, we can choose a representative diagram for $f$ built from morphisms in $\cC$ and morphisms $\tau_X$, $X \in \Ob(\cC)$.  Let $f_1$ and $f_2$ be two such representatives.  For an arbitrary morphism $g$ in $\cM$, $g \otimes f_1$ and $g \otimes f_2$ are both defined, using \cref{patacones}, but may not be equal.  However, if $g \in \cM(\one,X)$ for some $X \in \Ob(\cC)$, then, using \cref{tauinterchange}, we have
    \[
        \begin{tikzpicture}[anchorbase]
            \draw (0,-0.3) -- (0,0.5);
            \draw (0.5,-0.8) -- (0.5,0.5);
            \coupon{0,-0.3}{g};
            \coupon{0.5,0}{f_1};
        \end{tikzpicture}
        =
        \begin{tikzpicture}[anchorbase]
            \draw (0,0) -- (0,0.5);
            \draw (0.5,-0.8) -- (0.5,0.5);
            \coupon{0,0}{g};
            \coupon{0.5,0}{f_1};
        \end{tikzpicture}
        =
        \begin{tikzpicture}[anchorbase]
            \draw (0,0) -- (0,0.5);
            \draw (0.5,-0.8) -- (0.5,0.5);
            \coupon{0,0}{g};
            \coupon{0.5,-0.45}{f_1};
        \end{tikzpicture}
        =
        \begin{tikzpicture}[anchorbase]
            \draw (0,0) -- (0,0.5);
            \draw (0.5,-0.8) -- (0.5,0.5);
            \coupon{0,0}{g};
            \coupon{0.5,-0.45}{f_2};
        \end{tikzpicture}
        =
        \begin{tikzpicture}[anchorbase]
            \draw (0,0) -- (0,0.5);
            \draw (0.5,-0.8) -- (0.5,0.5);
            \coupon{0,0}{g};
            \coupon{0.5,0}{f_2};
        \end{tikzpicture}
        =
        \begin{tikzpicture}[anchorbase]
            \draw (0,-0.3) -- (0,0.5);
            \draw (0.5,-0.8) -- (0.5,0.5);
            \coupon{0,-0.3}{g};
            \coupon{0.5,0}{f_2};
        \end{tikzpicture}
        \ .
    \]
    Thus, $g \otimes f$ \emph{is} well defined.  This shows that we may place morphisms of $\cM$ in diagrams, so long as they are exposed to the left-hand side of the diagram.  This fact, which will be used in the arguments that follow, is the motivation behind \cref{jetlag}\cref{jetlag3}.
\end{rem}

Let $\Submod_\cC(\cM)$ denote the set of $\cC$-submodules of $\cM$.  Note that a $\cC$-submodule of $\cM$ is \emph{not} necessarily a module category over $\cC$ since it may not be a category.  (For example, it need not contain the identity morphisms of all objects.)  This is analogous to the fact that a tensor ideal of a monoidal category may not be a category.

For any $\cN \in \Submod_\cC(\cM)$, let
\[
    \cN_\one := \cN(\one,-) = \bigoplus_{X \in \Ob(\cM)} \cN(\one,X),
\]
which is naturally a left module over the path algebra
\[
    P(\cM) := \bigoplus_{X,Y \in \Ob(\cM)} \cM(X,Y).
\]
For a nonunital ring $R$ and a left $R$-module $M$, let $\Submod_R(M)$ denote the set of submodules of $M$, partially ordered by inclusion.  For any $N \in \Submod_{P(\cM)}(\cM_\one)$, we have
\[
    N = \bigoplus_{X \in \Ob(\cM)} N(X),
    \qquad \text{where } N(X) = 1_X N \subseteq \cM(\one,X).
\]

\begin{theo}[{cf.\ \cite[Th.~3.1.1]{Cou18}}] \label{ice}
    If $(\cM,\cC)$ satisfies \cref{jetlag}, then the mapping
    \[
        \Psi \colon \Submod_\cC(\cM) \to \Submod_{P(\cM)}(\cM_\one),\qquad
        \cN \mapsto \cN_\one,
    \]
    is an isomorphism of partially ordered sets.
\end{theo}

We need some preliminary results before proving \cref{ice}.  Taking $\cM = \cC$ in \cref{ice} recovers \cite[Th.~3.1.1]{Cou18} (except that we make additional assumptions on $\cC$, which are not needed when $\cM=\cC$), which is based on the ideas of \cite[\S6]{AK02}.  Our proof is very similar to that of \cite[Th.~3.1.1]{Cou18}, the only substantial difference being that we use the language of string diagrams.

For $N \in \Submod_{P(\cM)}(\cM_\one)$ and $X,Y \in \Ob(\cM)$, define
\begin{equation} \label{Phi}
    \Phi(N)(X,Y) :=
    \left\{
        \begin{tikzpicture}[anchorbase]
            \genbox{-0.4,-0.2}{0.4,0.2}{f};
            \draw[->] (-0.2,0.2) -- (-0.2,0.7) \toplabel{Y};
            \draw[->] (0.6,-0.5) \botlabel{X} -- (0.6,0.2) to[out=up,in=up,looseness=2] (0.2,0.2);
        \end{tikzpicture}
        : f \in N(Y \otimes X^\vee)
    \right\}.
\end{equation}

\begin{prop}[{cf.\ \cite[Prop.~3.2.1(a)]{Cou18}}] \label{peach}
    For all $N \in \Submod_{P(\cM)}(\cM_\one)$, we have $\Phi(N) \in \Submod_\cC(\cM)$.
\end{prop}

\begin{proof}
    First we show that $\Phi(N)$ is closed under composition from below.  For $f \in N(Y \otimes X^\vee)$, $g \in \cM(Z,X)$, we have, using \cref{gold},
    \[
        \begin{tikzpicture}[anchorbase]
            \genbox{-0.4,-0.2}{0.4,0.2}{f};
            \draw[->] (-0.2,0.2) -- (-0.2,0.7) \toplabel{Y};
            \draw[->] (0.6,-1) \botlabel{Z} -- (0.6,0.2) node[anchor=west] {\strandlabel{X}} to[out=up,in=up,looseness=2] (0.2,0.2);
            \coupon{0.6,-0.5}{g};
        \end{tikzpicture}
        =
        \begin{tikzpicture}[anchorbase]
            \genbox{-0.4,-0.2}{1.2,0.2}{f};
            \draw[->] (-0.2,0.2) -- (-0.2,1.4) \toplabel{Y};
            \draw[->] (1.4,-0.5) \botlabel{Z} -- (1.4,1) to[out=up,in=up,looseness=2] (1,1) -- (1,0.5) to[out=down,in=down,looseness=2] (0.6,0.5) -- (0.6,0.7) to[out=up,in=up,looseness=2] (0.2,0.7) -- (0.2,0.2);
            \coupon{0.6,0.6}{g};
            \indicate{-0.5,-0.3}{1.3,1};
        \end{tikzpicture}
        =
        \begin{tikzpicture}[anchorbase]
            \genbox{-0.4,-0.2}{0.4,0.2}{h};
            \draw[->] (-0.2,0.2) -- (-0.2,0.7) \toplabel{Y};
            \draw[->] (0.6,-0.5) \botlabel{Z} -- (0.6,0.2) to[out=up,in=up,looseness=2] (0.2,0.2);
        \end{tikzpicture}
        \in \Phi(N)(Z,Y),
    \]
    for $h$ equal to the composite
    \begin{multline*}
        \one
        \xrightarrow{f} Y \otimes X^\vee
        \xrightarrow{1_Y \otimes 1_{X^\vee} \otimes \leftcup_Z} Y \otimes X^\vee \otimes Z \otimes Z^\vee
        \\
        \xrightarrow{1_Y \otimes 1_{X^\vee} \otimes g \otimes 1_{Z^\vee}} Y \otimes X^\vee \otimes X \otimes Z^\vee
        \xrightarrow{1_Y \otimes \leftcap_X \otimes 1_{Z^\vee}} Y \otimes Z^\vee,
    \end{multline*}
    which is the morphism enclosed in the dotted rectangle.  It follows that $\Phi(N)$ is closed under composition from below.
    
    The fact that $\Phi(N)$ is closed under composition from above is straightforward.  It remains to show that $\Phi(N)$ is closed under the right action of $\cC$.  Since $\Phi(N)$ is closed under composition, it suffices to check closure under the action of identity morphisms.  For $f \in N(Y \otimes X^\vee)$ and $Z \in \Ob(\cC)$, we have
    \[
        \begin{tikzpicture}[anchorbase]
            \genbox{-0.4,-0.2}{0.4,0.2}{f};
            \draw[->] (-0.2,0.2) -- (-0.2,0.7) \toplabel{Y};
            \draw[->] (0.6,-0.5) \botlabel{X} -- (0.6,0.2) to[out=up,in=up,looseness=2] (0.2,0.2);
            \draw[->] (0.9,-0.5) \botlabel{Z} -- (0.9,0.7);
        \end{tikzpicture}
        =
        \begin{tikzpicture}[anchorbase]
            \genbox{-0.4,-0.2}{0.6,0.2}{f};
            \draw[->] (-0.2,0.2) -- (-0.2,1.2) \toplabel{Y};
            \draw[->] (0.8,-0.5) \botlabel{X} -- (0.8,0.4) to[out=up,in=up,looseness=2] (0.4,0.4) -- (0.4,0.2);
            \draw[->] (1.1,-0.5) \botlabel{Z} -- (1.1,0.4) to[out=up,in=up,looseness=1.5] (0.2,0.4) to[out=down,in=down,looseness=1.5] (0,0.4) -- (0,1.2);
            \indicate{-0.5,-0.3}{0.7,0.45};
        \end{tikzpicture}
        =
        \begin{tikzpicture}[anchorbase]
            \genbox{-0.4,-0.2}{0.4,0.2}{g};
            \draw[->] (-0.2,0.2) -- (-0.2,0.7) \toplabel{Y \otimes Z};
            \draw[->] (0.6,-0.5) \botlabel{X \otimes Z} -- (0.6,0.2) to[out=up,in=up,looseness=2] (0.2,0.2);
        \end{tikzpicture}
    \]
    for $g$ equal to the composite
    \[
        \one \xrightarrow{f} Y \otimes X^\vee = Y \otimes \one \otimes X^\vee \xrightarrow{1_Y \otimes \leftcup_Z \otimes 1_{X^\vee}} Y \otimes Z \otimes Z^\vee \otimes X^\vee,
    \]
    which is the morphism enclosed in the dotted rectangle.  Since $N \in \Submod_{P(\cM)}(\cM_\one)$, we have $g \in N$.

\end{proof}

\begin{prop}[{cf.\ \cite[Prop.~3.2.1(b)]{Cou18}}] \label{cold}
    The mappings $\Phi$ and $\Psi$ are mutually inverse.
\end{prop}

\begin{proof}
    Suppose $\cN \in \Submod_\cC(\cM)$.  For $f \in \cN(X,Y)$, we have
    \[
        \begin{tikzpicture}[anchorbase]
            \draw[->] (0,-0.5) \botlabel{X} -- (0,0.5) \toplabel{Y};
            \coupon{0,0}{f};
        \end{tikzpicture}
        =
        \begin{tikzpicture}[anchorbase]
            \draw[->] (0.4,-0.7) \botlabel{X} -- (0.4,0.2) to[out=up,in=up,looseness=2] (0,0.2) -- (0,-0.2) to[out=down,in=down,looseness=2] (-0.4,-0.2) -- (-0.4,0.7) \toplabel{Y};
            \coupon{-0.4,0}{f};
            \indicate{-0.65,-0.5}{0.2,0.3};
        \end{tikzpicture}
        \in \Phi(\cN_\one)(X,Y) = \big( \Phi \circ \Psi(\cN) \big)(X,Y).
    \]
    Thus, $\cN \subseteq \Phi \circ \Psi(\cN)$.  Now suppose that $f \in \big( \Phi \circ \Psi(\cN) \big)(X,Y) = \Phi(\cN_\one)(X,Y)$.  Then there exists $g \in \cN(\one, Y \otimes X^\vee)$ such that
    \[
        \begin{tikzpicture}[anchorbase]
            \draw[->] (0,-0.5) \botlabel{X} -- (0,0.5) \toplabel{Y};
            \coupon{0,0}{f};
        \end{tikzpicture}
        =
        \begin{tikzpicture}[anchorbase]
            \genbox{-0.4,-0.2}{0.4,0.2}{g};
            \draw[->] (-0.2,0.2) -- (-0.2,0.7) \toplabel{Y};
            \draw[->] (0.6,-0.5) \botlabel{X} -- (0.6,0.2) to[out=up,in=up,looseness=2] (0.2,0.2);
        \end{tikzpicture}
        \ .
    \]
    This immediately implies that $f \in \cN$, since $\cN$ is closed under composition and the right action of $\cC$.  Thus $\Psi \circ \Phi(\cN) \subseteq \cN$, and so $\Phi \circ \Psi(\cN) = \cN$, as desired.  The fact that $\Psi \circ \Phi(N) = N$ for all $N \in \Submod_{P(\cM)}(\cM_\one)$ is immediate.
\end{proof}

\begin{proof}[Proof of \cref{ice}]
    It is straightforward to verify that $\Psi$ respects the partial orders.  Thus, the result follows from \cref{cold}.
\end{proof}

\begin{cor}[{cf.\ \cite[Cor.~3.1.3]{Cou18}}] \label{bamboo}
    Assume that $(\cM,\cC)$ and $(\cM',\cC)$ both satisfy \cref{jetlag}.  An equivalence $\bF \colon \cM \to \cM'$ of preadditive categories with $\bF(\one) = \one$ induces an isomorphism of partially ordered sets $\bF \colon \Submod_\cC(\cM') \to \Submod_\cC(\cM)$, where
    \[
        \bF(\cN)(X,Y)
        =
        \left\{
            \begin{tikzpicture}[anchorbase]
                \genbox{-0.4,-0.2}{0.4,0.2}{g};
                \draw[->] (-0.2,0.2) -- (-0.2,0.7) \toplabel{Y};
                \draw[->] (0.6,-0.5) \botlabel{X} -- (0.6,0.2) to[out=up,in=up,looseness=2] (0.2,0.2);
            \end{tikzpicture}
            : g \in \bF^{-1} \Big( \cN \big( \one,\bF(X^\vee \otimes Y) \big) \Big)
        \right\}.
    \]
\end{cor}

\begin{proof}
    The proof is the same as that of \cite[Cor.~3.1.3]{Cou18}.
\end{proof}

\begin{cor} \label{dns}
    Suppose that $(\cM,\cC)$ satisfies \cref{jetlag}, that $\cM(\one,\one) = \kk 1_\one \cong \kk$ is a field, and that $\cM$ is semisimple.  Then the only $\cC$-submodules of $\cM$ are $\cM$ and the zero submodule.
\end{cor}

\begin{proof}
    Since $\cM(\one,\one) \cong \kk$, we see that $\one$ is a simple object of $\cM$.  Suppose that $\cN$ is a nonzero $\cC$-submodule of $\cM$.  By \cref{ice}, there exists $X \in \Ob(\cM)$ such that $\cN(\one,X) \ne 0$.  Because $\cM$ is semisimple, this implies that $X$ contains a summand isomorphic to $\one$.  Composing with projection onto this summand, we see that $1_\one \in \cN$, which implies that $\cN_\one = \cM_\one$.  By \cref{ice}, we have $\cN = \cM$.
\end{proof}

The following lemma will be used later on.

\begin{lem}[{cf.~\cite[\S2.3.2]{Cou18}}] \label{bei}
    For every $X \in \Ob(\cC)$, $\cM(\one,X \otimes X^\vee) $ is generated by $\leftcup_X$ as an $\cM(X \otimes X^\vee,X \otimes X^\vee)$-module.
\end{lem}

\begin{proof}
    Suppose $f \in \cM(\one,X \otimes X^\vee)$.  Then
    \[
        \begin{tikzpicture}[centerzero]
            \draw[->] (-0.2,-0.3) -- (-0.2,0.3) \toplabel{X};
            \draw[<-] (0.2,-0.15) -- (0.2,0.3) \toplabel{X};
            \genbox{-0.4,-0.5}{0.4,-0.15}{f};
        \end{tikzpicture}
        =
        \begin{tikzpicture}[centerzero={0,0.2},xscale=-1]
            \draw[<-] (0,0.15) -- (0,0.2);
            \draw (0,0.2) to[out=up,in=up,looseness=2] (-0.4,0.2) -- (-0.4,-0.3) to [out=down,in=down,looseness=2] (-0.8,-0.3) -- (-0.8,0.7) \toplabel{X};
            \draw[->] (0.4,0.15) -- (0.4,0.7) \toplabel{X};
            \genbox{-0.2,-0.2}{0.6,0.15}{f};
            \indicate{-1,-0.3}{0.8,0.5};
        \end{tikzpicture}
        \in \cM(X \otimes X^\vee,X \otimes X^\vee) \circ \leftcup_X,
    \]
    as desired.
\end{proof}

\section{Idempotent completion\label{sec:Kar}}

In this section we discuss the compatibility of various notions of interest to us with the operation of taking the idempotent completion of a category.  Throughout this section $\kk$ is an arbitrary commutative ring.  Recall that the \emph{idempotent completion}, or \emph{Karoubi envelope}, of a category $\cD$ is the category $\Kar(\cD)$ whose objects are pairs $X = (X_0,e_X)$ with $X_0 \in \Ob(\cD)$ and $e_X \in \cD(X_0,X_0)$ an idempotent, and whose morphisms are given by
\[
    \Kar(\cD) (X,Y)
    = e_Y \cD(X_0,Y_0) e_X.
\]
We view $\cD$ as a subcategory of $\Kar(\cD)$ by identifying $X_0 \in \Ob(\cD)$ with $(X_0,1_{X_0}) \in \Ob(\Kar(\cD))$.

Suppose that $\cM$ is a right $\cC$-module category.  Then $\Kar(\cM)$ is a right $\Kar(\cC)$-module category, with action given on objects by
\[
    X \otimes Y = (X_0 \otimes Y_0, e_X \otimes e_Y),\qquad
    X \in \Ob(\Kar(\cM)),\ Y \in \Ob(\Kar(\cC)),
\]
and action on morphisms the same as for the action of $\cC$ on $\cM$.

Suppose $\cN \in \Submod_\cC(\cM)$.  For $X,Y \in \Ob(\Kar(\cM))$, define
\begin{equation} \label{rice}
    \Kar(\cN)(X,Y)
    = e_Y \cN(X_0,Y_0) e_X
    \subseteq \Kar(\cM)(X,Y).
\end{equation}
It is straightforward to verify that $\Kar(\cN) = \big( \Kar(\cN)(X,Y) \big)_{X,Y \in \Ob(\Kar(\cM))}$ is a $\Kar(\cC)$-submodule of $\Kar(\cM)$.

\begin{prop} \label{pork}
    If $\cM$ is a right $\cC$-module category, then the map
    \[
        \Kar \colon \Submod_\cC(\cM) \to \Submod_{\Kar(\cC)}(\Kar(\cM))
    \]
    is an isomorphism of partially ordered sets.
\end{prop}

\begin{proof}
    It is clear that the map $\Kar$ respects the partial order.  To show that it is an isomorphism, we will describe the inverse map.  For $\cN \in \Submod_{\Kar(\cC)}(\Kar(\cM))$, define
    \[
        \Res_\cM(\cN) = \big(\Res_\cM(\cN)(X,Y)\big)_{X,Y \in \Ob(\cM)},\quad \text{where} \quad
        \Res_\cM(\cN)(X,Y) = \cN(X,Y).
    \]
    It is straightforward to verify that $\Res_\cM(\cN) \in \Submod_\cC(\cM)$ for all $\cN \in \Submod_{\Kar(\cC)}(\Kar(\cM))$ and that $\Res_\cM (\Kar(\cN)) = \cN$ for all $\cN \in \Submod_\cC(\cM)$.
    
    Now suppose that $\cN \in \Submod_{\Kar(\cC)}{\Kar(\cM)}$.  For $X,Y \in \Ob(\Kar(\cM))$,
    \[
        \Kar (\Res_\cM(\cN))(X,Y)
        = e_Y \Res_\cM(\cN)(X_0,Y_0) e_X
        = e_Y \cN(X_0,Y_0) e_X.
    \]
    On the other hand, since $e_Z = 1_Z$ for $Z \in \Ob(\Kar(\cM))$, we have
    \[
        \cN(X,Y)
        = e_Y \cN(X,Y) e_X
        = e_Y \big( e_Y \cN(X,Y) e_X \big) e_X
        \subseteq e_Y \cN(X_0,Y_0) e_X,
    \]
    where the inclusion uses the fact that $e_X \in \Kar(\cM) \big( X_0, X \big)$ and $e_Y \in \Kar(\cM)(Y,Y_0)$, together with the closure of $\cN$ under composition with morphisms in $\Kar(\cM)$.  We also have the reverse inclusion
    \[
        e_Y \cN(X_0,Y_0) e_X
        = e_Y \big( e_Y \cN(X_0,Y_0) e_X \big) e_X
        \subseteq e_Y \cN(X,Y) e_X,
    \]
    since $e_X \in \Kar(\cM)(X,X_0)$ and $e_Y \in \Kar(\cM)(Y_0,Y)$.  Thus, we have $\Kar(\Res_\cM)(\cN) = \cN$, as desired.
\end{proof}

For $N \in \Submod_{P(\cM)}(\cM_\one)$, define
\begin{equation} \label{noodle}
    \Kar(N) = \bigoplus_{X \in \Ob(\Kar(\cM))} \Kar(N)(X),
    \quad \text{where} \quad
    \Kar(N)(X) = e_X N(X_0).
\end{equation}
It is straightforward to verify that $\Kar(N) \in \Submod_{P(\Kar(\cM))}(\Kar(\cM)_\one)$.

\begin{prop}
    The map
    \[
        \Kar \colon \Submod_{P(\cM)}(\cM_\one) \to \Submod_{P(\Kar(\cM))}(\Kar(\cM)_\one)
    \]
    is an isomorphism of partially ordered sets.
\end{prop}

\begin{proof}
    It is clear that the map respects the partial order.  To show it is an isomorphism, we will describe the inverse map.  For $N \in \Submod_{P(\Kar(\cM))}(\Kar(\cM)_\one)$, define
    \[
        \Res_\cM(N) = \bigoplus_{X \in \Ob(\cM)} \Res_\cM(N)(X),
        \quad \text{where} \quad
        \Res_\cM(N)(X) = N(X).
    \]
    It is straightforward to verify that $\Res_\cM(N) \in \Submod_{P(\cM)}(\cM_\one)$ for all $N \in \Submod_{P(\Kar(\cM))}(\Kar(\cM)_\one)$ and that $\Res_\cM(\Kar(N)) = N$ for all $N \in \Submod_{P(\cM)}(\cM_\one)$.  The proof that $\Kar(\Res_\cM(N)) = N$ for all $N \in \Submod_{P(\Kar(\cM))}(\Kar(\cM)_\one)$ is analogous to the last paragraph of the proof of \cref{pork}.
    \details{
        Suppose that $N \in \Submod_{P(\Kar(\cM))}(\Kar(\cM)_\one)$.  For $X \in \Kar(\cM)$, we have
        \[
            \Kar(\Res_\cM(N))(X)
            = e_X \Res_\cM(N)(X_0)
            = e_X N(X_0).
        \]
        On the other hand, since $e_X = 1_X$ for $X \in \Kar(\cM)$, we have
        \[
            N(X)
            = e_X N(X)
            = e_X \big( e_X N(X) \big)
            \subseteq e_X N(X_0),
        \]
        where the inclusion uses the fact that $e_X \in \Kar(\cM)(X,X_0)$ and the closure of $N$ under the action of the path algebra $P(\Kar(\cM)$.  We also have the reverse inclusion
        \[
            e_X N(X_0)
            = e_X \big( e_X N(X_0) \big)
            \subseteq e_X N(X)
            = N(X),
        \]
        since $e_X \in \Kar(\cM)(X_0,X)$.  Thus, $\Kar(\Res_\cM(N)) = N$, as desired.
    }
\end{proof}

For $\cN \in \Submod_\cC(\cM)$, we have
\[
    \Kar(\cN)_\one
    = \bigoplus_{X \in \Kar(\cM)} \Kar(\cN) (\one, X)
    \overset{\cref{rice}}{=} \bigoplus_{X \in \Kar(\cM)} e_X \cN(\one,X_0)
    \overset{\cref{noodle}}{=} \Kar(\cN_\one).
\]
Thus, the following diagram commutes:
\begin{equation} \label{dumpling}
    \begin{tikzcd}
        \Submod_\cC(\cM) \arrow[r, "\Psi"] \arrow[d, "\Kar"', "\cong"] & \Submod_{P(\cM)}(\cM_\one) \arrow[d, "\Kar"', "\cong"] \\
        \Submod_\cC(\Kar(\cM)) \arrow[r, "\Psi"] & \Submod_{P(\Kar(\cM))}(\Kar(\cM)_\one)
    \end{tikzcd}
\end{equation}

If $\cD$ is a left rigid monoidal category, then $\Kar(\cD)$ is also left rigid.  The left dual of $X = (X_0,e_X) \in \Kar(\cD)$ is $X^\vee = (X_0^\vee,e_X^\vee)$, where $X_0^\vee$ is the left dual of $X_0$ in $\cD$ and
\[
    \begin{tikzpicture}[centerzero]
        \draw[<-] (0,-0.6) -- (0,0.6);
        \coupon{0,0}{e_X^\vee};
    \end{tikzpicture}
    =
    \begin{tikzpicture}[centerzero]
        \draw[<-] (-0.4,-0.6) -- (-0.4,0.1) arc(180:0:0.2) -- (0,-0.1) arc(180:360:0.2) -- (0.4,0.6);
        \coupon{0,0}{e_X};
    \end{tikzpicture}
\]
is the left mate of $e_X$.  The unit and counit are given by
\begin{equation} \label{mug}
    \begin{tikzpicture}[centerzero]
        \draw[<-] (-0.3,-0.2) -- (-0.3,0) arc(180:0:0.3) -- (0.3,-0.2) \botlabel{X};
    \end{tikzpicture}
    :=
    \begin{tikzpicture}[centerzero]
        \draw[<-] (-0.3,-0.4) \botlabel{X_0} -- (-0.3,0.1) arc(180:0:0.3) -- (0.3,-0.4) \botlabel{X_0};
        \coupon{0.3,0}{e_X};
    \end{tikzpicture}
    =
    \begin{tikzpicture}[centerzero]
        \draw[<-] (-0.3,-0.4) \botlabel{X_0} -- (-0.3,0.1) arc(180:0:0.3) -- (0.3,-0.4) \botlabel{X_0};
        \coupon{0.3,0}{e_X};
        \coupon{-0.3,0}{e_X^\vee};
    \end{tikzpicture}
    \qquad \text{and} \qquad
    \begin{tikzpicture}[centerzero]
        \draw[<-] (-0.3,0.2) -- (-0.3,0) arc(180:360:0.3) -- (0.3,0.2) \toplabel{X};
    \end{tikzpicture}
    :=
    \begin{tikzpicture}[centerzero]
        \draw[<-] (-0.3,0.4) \toplabel{X_0} -- (-0.3,-0.1) arc(180:360:0.3) -- (0.3,0.4) \toplabel{X_0};
        \coupon{-0.3,0}{e_X};
    \end{tikzpicture}
    =
    \begin{tikzpicture}[centerzero]
        \draw[<-] (-0.3,0.4) \toplabel{X_0} -- (-0.3,-0.1) arc(180:360:0.3) -- (0.3,0.4) \toplabel{X_0};
        \coupon{-0.3,0}{e_X};
        \coupon{0.3,0}{e_X^\vee};
    \end{tikzpicture}
    \ .
\end{equation}

\begin{cor} \label{soup}
    If $(\cM,\cC)$ satisfies \cref{jetlag}, then the mapping
    \[
        \Psi \colon \Submod_{\Kar(\cC)}(\Kar(\cM)) \to \Submod_{P(\Kar(\cC))}(\Kar(\cM)_\one),\qquad
        \cN \to \cN_\one,
    \]
    is an isomorphism of partially ordered sets, with inverse
    \[
        \Psi^{-1} \colon \Submod_{P(\Kar(\cC))}(\Kar(\cM)_\one) \to \Submod_{\Kar(\cC)}(\Kar(\cM))
    \]
    given by
    \begin{equation} \label{corn}
        \Psi^{-1}(N)(X,Y) :=
        \left\{
            \begin{tikzpicture}[anchorbase]
                \genbox{-0.4,-0.2}{0.4,0.2}{f};
                \draw[->] (-0.2,0.2) -- (-0.2,1) \toplabel{Y_0};
                \draw[->] (0.6,-0.8) \botlabel{X_0} -- (0.6,0.2) to[out=up,in=up,looseness=2] (0.2,0.2);
                \coupon{-0.2,0.6}{e_Y};
                \coupon{0.6,-0.5}{e_X};
            \end{tikzpicture}
            : f \in N(Y_0 \otimes X_0^\vee)
        \right\}.
        \end{equation}
\end{cor}

\begin{proof}
    This follows from \cref{ice} and the commutativity of \cref{dumpling}.
\end{proof}

\begin{rem}
    The importance of \cref{soup} is that it extends the setting of \cref{ice}.  For $(\cM,\cC)$ satisfying \cref{jetlag}, it does \emph{not} follow that $(\Kar(\cM), \Kar(\cC))$ satisfies \cref{jetlag}.  In particular, we may have primitive idempotents in $\cC$ that split in $\cM$, meaning that $\Ob(\Kar(\cM)) \ne \Ob(\Kar(\cC))$, and so $(\Kar(\cM), \Kar(\cC))$ is not a strict tensor pair.
\end{rem}

\section{The decategorification map\label{sec:decat}}

We now pass from categorical data to Grothendieck groups and study how submodules behave under decategorification.  The goal of this section is to adapt \cite[\S4]{Cou18} to the setting of module categories.  We will use the following notation:
\begin{itemize}
    \item calligraphic characters $\cN$, $\cM$, etc.\ for categories and submodules of module categories,
    \item roman characters $N$, $M$, etc.\ for modules over path algebras of categories,
    \item typewriter font $\tN$, $\mathtt{M}$, etc.\ for modules over Grothendieck rings.
\end{itemize}
An additive category $\cD$ is \emph{Krull--Schmidt} if $\cD(X,X)$ is a local ring for every indecomposable object $X$, and every object is a finite direct sum of indecomposable objects.  Then $\cD$ is idempotent complete and every object decomposes uniquely (up to isomorphism and reordering) as a finite direct sum of indecomposable objects.  If $\cD$ is an additive $\kk$-linear category, over some field $\kk$, with finite-dimensional endomorphism algebras, then $\Kar(\cD)$ is Krull-Schmidt.  This follows from \cite[Cor.~4.4]{Kra15} and the fact that finite-dimensional algebras are semiperfect.  For the remainder of this section $\kk$ is an arbitrary commutative ring, unless otherwise specified.  

\begin{assumption} \label{latte}
    We assume that $\cM = \Kar(\cM_0)$ and $\cC = \Kar(\cC_0)$ for some pair $(\cM_0,\cC_0)$ of additive $\kk$-linear categories satisfying \cref{jetlag}.  We also assume that $\cM$ and $\cC$ are Krull--Schmidt.  As noted above, this holds if $\kk$ is a field and $\cM$ and $\cC$ have finite-dimensional endomorphism spaces.
\end{assumption}

Throughout this section, we make \cref{latte}.  By \cref{soup},
\[
    \Psi \colon \Submod_\cC(\cM) \to \Submod_{P(\cM)}(\cM_\one)
\]
is an isomorphism of partially ordered sets.  

A \emph{thick Ob-submodule} of $\cM$ is an isomorphism-closed subset $I$ of $\Ob(\cM)$ such that
\begin{itemize}
    \item $X \oplus Y \in I$ if and only if $X,Y \in I$; 
    \item $X \in I$ implies $X \otimes Y \in I$ for all $Y \in \Ob(\cC)$.
\end{itemize}

For any Krull--Schmidt category $\cD$, we let $K_\oplus(\cD)$ denote its split Grothendieck group.  Since $\cC$ is monoidal, $K_\oplus(\cC)$ is a ring, with multiplication induced by the tensor product in $\cC$.  The right action of $\cC$ on $\cM$ induces a right action of $K_\oplus(\cC)$ on $K_\oplus(\cM)$.  We say that a $K_\oplus(\cC)$-submodule of $K_\oplus(\cM)$ is \emph{thick} if it is spanned by classes of indecomposable objects.  We let $\Thick(K_\oplus(\cM))$ denote the set of thick $K_\oplus(\cC)$-submodules of $K_\oplus(\cM)$.  The set $\Thick(K_\oplus(\cM))$ is partially ordered by inclusion.  We will often identify elements of $\Thick(K_\oplus(\cM))$ with the set of objects whose classes they contain.

For any $\cC$-submodule $\cN$ of $\cM$, we define the set
\[
    \Ob(\cN) = \{X \in \Ob(\cM) : 1_X \in \cN(X,X)\}
     = \{X \in \Ob(\cM) : \cN(X,X) = \cM(X,X)\}.
\]
Then $\Ob(\cN)$ is a thick Ob-submodule of $\cM$ and we have the corresponding morphism of partially ordered sets
\begin{equation} \label{Obmap}
    \Ob \colon \Submod_\cC(\cM) \to \Thick(K_\oplus(\cM)).
\end{equation}
We refer to $\Ob$ as the \emph{decategorification map}.

For $X \in \Ob(\cM)$, we have $X=(X_0,e_X)$ for some $X_0 \in \cM_0$ and $e_X \in \cM_0(X_0,X_0)$ an idempotent.  For $X,Y \in \Ob(\cM)$, we write $Y \inplus X$ if $Y$ is a direct summand of $X$, that is, $X \cong Y \oplus Z$ for some $Z \in \Ob(\cM)$.  To allow ourselves greater flexibility in what follows (see \ref{P1}), we fix, for each $X \in \Ob(\cM)$, an $X_1 \in \Ob(\cC)$ such that
\begin{equation} \label{fix}
    X \inplus X_1 \inplus X_0
    \qquad \text{and} \qquad (X \otimes Y)_1 = X_1 \otimes Y_1 \text{ for all } X \in \Ob(\cM),\ Y \in \Ob(\cC).
\end{equation}
(Such objects always exist, since we may take $X_1 = X_0$.) 
We have $X_1 = (X_0,f_X)$ for some idempotent $f_X \in \cC_0(X_0,X_0)$ satisfying $f_X e_X = e_X f_X = e_X$.  Thus, recalling \cref{mug}, we may view
\[
    \leftcup_X
    :=
    \begin{tikzpicture}[centerzero]
        \draw[<-] (-0.3,0.4) \toplabel{X_0} -- (-0.3,-0.1) arc(180:360:0.3) -- (0.3,0.4) \toplabel{X_0};
        \coupon{-0.3,0}{e_X};
    \end{tikzpicture}
    =
    \begin{tikzpicture}[centerzero]
        \draw[<-] (-0.3,0.6) \toplabel{X_0} -- (-0.3,-0.3) arc(180:360:0.3) -- (0.3,0.6) \toplabel{X_0};
        \coupon{-0.3,0.25}{e_X};
        \coupon{-0.3,-0.2}{f_X};
    \end{tikzpicture}
    =
    \begin{tikzpicture}[centerzero]
        \draw[<-] (-0.3,0.4) \toplabel{X_0} -- (-0.3,-0.1) arc(180:360:0.3) -- (0.3,0.4) \toplabel{X_0};
        \coupon{-0.3,0}{e_X};
        \coupon{0.3,0}{f_X^\vee};
    \end{tikzpicture}
\]
as an element of $\cM(\one,X \otimes X_1^\vee)$. (The dual $X_1^\vee$ exists since $X_1 \in \Ob(\cC)$.) 
 We will also view $e_X$ as an endomorphism of $X_1$ with image $X$. \Cref{soup} implies that
\begin{equation} \label{pumpkin}
    \Ob \circ \Psi^{-1} (N)
    =
    \left\{
        X \in \Ob(\cM) :
        \leftcup_X \in N(X \otimes X_1^\vee)
    \right\}.
\end{equation}
\details{
    Since $e_X = f_X e_X$, replacing the condition $f \in N(Y_0 \otimes X_0^\vee)$ by the condition $f \in N(Y_0 \otimes X_1^\vee)$ in \cref{corn} does not change the set there.
}
We have the following generalization of \cref{bei}.

\begin{lem} \label{nan}
    For all $X \in \cM$, $\cM(\one, X \otimes X_1^\vee)$ is generated by $\leftcup_X$ as an $\cM(X \otimes X_1^\vee, X \otimes X_1^\vee)$-module.
\end{lem}

\begin{proof}
    The proof is the same as for \cref{bei}, using \cref{gold}.
    \details{
        Suppose $f \in \cM(\one,X \otimes X_1^\vee)$.  Then
        \[
            \begin{tikzpicture}[anchorbase]
                \draw[->] (-0.2,-0.3) -- (-0.2,0.6) \toplabel{X_0};
                \draw[<-] (0.2,-0.15) -- (0.2,0.6) \toplabel{X_1};
                \genbox{-0.4,-0.5}{0.4,-0.15}{f};
                \coupon{-0.2,0.2}{e_X};
            \end{tikzpicture}
            =
            \begin{tikzpicture}[anchorbase]
                \draw[->] (0,0) -- (0,0.2) to[out=up,in=up,looseness=2] (-0.4,0.2) -- (-0.4,-0.8) to[out=down,in=down,looseness=2] (-0.8,-0.8) -- (-0.8,0.7) \toplabel{X_0};
                \draw[<-] (0.4,0.15) -- (0.4,0.7) \toplabel{X_1};
                \genbox{-0.2,-0.2}{0.6,0.15}{f};
                \indicate{-1.2,-0.3}{0.8,0.5};
                \coupon{-0.8,-0.6}{e_X};
                \coupon{-0.8,0}{e_X};
            \end{tikzpicture}
            \in \cM(X \otimes X_1^\vee,X \otimes X_1^\vee) \circ \leftcup_X,
        \]
        as desired.
    }
\end{proof}

For any subset $S \subseteq \Ob(\cM)$, we define the \emph{trace submodule}
\begin{equation} \label{tracesub}
    \begin{gathered}
        \Tr_S \cM_\one \in \Submod_{P(\cM)}(\cM_\one),
        \\
        \Tr_S \cM_\one(X) = \Span_\Z \{f \circ g : f \in \cM(Z,X),\ g \in \cM(\one,Z),\ Z \in S\}.
    \end{gathered}
\end{equation}
For $Z \in \Ob(\cM)$, we will write $\Tr_Z$ instead of $\Tr_{\{Z\}}$.

\begin{theo}[{cf.\ \cite[Th.~4.1.2]{Cou18}}] \label{BBQ}
    \begin{enumerate}
        \item \label{BBQ1} The map $\Ob$ is a surjective morphism of partially ordered sets
        \item \label{BBQ2} For $\tN \in \Thick(K_\oplus(\cM))$, the minimal element in the fibre $\Ob^{-1}(\tN)$ is the $\cC$-submodule with morphism spaces
            \[
                \cN_\tN^{\min}(X,Y) = \{f \in \cM(X,Y) : \exists\ Z \in \tN \text{ such that } f \text{ factors as } X \to Z \to Y\}.
            \]
        \item \label{BBQ3} For $\tN \in \Thick(K_\oplus(\cM))$, the minimal element in $\Psi(\Ob^{-1}(\tN)) \subseteq \cM_\one$ is the trace submodule $\Tr_\tN \cM_\one$.
    \end{enumerate}
\end{theo}

\begin{proof}
    \begin{enumerate}[wide]
        \item Suppose $\tN \in \Thick(K_\oplus(\cM))$.  It is clear that $\cN^{\min}_\tN$ is closed under composition with morphisms in $\cM$.  Since $\tN$ is a $K_\oplus(\cC)$-submodule of $K_\oplus(\cM)$, we see that $\cN^{\min}_\tN$ is also closed under the action of $1_Z$, $Z \in \Ob(\cC)$.  Thus it is a $\cC$-submodule.
        
            Now suppose that $X \in \Ob(\cN^{\min}_\tN)$.  Thus, $X \in \Ob(\cM)$ with $1_X \in \cN_\tN^{\min}(X,X)$.  Then there exist $Z \in \tN$, $f \in \cM(X,Z)$, and $g \in \cM(Z,X)$ such that $1_X = g \circ f$.  Since $\cM$ is idempotent complete, this implies that there exists $Y \in \Ob(\cM)$ such that $Z \cong X \oplus Y$.  Since $\tN$ is thick, we conclude that $X \in \tN$.  So, $\Ob(\cN^{\min}_\tN) \subseteq \tN$.  On the other hand, it follows by definition that $\tN \subseteq \Ob(\cN^{\min}_\tN)$.  Thus, $\Ob(\cN_\tN^{\min}) = \tN$.  Hence, $\Ob$ is surjective.  It is clear that $\Ob$ is a morphism of partially ordered sets.
            
        \item Suppose $\tN \in \Thick(K_\oplus(\cM))$.  Any $\cC$-submodule in $\Ob^{-1}(\tN)$ must contain $1_Z$ for all $Z \in \tN$.  Thus, by construction, $\cN^{\min}_\tN$ is minimal in the fibre over $\tN$.
            
        \item This follows immediately from the definition of $\Psi$.
        \qedhere
    \end{enumerate}
\end{proof}

Let $\indec(\cM)$ denote the set of isomorphism classes of indecomposable objects of $\cM$, and define
\[
    \bB_\cM := \{X \in \indec(\cM) : \cM(\one,X) \ne 0\}.
\]
Let $\add(X)$ denote the class of objects of $\cM$ that are direct sums of direct summands of $X$.  

\begin{lem}[{cf.\ \cite[Lem.~4.2.5]{Cou18}}] \label{bin}
    Suppose that the $\cM(X,X)$-module $\cM(\one,X)$ is simple for every $X \in \bB_\cM$.  Then every submodule of $\cM_\one$ is a trace submodule.
\end{lem}

\begin{proof}
    Suppose $N$ is a submodule of $\cM_\one$.  For $Y \in \bB_\cM$, since $\cM(\one,Y)$ is simple, either $N(Y)=0$ or $N(Y) = \cM(\one,Y)$.  Let $S$ be the subset of $\bB_\cM$ containing those $Y \in \bB_\cM$ with $N(Y) = \cM(\one,Y)$.  Then $N = \Tr_S \cM_\one$.
\end{proof}

\begin{lem}[{cf.\ \cite[Lem.~4.2.6]{Cou18}}] \label{green}
    Suppose $Z \in \bB_\cM$, $X \in \indec(\cM)$, and $\leftcup_X$ is a composite $\one \to Z^{\oplus k} \to X \otimes X_1^\vee$ for some $k>0$.  Then $X \inplus Z \otimes X_1$.
\end{lem}
 
\begin{proof}
    Suppose $\leftcup_X = f \circ g$ for some $g \colon \one \to Z^{\oplus k}$ and $f \colon Z^{\oplus k} \to X \otimes X_1^\vee$.  Then
    \[
        \begin{tikzpicture}[centerzero]
            \draw[->] (0,-0.6) \botlabel{X_1} to (0,0.6);
            \coupon{0,0}{e_X};
        \end{tikzpicture}
        \ =\
        \begin{tikzpicture}[centerzero={0,-0.25}]
            \draw[->] (0.4,-1.2) \botlabel{X_1} to (0.4,0.1) to[out=up,in=up,looseness=2] (0,0.1) -- (0,-0.1) to[out=down,in=down,looseness=2] (-0.4,-0.1) to (-0.4,0.6);
            \coupon{0.4,-0.7}{e_X};
            \coupon{-0.4,0}{e_X};
        \end{tikzpicture}
        \ =\
        \begin{tikzpicture}[anchorbase]
            \genbox{-0.4,-0.2}{0.4,0.2}{f};
            \genbox{-0.4,-1.2}{0.4,-0.8}{g};
            \draw[->] (-0.2,0.2) -- (-0.2,1) \toplabel{X_1};
            \draw (0,-0.8) -- (0,-0.2) node[midway,anchor=east] {\strandlabel{Z^{\oplus k}}};
            \draw[->] (0.6,-2) \botlabel{X_1} -- (0.6,0.2) to[out=up,in=up,looseness=2] (0.2,0.2);
            \coupon{0.6,-1.6}{e_X};
            \coupon{-0.2,0.55}{e_X};
        \end{tikzpicture}
        \ ,
    \]
    and so $e_X = 1_X$ is a composition $X \to Z^{\oplus k} \otimes X_1 \to X$.  Thus, $X \inplus Z^{\oplus k} \otimes X_1$.  Since $X$ is indecomposable, it follows that $X \inplus Z \otimes X_1$, as desired.  
\end{proof}

\begin{lem}[{cf.\ \cite[Lem.~4.2.7]{Cou18}}] \label{blue}
    If $Z \in \bB_\cM$ and $X \in \indec(\cM)$ satisfy
    \[
        \add \big( X \otimes X_1^\vee \big) \cap \bB_\cM = \{Z\},
    \]
    then the following hold:
    \begin{enumerate}
        \item \label{blue1} $\leftcup_X$ is a composition $\one \to Z^{\oplus k} \to X \otimes X_1^\vee$ for some $k > 0$.
        \item \label{blue2a} $\Tr_Z \cM_\one$ is generated, as submodule of $\cM_\one$, by $\leftcup_X$.
        \item \label{blue2} $\Psi^{-1}(\Tr_Z \cM_\one)$ is generated by $e_X$.
        \item \label{blue3} $\Ob \circ \Psi^{-1}(\Tr_Z \cM_\one)$ is generated by $X$ as a thick $K_\oplus(\cC)$-module.
        \item \label{blue4} $\Ob \circ \Psi^{-1}(\Tr_Z \cM_\one)$ is generated by $Z$ as a thick $K_\oplus(\cC)$-module.
    \end{enumerate}
\end{lem}

\begin{proof}
    \begin{enumerate}[wide]
        \item Since $Z$ is the only indecomposable summand of $X \otimes X_1^\vee$ admitting a nonzero morphism from $\one$, it follows that $\leftcup_X$ must factor through $Z^{\oplus k}$ for some $k>0$.
            
        \item By part~\cref{blue1}, 
        $\leftcup_X$ is contained in $\Tr_Z \cM_\one$. Now let $i \colon Z \to X \otimes X_1^\vee$ and $p \colon X \otimes X_1^\vee \to Z$ denote the inclusion and projection of one of the $Z$-summands of $X \otimes X_1^\vee$.  Then, if $f \in \cM(\one,Z)$, we have, using \cref{gold},
            \[
                \begin{tikzpicture}[anchorbase]
                    \draw[->] (0,0) -- (0,0.6) \toplabel{Z};
                    \coupon{0,0}{f};
                \end{tikzpicture}
                =
                \begin{tikzpicture}[centerzero={0,0.5}]
                    \draw[->] (-0.2,0.15) -- (-0.2,0.8);
                    \draw[<-] (0.2,0.15) -- (0.2,0.8);
                    \draw (0,-0.6) -- (0,-0.15);
                    \draw[->] (0,1.1) -- (0,1.55) \toplabel{Z};
                    \genbox{-0.35,-0.15}{0.35,0.15}{i};
                    \genbox{-0.35,0.8}{0.35,1.1}{p};
                    \coupon{0,-0.6}{f};
                    \draw (0.2,0.5) node[anchor=west] {\strandlabel{X_1}};
                    \coupon{-0.2,0.45}{e_X};
                \end{tikzpicture}
                =
                \begin{tikzpicture}[centerzero={0,0.5}]
                    \draw[->] (-0.2,0.15) to[out=up,in=up,looseness=2] (-0.5,0.15) -- (-0.5,-1.2) to[out=down,in=down,looseness=2] (-1,-1.2) -- (-1,0.2) \braidup (-0.2,0.8);
                    \draw[<-] (0.2,0.15) -- (0.2,0.8);
                    \draw (0,-0.6) -- (0,-0.15);
                    \draw[->] (0,1.1) -- (0,1.55) \toplabel{Z};
                    \genbox{-0.35,-0.15}{0.35,0.15}{i};
                    \genbox{-0.35,0.8}{0.35,1.1}{p};
                    \coupon{0,-0.6}{f};
                    \draw (0.2,0.5) node[anchor=west] {\strandlabel{X_1}};
                    \coupon{-1,-1.1}{e_X};
                \end{tikzpicture}
                \ .
            \]
            Therefore, $f$ is in the submodule of $\cM_\one$ generated by $\leftcup_X$.
            
        \item It follows from \cref{soup} and part~\cref{blue2a} that $\Psi^{-1}(\Tr_Z \cM_\one)$ is the $\cC$-submodule of $\cM$ generated by $\leftcup_X$.  By \cref{zigleft}, any $\cC$-submodule of $\cM$ containing $\leftcup_X$ also contains $e_X = 1_X$, and vice versa.  This proves part~\cref{blue2}.
            
        \item This follows immediately from part~\cref{blue2}.
            \details{
                Let $\tN$ be the thick $K_\oplus(\cC)$-module generated by $X$.  By part~\ref{blue2}, $1_X=e_X \in \Psi^{-1}(\Tr_Z \cM_\one)$, and so $X \in \Ob \circ \Psi^{-1}(\Tr_Z \cM_\one)$.  Since $\Ob \circ \Psi^{-1}(\Tr_Z \cM_\one)$ is a thick $K_\oplus(\cC)$-module, this implies that $\tN \subseteq \Ob \circ \Psi^{-1}(\Tr_Z \cM_\one)$.
                
                Now suppose that $W \in \Ob \circ \Psi^{-1}(\Tr_Z \cM_\one)$.  Then $1_W \in \Psi^{-1}(\Tr_Z \cM_\one)$, which is generated by $1_X = e_X$ by part~\cref{blue2}.  Therefore $1_W$ can be written as a composition
                \[
                    W \xrightarrow{f} X \otimes Y \xrightarrow{g} W
                \]
                for some $Y \in \Ob(\cC)$, $f \in \cM(W, X \otimes Y)$, $g \in \cM(X \otimes Y, W)$.  This implies that $W \inplus X \otimes Y$, and hence $W \in \tN$.  Therefore, we have  $\Ob \circ \Psi^{-1}(\Tr_Z \cM_\one) \subseteq \tN$.
            }
        
        \item By assumption, $Z \inplus X \otimes X_1^\vee$.  Thus, $Z$ is in the thick $K_\oplus(\cC)$-module generated by $X$.  On the other hand, by \cref{green}, $X$ is contained in the thick $K_\oplus(\cC)$-module generated by $Z$.  Thus, part~\cref{blue4} follows from part~\cref{blue3}.
    \end{enumerate}
\end{proof}

\begin{theo}[{cf.\ \cite[Th.~4.3.1]{Cou18}}] \label{coconut}
    Suppose that $\cM$ and $\cC$ satisfy \cref{latte} and that, for each $Z \in \bB_\cM$,
    \begin{enumerate}
        \item \label{coconut1} the $\cC(Z,Z)$-module $\cM(\one,Z)$ is simple;
        \item \label{coconut2} there exists $X_Z \in \indec(\cM)$ with $\add(X_Z \otimes (X_Z)_1^\vee) \cap \bB_\cM = \{Z\}$.
    \end{enumerate} 
    Then the map $\Ob$, defined in \cref{Obmap}, is an isomorphism of partially ordered sets.
\end{theo}

\begin{proof}
    By \cref{soup} and \cref{BBQ}\cref{BBQ3}, it suffices to prove that every submodule of $\cM_\one$ is of the form $\Tr_\tN \cM_\one$ for some $\tN \in \Thick(K_\oplus(\cM))$.  Let $N$ be a submodule of $\cM_\one$.  By \cref{bin}, we have $N = \Tr_S \cM_\one$ for some subset $S$ of $\bB_\cM$.  Let $\tN \in \Thick(K_\oplus(\cM))$ be the thick $K_\oplus(\cC)$-submodule of $K_\oplus(\cM)$ generated by $S$.  We will prove that $\Tr_S \cM_\one = \Tr_\tN \cM_\one$.
    
    Choose $Z \in \tN \cap \bB_\cM$ and let $X = X_Z$.  By the definition of $\tN$, there exist $Z' \in S$ and $Y' \in \Ob(\cC)$ such that $Z \inplus Z' \otimes Y'$.  Since $Y' \inplus Y'_1$, we may assume, without loss of generality, that $Y' = Y'_1$.  Let $X' = X_{Z'}$.  Thus, $Z' \inplus X' \otimes (X')_1^\vee$.  It follows from \cref{coconut2} and \cref{blue}\ref{blue1} that $\leftcup_X$ factors through $Z^{\oplus k}$ for some $k > 0$.  Thus, by \cref{green}, we have $X \inplus Z \otimes X_1$.  Therefore,
    \begin{gather*}
        X \inplus Z \otimes X_1
        \inplus Z' \otimes Y' \otimes X_1
        \inplus X' \otimes (X')_1^\vee \otimes Y' \otimes X_1
        = X' \otimes Y,
        \\
        \text{where } Y = (X')_1^\vee \otimes Y' \otimes X_1 \in \Ob(\cC).
    \end{gather*}
    Since $Z \inplus X \otimes X_1^\vee$ by \cref{coconut2}, we have
    \begin{equation} \label{kmc}
        Z \inplus X' \otimes Y \otimes (X' \otimes Y)_1^\vee
        \overset{\cref{fix}}{=} X' \otimes Y \otimes \big( (X')_1 \otimes Y \big)^\vee
        = X' \otimes Y \otimes Y^\vee \otimes (X')_1^\vee.
    \end{equation}
    
    \Cref{blue}\cref{blue1} implies that 
    \[
        \leftcup_{X'} \in \Tr_{Z'} \cM_\one(X' \otimes (X')_1^\vee) \subseteq \Tr_S \cM_\one.
    \]
    Since
    \[
        \leftcup_{X' \otimes Y}
        =
        \begin{tikzpicture}[centerzero]
            \draw[<-] (-0.5,0.5) \toplabel{(X')_1 \otimes Y} -- (-0.5,0) to[out=down,in=down,looseness=2] (0.5,0) -- (0.5,0.5);
            \coupon{-0.5,0}{e_{X'} \otimes 1_Y};
        \end{tikzpicture}
        \ =
        \begin{tikzpicture}[centerzero]
            \draw[<-] (-0.7,0.5) \toplabel{(X')_1} -- (-0.7,0) to[out=down,in=down,looseness=2] (0.7,0) -- (0.7,0.5);
            \draw[<-] (-0.2,0.5) -- (-0.2,0) to[out=down,in=down,looseness=2] (0.2,0) -- (0.2,0.5) \toplabel{Y};
            \coupon{-0.7,0}{e_{X'}};
        \end{tikzpicture}
        \ ,
    \]
    we have $\leftcup_{X' \otimes Y} \in \Tr_S \cM_\one (X' \otimes Y \otimes Y^\vee \otimes (X')_1^\vee)$.  Therefore, it follows from \cref{nan} that
    \[
        \Tr_S \cM_\one (X' \otimes Y \otimes Y^\vee \otimes (X')_1^\vee)
        = \cM(\one, X' \otimes Y \otimes Y^\vee \otimes (X')_1^\vee).
    \]
    Thus, by \cref{kmc}, we have $\Tr_S \cM_\one(Z) = \cM(\one,Z)$.
    \details{
        Any morphism $\one \to Z$ can be composed with the embedding of $Z$ into $X' \otimes Y \otimes Y^\vee \otimes (X')_1^\vee$ from \cref{kmc} to give a map $\one \to X' \otimes Y \otimes Y^\vee \otimes (X')_1^\vee$.  By the above centered equation, this map factors through $S$.  Then compose with the projection from $X' \otimes Y \otimes Y^\vee \otimes (X')_1^\vee$ to $Z$ to see that the original map $\one \to Z$ factors through $S$.
    }
    Since $Z \in \tN \cap \bB_\cM$ was arbitrary, this implies that $\Tr_\tN \cM_\one \subseteq \Tr_S \cM_\one$.  Because the reverse inclusion obviously holds, we have $\Tr_S \cM_\one = \Tr_\tN \cM_\one$, as desired.
\end{proof}

\begin{cor}[{cf.\ \cite[Th.~4.3.4]{Cou18}}] \label{dove}
    Suppose that $\cM$ and $\cC$ satisfy \cref{latte} and that $\kk$ is a field.  Assume that, for each $Z \in \bB_\cM$, there exists $X_Z \in \indec(\cM)$ such that
    \[
        Z \inplus X_Z \otimes (X_Z)_1^\vee
        \qquad \text{and} \qquad
        \dim_\kk \cM \big( (X_Z)_1, X_Z \big) = 1.
    \]
    Then the map $\Ob$, defined in \cref{Obmap}, is an isomorphism of partially ordered sets.
\end{cor}

\begin{proof}
    By assumption, for every $Z \in \bB_\cM$, we have
    \[
        0 < \dim_\kk \cM(\one,Z)
        \le \dim_\kk \cM \big( \one, X_Z \otimes (X_Z)_1^\vee \big)
        = \dim_\kk \cM \big( (X_Z)_1, X_Z \big)
        = 1.
    \]
    Thus, $\cM(\one,Z)$ is one-dimensional, which implies that condition~\cref{coconut1} of \cref{coconut} is satisfied.  Furthermore, the fact that $\dim_\kk \cM \left( \one, X_Z \otimes (X_Z)_1^\vee \right) = 1$ implies that $X_Z \otimes (X_Z)_1^\vee$ contains precisely one direct summand lying in $\bB_\cM$.  Thus, condition~\ref{coconut2} of \cref{coconut} is also satisfied.
\end{proof}

We conclude this section with a result that will be used in the proof of \cref{tree}.

\begin{lem} \label{fish}
    Suppose that $Z,Z' \in \bB_\cM$, that $\cM(\one,Z)$ is a simple $\cM(Z,Z)$-module, and that $\cM(\one,Z')$ is a simple $\cM(Z',Z')$-module.  Then $\Tr_Z \cM_\one = \Tr_{Z'} \cM_\one$ if and only if $Z \cong Z'$.
\end{lem}

\begin{proof}
    Suppose $\Tr_{Z'} \cM_\one = \Tr_Z \cM_\one$.  Choose a nonzero $f \in \cM(\one,Z)$.  Since $\cM(\one,Z)$ is simple as a $\cM(Z,Z)$-module, $f$ generates $\cM(\one,Z)$ as a $\cM(Z,Z)$-module, and hence $f$ also generates $\Tr_Z \cM_\one$ as a $P(\cM)$-module.  Because $f \in \Tr_{Z'} \cM_\one$, we have $f = g \circ h$ for some $r \in \N$, $h \in \cM(\one,(Z')^{\oplus r})$, and $g \in \cM((Z')^{\oplus r}, Z)$.  Then $h \in \Tr_{Z'} \cM_\one  = \Tr_Z \cM_\one$, which is generated by $f$.  So, $h = a \circ f$ for some $a \in \cM(Z,(Z')^{\oplus r})$.  Thus, we have $f = b \circ f$, where $b = g \circ a \in \cM(Z,Z)$ factors through $(Z')^{\oplus r}$.  Since $(1_Z - b) \circ f = 0$, the morphism $1_Z - b$ is not an isomorphism.  Thus, since $\cM(Z,Z)$ is local, $b$ is an isomorphism, which implies $Z \cong Z'$.
\end{proof}

\section{Pagodas\label{sec:pagodas}}

The goal in this section is to develop some sufficient conditions guaranteeing that the decategorification map is an isomorphism and allowing one to precisely describe the thick $K_\oplus(\cC)$-submodules of $K_\oplus(\cM)$.  The results here are a module-category analogue of \cite[\S4.4]{Cou18}.  Throughout this section $\kk$ is a field.

\begin{defin}[{\cite[Def.~4.4.1]{Cou18}}] \label{tier}
    A \emph{tiered monoidal category} is a pair $(\cS,\ell)$ where $\cS$ is a Krull--Schmidt monoidal category with $\End_\cS(\one) = \kk 1_\one \cong \kk$ and $\ell \colon \indec(\cS) \to \N$ is a set map such that
    \begin{description}[style=multiline, leftmargin=1cm, labelwidth=0.8cm]
        \item[\namedlabel{T1}{T1}] $\ell^{-1}(0) = \{\one\}$;
        \item[\namedlabel{T2}{T2}] for $\kappa,\lambda,\mu \in \indec(\cS)$, $\kappa \inplus \mu \otimes \lambda$ implies that $\ell(\kappa) = \ell(\mu) + \ell(\lambda)$.
    \end{description}
\end{defin}

If $(\cS,\ell)$ is a tiered monoidal category, then we define a binary relation $\preceq$ on $\indec(\cS)$ by setting
\begin{equation} \label{minion}
    \lambda \preceq \mu \iff \mu \inplus \lambda \otimes \nu \text{ for some } \nu \in \indec(\cS).
\end{equation}
Note that
\begin{equation} \label{bright}
    \lambda \preceq \mu \implies \lambda = \mu \text{ or } \ell(\lambda) < \ell(\mu).
\end{equation}
Note that $\preceq$ is a partial order.
\details{
    Reflexivity and antisymmetry are clear.  Suppose $\lambda \preceq \mu$ and $\mu \preceq \nu$.  Then there exist $\pi, \pi' \in \indec(\cS)$ such that
    \[
        \nu \inplus \mu \otimes \pi \inplus \lambda \otimes \pi' \otimes \pi.
    \]
    Since $\nu$ is indecomposable, it follows that there is an indecomposable summand $\pi''$ of $\pi' \otimes \pi$ such that $\nu \inplus \lambda \otimes \pi''$.  Thus, $\lambda \preceq \nu$.
}

\begin{defin}[{cf.\ \cite[\S4.4]{Cou18}}] \label{pagoda}
    A \emph{pagoda} is a tuple $(\cM,\cC,\cS,\ell,\phi,\bT)$, where
    \begin{itemize}
        \item $(\cM,\cC)$ satisfies \cref{latte},
        \item $\cM(\one,\one) = \kk 1_\one \cong \kk$,
        \item $(\cS,\ell)$ is a tiered monoidal category, 
        \item $\phi \colon \Lambda \to \Lambda$ is a set involution, with $\Lambda := \indec(\cS)$, and
        \item $\bT \colon \cS \to \cC$ is a monoidal functor,
    \end{itemize}
    such that the following conditions are satisfied:
    \begin{description}[style=multiline, leftmargin=1cm, labelwidth=0.8cm]
        \item[\namedlabel{P1}{P1}] We have a bijection $R \colon \Lambda \to \indec(\cM)$ such that, for all $\lambda \in \Lambda$, $\bT(\lambda) \cong R(\lambda) \oplus X$ (isomorphism in $\cM$) for some $X \in \Ob(\cM)$ such that $R(\mu) \inplus X \implies \mu \prec \lambda$.  In particular, this means that we can assume (in the notation of \cref{fix}) that
            \[
                R(\lambda)_1 = \bT(\lambda) \qquad \text{for all } \lambda \in \Lambda.
            \]
            
        \item[\namedlabel{P2}{P2}] We have $\bT(\lambda)^\vee \cong \bT(\phi(\lambda))$ (isomorphism in $\cM$) for all $\lambda \in \Lambda$.
                    
        \item[\namedlabel{P3}{P3}] If $\nu,\nu' \in \Lambda$, $\nu \ne \nu'$, and $R(\nu), R(\nu') \in \bB_\cM$, then $\ell(\nu) \ne \ell(\nu')$.  Let
            \[
                \bbL =
                \begin{cases}
                    \{0,1,2,\dotsc,n\} & \text{if } n+1 = |\bB_\cM| < \infty, \\
                    \N & \text{otherwise}.
                \end{cases}
            \]
            Then we define $\nu^{(j)}$, $j \in \bbL$, by setting $\bB_\cM = \{ R(\nu^{(j)}) : j \in \bbL \}$, where
            \[
                0 = \ell(\nu^{(0)}) < \ell(\nu^{(1)}) < \ell(\nu^{(2)}) < \dotsb.
            \]
            
        \item[\namedlabel{P4}{P4}] For each $j \in \bbL$, the set
            \[
                \Lambda_j := \{ \lambda \in \Lambda : \nu^{(j)} \inplus \lambda \otimes \phi(\lambda) \}
            \]
            is nonempty, and
            \[
                \dim_\kk \cM(\bT(\lambda),R(\lambda)) = 1
                \qquad \text{for all } \lambda \in \Lambda_j.
            \]
            
        \item[\namedlabel{P5}{P5}] If $j,j' \in \bbL$ with $j' < j$, then, for each $\lambda \in \Lambda_j$, there exists $\lambda' \in \Lambda_{j'}$ such that $\lambda' \prec \lambda$.
            
        \item[\namedlabel{P6}{P6}] If $\nu^{(j)} \preceq \kappa \inplus \lambda \otimes \phi(\lambda)$ for some $\lambda,\kappa \in \Lambda$ and $j \in \bbL$, then there exists $\mu \in \Lambda_j$ such that $\mu \preceq \lambda$.
    \end{description}
\end{defin}

\begin{lem}[{cf.\ \cite[Lem.~4.4.3]{Cou18}}] \label{bulb}
    Suppose $(\cM,\cC,\cS,\ell,\bT)$ is a pagoda and $\lambda, \mu, \kappa \in \Lambda$.
    \begin{enumerate}
        \item \label{bulb1} If $\ell(\kappa) = \ell(\lambda) + \ell(\mu)$, then the multiplicity of $R(\kappa)$ as a direct summand of $R(\lambda) \otimes \bT(\mu)$ is equal to the multiplicity of $\kappa$ as a direct summand of $\lambda \otimes \mu$.
            
        \item \label{bulb2} If $\lambda \preceq \mu$, and $R(\lambda) \in \tN \in \Thick(K_\oplus(\cM))$, then $R(\mu) \in \tN$.
    \end{enumerate}
\end{lem}

\begin{proof}
    \begin{enumerate}[wide]
        \item Let $l = \ell(\lambda) + \ell(\mu)$.  By \ref{T2}, we have
            \[
                \lambda \otimes \mu = \bigoplus_{\kappa : \ell(\kappa) = l} \kappa^{\oplus c^\kappa_{\lambda,\mu}}
                \qquad \text{for some } c^\kappa_{\lambda,\mu} \in \N.
            \]
            Since $\bT$ is a monoidal functor, this implies that
            \begin{equation} \label{chicken}
                \bT(\lambda) \otimes \bT(\mu)
                = \bigoplus_\kappa \bT(\kappa)^{\oplus c^\kappa_{\lambda,\mu}}.
            \end{equation}
            By \ref{P1}, we have
            \[
                \bT(\lambda) \otimes \bT(\mu)
                = \big( R(\lambda) \oplus X \big) \otimes \bT(\mu)
                \cong \big( R(\lambda) \otimes \bT(\mu) \big) \oplus \big( X \otimes \bT(\mu) \big)
            \]
            for some $X \in \Ob(\cM)$ such that $R(\nu) \inplus X \implies \nu \prec \lambda$.  By \ref{T2} and \ref{P1}, $X \otimes T(\mu)$ is a sum of objects $R(\nu)$ with $\ell(\nu) < l$.  Thus, for $\kappa \in \Lambda$ with $\ell(\kappa)=l$, the multiplicity of $R(\kappa)$ in $\bT(\lambda) \otimes \bT(\mu)$ is equal to the multiplicity of $R(\kappa)$ in $R(\lambda) \otimes \bT(\mu)$.  Combined with \ref{P1}, \cref{bright}, and \cref{chicken}, this implies that the multiplicity of $R(\kappa)$ in $R(\lambda) \otimes \bT(\mu)$ is equal to $c^\kappa_{\lambda,\mu}$, as desired.
            
        \item Suppose $\lambda \preceq \mu$ and $R(\lambda) \in \tN \in \Thick(K_\oplus(\cM))$.  Then $\mu \inplus \lambda \otimes \nu$ for some $\nu \in \Lambda$.  Thus, by part \cref{bulb1}, $R(\mu) \inplus R(\lambda) \otimes \bT(\nu)$, and so $R(\mu) \in \tN$, as desired.
        \qedhere
    \end{enumerate}
\end{proof}

\begin{theo}[{cf.\ \cite[Th.~4.4.4]{Cou18}}] \label{tree}
    Suppose $(\cM,\cC,\cS,\ell,\bT)$ is a pagoda.  For $j \in \bbL$, define
    \begin{equation}
        N_j := \Tr_{R(\nu^{(j)})} \cM_\one,\qquad
        \cN_j := \Psi^{-1}(N_j),\qquad
        \tN_j := \Ob(\cN_j).
    \end{equation}
    \begin{enumerate}
        \item \label{tree1} The decategorification map $\Ob \colon \Submod_\cC(\cM) \to \Thick(K_\oplus(\cM))$ is an isomorphism of partially ordered sets.
        
        \item \label{tree2} We have $\Thick(K_\oplus(\cM)) = \{\tN_i : i \in \bbL\}$ with $\Ob(\cM) = \tN_0 \supsetneq \tN_1 \supsetneq \tN_2 \supsetneq \dotsb$ and
            \begin{equation} \label{branch}
                R(\mu) \in \tN_j \iff \lambda \preceq \mu \text{ for some } \lambda \in \Lambda_j.
            \end{equation}
            
        \item \label{tree3} For $i,j \in \bbL$,
            \[
                \cN_i \left( \one, R(\nu^{(j)}) \right)
                =
                \begin{cases}
                    0 & \text{if } j < i, \\
                    \cM \left( \one, R(\nu^{(j)}) \right) \cong \kk & \text{if } j \ge i.
                \end{cases}
            \]
    \end{enumerate}
\end{theo}

\begin{proof}
    Let $j \in \bbL$ and choose $\lambda \in \Lambda_j$.  By the definition of $\Lambda_j$, we have $\nu^{(j)} \inplus \lambda \otimes \phi(\lambda)$.  Then, by \ref{T2}, \ref{P2}, and \cref{bulb}\cref{bulb1}, we have $R(\nu^{(j)}) \inplus R(\lambda) \otimes \bT(\lambda)^\vee$.  By \ref{P4} and the fact that $R(\lambda)_1 = \bT(\lambda)$, the hypotheses of \cref{dove} are satisfied.  Thus, part~\ref{tree1} follows.  As in the proof of \cref{dove}, we also have that $\dim_\kk \cM(\one, R(\nu^{(j)})) = 1$ and that $R(\nu^{(j)})$ is the only element of $\bB_\cM$ appearing as a summand in $R(\lambda) \otimes \bT(\lambda)^\vee$ for $\lambda \in \Lambda_j$.  Thus, by \cref{blue}\cref{blue3}, $\tN_j$ is generated by $R(\lambda)$ as a thick $K_\oplus(\cC)$-module, for all $\lambda \in \Lambda_j$.
    
    The trace submodules $N_j$, $j \in \bbL$, are pairwise distinct by \cref{fish}.  By \ref{P5} and \cref{bulb}\cref{bulb2}, this implies that, for $j' < j$, we have $\tN_j \subseteq \tN_{j'}$, hence also $N_j \subseteq N_{j'}$.
    \details{
        Suppose $j,j' \in \bbL$ satisfy $j'<j$.  Then, by \ref{P5}, there exists $\lambda' \in \Lambda_{j'}$ such that $\lambda' \prec \lambda$.  Since $R(\lambda') \in \tN_{j'}$, \cref{bulb}\cref{bulb2} implies that $R(\lambda) \in \tN_{j'}$.
    }
    By \cref{bin} (which can be applied since $\dim_\kk \cM(\one, R(\nu^{(j)})) = 1$) and the above discussion, we see that $\Submod_{P(\cM)}(\cM_\one) = \{N_j : j \in \bbL\}$.  Part~\cref{tree3} is now clear.
    
    It remains to prove \cref{branch}.  Since $R(\lambda) \in \tN_j$ for all $\lambda \in \Lambda_j$, \cref{bulb}\cref{bulb2} implies that the reverse implication $\impliedby$ in \cref{branch} holds.  To see the forward implication, suppose $R(\mu) \in \tN_j$.  By \cref{pumpkin}, we have $\leftcup_{R(\mu)} \in N_j$.  By part~\cref{tree3}, this implies that $R(\nu^{(j')}) \inplus R(\mu) \otimes \bT(\mu)^\vee$ for some $j' \ge j$.  By \ref{P1} and \ref{P2}, this implies that $R(\nu^{(j')}) \inplus \bT(\mu \otimes \phi(\mu))$.  Since $R(\nu^{(j')})$ is indecomposable, this implies that $R(\nu^{(j')}) \inplus \bT(\nu)$ for some $\nu \in \Lambda$ with $\nu \inplus \mu \otimes \phi(\mu)$.  By \ref{P1}, we have $\nu^{(j')} \preceq \nu$. Thus, by \ref{P6}, there exists $\lambda' \in \Lambda_{j'}$ with $\lambda' \preceq \mu$.  Then, by \ref{P5}, there exists $\lambda \in \Lambda_j$ such that $\lambda \preceq \lambda' \preceq \mu$.  This completes the proof.
\end{proof}

\section{Indecomposable objects in the disoriented skein category\label{sec:DSindec}}

Throughout this section we assume that $\kk$ is a field.  We fix elements $q,t \in \kk^\times$, $\delta \in \kk$ such that \cref{delta} is satisfied.  We also assume that the quantum characteristic is $\infty$, i.e., that either
\begin{itemize}
    \item $q$ is not a root of unity, or 
    \item $q=1$ and the characteristic of $\kk$ is zero.
\end{itemize}
Recall the categories $\OScat(q,t)$ and $\DScat(q,t)$ from \cref{DStwist}.  Our goal in this section is to classify the indecomposable objects in the additive Karoubi envelope $\Kar(\DScat(q,t))$, up to isomorphism.  We follow the approach of \cite[\S4]{CH17}.  For an idempotent $e$ in a category $\cC$, we let $\obj{e}$ denote the corresponding object in $\Kar(\cC)$.

It follows from the basis theorem \cite[Th.~7.1]{SSS25} that
\begin{itemize}
    \item $\End_{\DScat(q,t)}(\one) = \kk 1_\one \cong \kk$,
    \item $\Hom_{\DScat(q,t)}(X,Y)$ is a finitely-generated $\kk$-module for all $X,Y \in \Ob \DScat(q,t)$.
\end{itemize}
As noted at the start of \cref{sec:decat}, it follows that $\Kar(\DScat(q,t))$ is Krull--Schmidt.  Since the morphism spaces of $\DScat(q,t)$ are finite dimensional, to classify the indecomposable objects of $\DScat(q,t)$, it suffices to classify the primitive idempotents in $\DScat(q,t)$, up to conjugacy.
\details{
    Consider a $\kk$-linear additive category with finite-dimensional morphism spaces.  If $e \in \End(X)$ is not primitive, say $e=e_1+e_2$, then $(X,e)\cong (X,e_1)\oplus (X,e_2)$.  Thus, $(X,e)$ is not indecomposable. 

    Conversely, if $e \in A = \End(X)$ is primitive, then $\End(X,e) = eAe =\End_A(Ae)$. But $Ae$ is indecomposable as $A$-module (here we need finiteness) and thus $eAe$ is local.  Hence, it has no nontrivial idempotents. This proves that $(X,e)$ is indecomposable.  

    Now suppose that $(X,e)\cong (Y,f)$, with $e \in A=\End(X)$, $f \in \End(Y)$ primitive idempotents.  Since $(X,e) \cong (X \oplus Y, (e,0))$ and $(Y,f) \cong (X \oplus Y, (0,f))$, we may assume, without loss of generality, that $X=Y$.  Then we have
    \[
        g \in fAe = \Hom_A(Af,Ae)
        \qquad \text{and} \qquad
        h \in eAf = \Hom_A(Ae,Af)
    \]
    such that $gh = f$ and $gf=e$.  Thus, $Ae \cong Af$ as $A$-modules.  Since $A$ is finite dimensional over $\kk$, this implies that $e$ and $f$ are conjugate.  (See \cite[Prop.~1.7.2]{Ben98}.)
    
    Conversely, if $e = \End(X)$ and $f \in \End(Y)$ are conjugate, then there exists an invertible $h \colon Y \to X$ such that $e = h f h^{-1}$.  Then
    \[
        eh = hf = ehf \colon (Y,f) \to (X,e)
        \quad \text{and} \quad
        fh^{-1} = h^{-1} e = f h^{-1} e \colon (X,e) \to (Y,f)
    \]
    are mutually inverse isomorphisms, since
    \[
        (ehf)(fh^{-1}e) = ehfh^{-1}e = e h h^{-1} e = e,\quad
        (fh^{-1}e)(ehf) = fh^{-1}ehf = f h^{-1} h f = f.
    \]
}

Since every object of $\DScat(q,t)$ is isomorphic to $\upobj^r$ for some $r \in \N$, it is enough to consider the primitive idempotents in
\begin{equation} \label{DSalg}
    \DSalg_r = \DSalg_r(q,t) := \End_{\DScat(q,t)}(\upobj^r),\qquad r \in \N.
\end{equation}
Our method will be an induction on $r$ based on \cref{raft} below, which seems to be well known---for example, it is stated in \cite[Lem.~3.3]{CO11} under slightly stronger hypotheses.  We have included a proof since we could not find one in the literature.  Recall that a ring is \emph{semiperfect} if and only if its identity element $1$ can be written as a sum of mutually orthogonal local idempotents; see \cite[Th.~23.6]{Lam01} for this characterization and \cite[Ch.~8]{Lam01} for further properties of semiperfect rings.  Every artinian ring is semiperfect.  In particular, finite-dimensional algebras over fields are semiperfect.

\begin{prop} \label{raft}
    Suppose that $A$ is a semiperfect ring and that $\xi \in A$ is an idempotent.  There is a bijection of sets
    \[
        \varphi \colon
        \left\{ \parbox{4cm}{conjugacy classes of primitive idempotents in $A$} \right\}
        \xrightarrow{\cong}
        \left\{ \parbox{4cm}{conjugacy classes of primitive idempotents in $\xi A \xi$} \right\}
        \sqcup
        \left\{ \parbox{4cm}{conjugacy classes of primitive idempotents in $A/(A \xi A)$} \right\}
    \]
    given by
    \[
        \varphi(X)
        =
        \begin{cases}
            X \cap (\xi A \xi) & \text{if } X \subseteq A \xi A, \\
            \{ e + A \xi A : e \in X\} & \text{if } X \nsubseteq A \xi A.
        \end{cases}
    \]
\end{prop}

Before giving the proof, we note that, if $X$ is the conjugacy class of a primitive idempotent $e$ in a ring $A$, then
\[
    X \nsubseteq A \xi A \iff X \cap (A \xi A) = \varnothing.
\]
\details{
    We have
    \[
        \big( e \in A \xi A \iff X \subseteq A \xi A \big)
        \quad \text{and} \quad
        \big( e \notin A \xi A \iff X \cap (A \xi A) = \varnothing \big).
    \]
}

\begin{proof}    
    Let
    \begin{itemize}
        \item $I$ be the set of conjugacy classes of primitive idempotents in $A$, 
        \item $J$ be the set of conjugacy classes of primitive idempotents in $\xi A \xi$, and
        \item $K$ be the set of conjugacy classes of primitive idempotents in $A/A\xi A$.
    \end{itemize}
    For a ring $R$, let $\PI(R)$ denote the set of primitive idempotents in $R$.  For $e,f \in \PI(R)$, we write $e \sim_R f$ if $e$ and $f$ are conjugate in $R$.
    
    Since $A$ is semiperfect, we may choose decompositions
    \begin{equation} \label{nail}
        \xi = e_1 + \dotsb + e_m
        \qquad \text{and} \qquad
        1-\xi = f_1 \dotsb + f_n
    \end{equation}
    into sums of mutually orthogonal primitive idempotents.  (This follows, for example, by taking the projective module $R \xi$ in the left-module analogue of \cite[Cor.~24.14(1)]{Lam01}.)  Then
    \[
        1 = e_1 + \dotsb + e_m + f_1 + \dotsb + f_n
    \]
    is a decomposition into pairwise orthogonal primitive idempotents.
    \details{
        For $1 \le i \le m$ and $1 \le j \le n$, we have
        \[
            e_i f_j = (e_i \xi) \big( (1-\xi) f_j \big)
            = 0.
        \]
        Similarly $f_j e_i = 0$.
    }
    Every primitive idempotent in $A$ is conjugate to one of the $e_1,\dotsc,e_m,f_1,\dotsc,f_n$.  Rearranging the $f_i$ if necessary, we may assume that $f_1,\dotsc,f_l$ are each conjugate to some $e_j$, and that none of $f_{l+1},\dotsc,f_n$ are conjugate to any $e_j$.  (We allow $l=0$ or $l=n$.)  Thus, we have
    \[
        I = I_1 \sqcup I_2,\qquad
        I_1 = \{e_1,\dotsc,e_m\}/\sim_A,\qquad
        I_2 = \{f_{l+1},\dotsc,f_n\}/\sim_A.
    \]
    
    Since $\xi$ is the identity element of the corner ring $\xi A \xi$ (which is also semiperfect, by \cref{nail}), and the $e_i$ remain primitive in this ring, it follows from the decomposition of $\xi$ in \cref{nail} that  every primitive idempotent of $\xi A \xi$ is conjugate to some $e_i$.  Thus, we have 
    \begin{multline*}
        \{ e_1, e_2, \dotsc, e_m \} \subseteq \PI(\xi A \xi)
        = \{ e \in \PI(\xi A \xi) : e \sim_{\xi A \xi} e_i \text{ for some } 1 \le i \le m\}
        \\
        \subseteq \{ e \in \PI(A) : e \sim_A e_i \text{ for some } 1 \le i \le m \}.
    \end{multline*}
    Clearly,
    \[
        e_i \sim_{\xi A \xi} e_j \implies e_i \sim_A e_j.
    \]
    Conversely,
    \begin{align*}
        e_i \sim_A e_j
        &\implies A e_i \cong A e_j \text{ as $A$-modules} \\
        &\implies \xi Ae_i \cong \xi A e_j \text{ as $\xi A \xi$-modules} \\
        &\implies e_i \sim_{\xi A \xi} e_j.
    \end{align*} 
    \details{
        Let $\alpha \colon A e_i \to A e_j$ be an isomorphism of $A$-modules.  Then, $\alpha( \xi A e_i ) = \xi \alpha(A e_i) = \xi A e_j$, and $\alpha$ intertwines the action of $\xi A \xi$ since it intertwines the action of $A$.
    }
    Thus, we have a bijection
    \[
        I_1 \xrightarrow{\cong} J,\qquad X \mapsto X \cap (\xi A \xi).
    \]
    
    Let $B = A/ A \xi A$.  Since $e_i = e_i \xi \in A \xi A$ for all $1 \le i \le m$, we have a decomposition
    \begin{equation} \label{cheese}
        B \cong \bigoplus_{j={l+1}}^n A f_j/ (A \xi A) f_j
    \end{equation}
    of $B$ as a sum indecomposable projective $B$-modules.  We claim that each summand in \cref{cheese} is nonzero.  Suppose, towards a contradiction, that $A f_j / (A \xi A) f_j = 0$ for some $l < j \le n$.  Then $f_j \in A \xi A f_j \subseteq A \xi A$.  Thus, by \cref{nail}, we can write
    \[
        f_j = \sum_{r=1}^N a_r e_{i_r} b_r
        \quad \text{for some } N \in \N,\ a_r, b_r \in A,\ 1 \le i_r \le m,\ 1 \le r \le N.
    \]
    The map
    \[
        \bigoplus_{r=1}^N A e_{i_r} \to A f_j,\qquad
        (x_r)_r \mapsto \sum_{r=1}^N x_r e_{i_r} b_r f_j,
    \]
    is a surjective homomorphism of $A$-modules since, for all $c \in A$, we have
    \[
        (c a_r)_{r \in \N}
        \mapsto \sum_{r=1}^N c a_r e_{i_r} b_r f_j
        = c f_j^2
        = c f_j.
    \]
    Because $A f_j$ is projective, this homomorphism splits, and so $A f_j$ is a direct summand of $\bigoplus_r A e_{i_r}$.  Since the $e_{i_r}$ and $f_j$ are primitive, this implies that $A f_j \cong A e_i$ as $A$-modules, for some $i$.  Hence $f_j \sim_A e_i$, contradicting our choice of $l$ and the fact that $j > l$.  This proves the claim.
    
    It follows from \cref{cheese} that every primitive idempotent of $B$ is conjugate to one of the $f_j + (A \xi A)$, $l < j \le n$.  We claim that, for $l < i,j \le n$,
    \[
        f_i \sim_A f_j \iff (f_i + A \xi A) \sim_B (f_j + A \xi A).
    \]
    The forward implication is clear.  For the reverse implication, suppose $(f_i + A \xi A) \sim_B (f_j + A \xi A)$.  Then the indecomposable projective left $B$-modules
    \begin{equation} \label{mac}
        B(f_i + A \xi A) \cong A f_i / A \xi A f_i
        \qquad \text{and} \qquad
        B(f_j + A \xi A) \cong A f_j / A \xi A f_j
    \end{equation}
    are isomorphic.  Let $S$ be their simple top.  Viewing $S$ as an $A$-module via the quotient map $A \to B$, the natural maps $A f_i \to S$, $A f_j \to S$, are projective covers of $S$ as $A$-modules.  (By \cite[Prop.~24.12]{Lam01}, every finitely-generated module over a semiperfect ring has a projective cover.)  By uniqueness of projective covers, $A f_i \cong A f_j$ as $A$-modules, which implies that $f_i \sim_A f_j$.  This proves the claim.
   
    It follows from the above discussion that we have a bijection
    \[
        I_2 \xrightarrow{\cong} K,\qquad
        X \mapsto \{e + A \xi A : e \in X\},
    \]
    completing the proof of the proposition.
\end{proof}

For $r,i \in \N$, let
\begin{gather}
    \eta_{r,i} :=
    \begin{tikzpicture}[centerzero={0,-0.55}]
        \draw[multi,->] (0.2,-1) \botlabel{r-1} -- (0.2,-0.1);
        \draw[<-] (-0.2,-0.1) to[out=down,in=up] (-1.7,-0.8) -- (-1.7,-1);
        \draw[->] (-0.2,-1) -- (-0.2,-0.7) to[out=up,in=up,looseness=1.5] (-0.5,-0.7) -- (-0.5,-0.8);
        \draw[<-] (-0.5,-0.8) -- (-0.5,-1);
        \draw[->] (-1.1,-1) -- (-1.1,-0.7) to[out=up,in=up,looseness=1.5] (-1.4,-0.7) -- (-1.4,-0.8);
        \draw[<-] (-1.4,-0.8) -- (-1.4,-1);
        \node at (-0.8,-0.75) {$\cdots$};
        \opendot{-0.5,-0.8};
        \opendot{-1.4,-0.8};
    \end{tikzpicture}
    \in \Hom_\DScat(\upobj^{r+2i}, \upobj^r)
    ,\qquad
    \mu_{r,i} :=
    \begin{tikzpicture}[centerzero={0,0.35}]
        \draw[multi,->] (-0.1,0) -- (-0.1,0.7) \toplabel{r};
        \draw[<->] (-0.5,0.7) -- (-0.5,0.4) to[out=down,in=down,looseness=1.5] (-0.8,0.4) -- (-0.8,0.7);
        \draw[<->] (-1.4,0.7) -- (-1.4,0.4) to[out=down,in=down,looseness=1.5] (-1.7,0.4) -- (-1.7,0.7);
        \opendot{-0.5,0.5};
        \opendot{-1.4,0.5};
        \node at (-1.05,0.45) {$\cdots$};
    \end{tikzpicture}
    \in \Hom_\DScat(\upobj^r, \upobj^{r+2i}),
    \\ \label{xi}
    \xi_{r,i} := \mu_{r,i} \eta_{r,i}
    =
    \begin{tikzpicture}[centerzero={0,-0.25}]
        \draw[multi,->] (0.2,-1) \botlabel{r-1} -- (0.2,0.5);
        \draw[<->] (-0.5,0.5) -- (-0.5,0.2) to[out=down,in=down,looseness=1.5] (-0.8,0.2) -- (-0.8,0.5);
        \draw[<->] (-1.4,0.5) -- (-1.4,0.2) to[out=down,in=down,looseness=1.5] (-1.7,0.2) -- (-1.7,0.5);
        \opendot{-0.5,0.3};
        \opendot{-1.4,0.3};
        \node at (-1.06,0.25) {$\cdots$};
        \draw[<-] (-0.2,0.5) -- (-0.2,0.3) to[out=down,in=up] (-1.7,-0.9) -- (-1.7,-1);
        \draw[->] (-0.2,-1) -- (-0.2,-0.7) to[out=up,in=up,looseness=1.5] (-0.5,-0.7) -- (-0.5,-0.8);
        \draw[<-] (-0.5,-0.8) -- (-0.5,-1);
        \draw[->] (-1.1,-1) -- (-1.1,-0.7) to[out=up,in=up,looseness=1.5] (-1.4,-0.7) -- (-1.4,-0.8);
        \draw[<-] (-1.4,-0.8) -- (-1.4,-1);
        \node at (-0.8,-0.75) {$\cdots$};
        \opendot{-0.5,-0.8};
        \opendot{-1.4,-0.8};
    \end{tikzpicture}
    \in \DSalg_{r+2i}.
\end{gather}
It is straightforward to verify that
\begin{equation} \label{squiggle}
    \eta_{r,i} \mu_{r,i} = 1_r := 1_{\upobj^{\otimes r}} \qquad \text{for all } r,i \in \N.
\end{equation}
It follows that $\xi_{r,i}$ is idempotent.

\begin{lem} \label{jack}
    For $r,i \in \N$, the map
    \[
        \gamma_{r,i} \colon \DSalg_r \mapsto \DSalg_{r+2i},\qquad
        f \mapsto \mu_{r,i} f \eta_{r,i}
        =
        \begin{tikzpicture}[anchorbase]
            \draw[multi,->] (0,0) -- (0,0.7) \toplabel{r};
            \draw[multi,->] (0.2,-1) \botlabel{r-1} -- (0.2,-0.1);
            \genbox{-0.4,-0.1}{0.4,0.3}{f};
            \draw[<->] (-0.5,0.7) -- (-0.5,0.4) to[out=down,in=down,looseness=1.5] (-0.8,0.4) -- (-0.8,0.7);
            \draw[<->] (-1.4,0.7) -- (-1.4,0.4) to[out=down,in=down,looseness=1.5] (-1.7,0.4) -- (-1.7,0.7);
            \opendot{-0.5,0.5};
            \opendot{-1.4,0.5};
            \node at (-1.06,0.45) {$\cdots$};
            \draw[<-] (-0.2,-0.1) to[out=down,in=up] (-1.7,-0.8) -- (-1.7,-1);
            \draw[->] (-0.2,-1) -- (-0.2,-0.7) to[out=up,in=up,looseness=1.5] (-0.5,-0.7) -- (-0.5,-0.8);
            \draw[<-] (-0.5,-0.8) -- (-0.5,-1);
            \draw[->] (-1.1,-1) -- (-1.1,-0.7) to[out=up,in=up,looseness=1.5] (-1.4,-0.7) -- (-1.4,-0.8);
            \draw[<-] (-1.4,-0.8) -- (-1.4,-1);
            \node at (-0.8,-0.75) {$\cdots$};
            \opendot{-0.5,-0.8};
            \opendot{-1.4,-0.8};
        \end{tikzpicture}
        ,
    \]
    is an injective homomorphism of nonunital algebras.  Furthermore, $\gamma_{r,i}(\DSalg_r) = \xi_{r,i} \DSalg_{r+2i} \xi_{r,i}$.
\end{lem}

\begin{proof}
    The map $\gamma_{r,i}$ is clearly $\kk$-linear.  For $f,g \in \DSalg_r$, we have
    \[
        \gamma_{r,i}(fg)
        = \mu_{r,i}(fg)\eta_{r,i}
        \overset{\cref{squiggle}}{=} \mu_{r,i} f \eta_{r,i} \mu_{r,i} g \eta_{r,i}
        = \gamma_{r,i}(f) \gamma_{r,i}(g).
    \]
    Hence, $\gamma_{r,i}$ is a homomorphism of nonunital algebras.  Since
    \[
        \eta_{r,i} \gamma_{r,i}(f) \mu_{r,i}
        \overset{\cref{squiggle}}{=} f
        \qquad \text{for all } f \in \DSalg_r,
    \]  
    the homomorphism $\gamma_{r,i}$ is also injective.  We also have
    \[
        \xi_{r,i} \DSalg_{r+2i} \xi_{r,i}
        \overset{\cref{xi}}{=} \mu_{r,i} \eta_{r,i} \DSalg_{r+2i} \mu_{r,i} \eta_{r,i}
        \subseteq \gamma_{r,i}(\DSalg_r)
    \]
    and
    \[
        \gamma_{r,i}(\DSalg_r)
        = \mu_{r,i} \DSalg_r \eta_{r,i}
        \overset{\cref{xi}}{\underset{\cref{squiggle}}{=}} \xi_{r,i} \mu_{r,i} \DSalg_r \eta_{r,i} \xi_{r,i}
        \subseteq \xi_{r,i} \DSalg_{r+2i} \xi_{r,i},
    \]
    proving the final statement in the lemma.
\end{proof}

Bases for the morphism spaces of $\DScat$ are described in \cite[Th.~7.1]{SSS25}.  In particular, $\DSalg_r$ has basis given by a set $M_\DScat(\upobj^{\otimes r},\upobj^{\otimes r})$ of \emph{reduced diagrams}.  We refer the reader to \cite[\S7.1]{SSS25} for the precise definition.

\begin{lem} \label{coffee}
    Suppose $r \in \N$.  The following subspaces of $\DSalg_r$ are equal:
    \begin{enumerate}
        \item \label{coffee1} the two-sided ideal of $\DSalg_r$ generated by morphisms that factor through $\upobj^{\otimes (r-2)}$ (interpreted as the zero ideal if $r \le 1$),
        \item \label{coffee2} the two-sided ideal of $\DSalg_r$ generated by morphisms that factor through $\upobj^{\otimes s}$ for some $s < r$ (interpreted as the zero ideal if $r=0$),
        \item \label{coffee3} the span of the elements of $M_\DScat(\upobj^{\otimes r},\upobj^{\otimes r})$ with at least one cup or cap.
    \end{enumerate}
    If $r \ge 2$, then these spaces are also equal to the two-sided ideal of $\DSalg_r$ generated by the element
    \begin{equation} \label{milk}
        \begin{tikzpicture}[centerzero]
            \draw[multi,->] (0.9,-0.5) \botlabel{r-2} -- (0.9,0.5);
            \draw[<->] (0.6,0.5) -- (0.6,0.2) to[out=down,in=down,looseness=1.5] (0.3,0.2) -- (0.3,0.5);
            \draw[->] (0.3,-0.5) -- (0.3,-0.3);
            \draw[<-] (0.3,-0.3) -- (0.3,-0.2) to[out=up,in=up,looseness=1.5] (0.6,-0.2) -- (0.6,-0.5);
            \opendot{0.3,0.3};
            \opendot{0.3,-0.3};
        \end{tikzpicture}
        \in \DSalg_r.
    \end{equation}
    If $r \ge 3$, then these spaces are also equal to the two-sided ideal of $\DSalg_r$ generated by the idempotent
    \[
        \xi_{r-2,1}
        =
        \begin{tikzpicture}[centerzero]
            \draw[multi,->] (0,-0.5) \botlabel{r-3} -- (0,0.5);
            \draw[->] (-0.3,-0.5) -- (-0.3,-0.2) to[out=up,in=up,looseness=1.5] (-0.6,-0.2) -- (-0.6,-0.3);
            \draw[<-] (-0.6,-0.3) -- (-0.6,-0.5);
            \opendot{-0.6,-0.3};
            \draw[->] (-0.9,-0.5) to[out=up,in=down] (-0.3,0.5);
            \draw[<->] (-0.6,0.5) -- (-0.6,0.2) to[out=down,in=down,looseness=1.5] (-0.9,0.2) -- (-0.9,0.5);
            \opendot{-0.6,0.3};
        \end{tikzpicture}
        \in \DSalg_r.
    \]
\end{lem}

\begin{proof}
    The proof that the spaces described in \ref{coffee1}, \ref{coffee2}, and \cref{coffee3} are equal is straightforward.  Let $I$ denote this ideal.  Let $I'$ be the two-sided ideal generated by the element \cref{milk}, and let $I''$ be the two-sided ideal generated by $\xi_{r-2,1}$.  Clearly $I',I'' \subseteq I$.
    
    Let $f$ be an element of $M_\DScat(\upobj^{\otimes r}, \upobj^{\otimes r})$ with at least one cup or cap.  Then $f$ has at least one cup \emph{and} at least one cap.  Composing $f$ on the top and bottom with diagrams involving only toggles and crossings, we may move one of the cups and one of the caps to the left edge of the diagram, resulting in an element of $I'$.  For example, if
    \[
        f =
        \begin{tikzpicture}[centerzero]
            \draw[->] (0.3,-0.5) \braidup (0,0.5);
            \draw[over,->] (0,-0.5) \braidup (0.3,0.5);
            \draw[->] (0.6,-0.5) \braidup (0.9,0.5);
            \draw[->] (0.9,-0.5) \braidup (1.8,0.5);
            \draw[->] (1.2,-0.5) -- (1.2,-0.3);
            \draw[<-] (1.2,-0.3) to[out=up,in=up] (1.8,-0.3) -- (1.8,-0.5);
            \opendot{1.2,-0.3};
            \draw[over,<->] (0.6,0.5) -- (0.6,0.3) to[out=down,in=down] (1.5,0.3) -- (1.5,0.5);
            \opendot{0.6,0.3};
            \draw[over,->] (1.5,-0.5) \braidup (1.2,0.5);
        \end{tikzpicture}
    \]
    then
    \[
        \begin{tikzpicture}[centerzero]
            \draw[->] (0,0.2) -- (0,0.4) \braidup (0.6,1.1) -- (0.6,1.3);
            \draw[->] (0.3,0.2) -- (0.3,0.4) \braidup (0.9,1.1) -- (0.9,1.3);
            \draw[->] (0.6,0.2) -- (0.6,0.4);
            \draw[over,<->] (0.6,0.4) \braidup (0,1.1) -- (0,1.3);
            \opendot{0.6,0.4};
            \opendot{0,1.1};
            \draw[->] (0.9,0.2) -- (0.9,0.4) \braidup (1.2,1.1) -- (1.2,1.3);
            \draw[over,->] (1.5,0.2) -- (1.5,0.4) \braidup (0.3,1.3);
            \draw[over,->] (1.2,0.2) -- (1.2,0.4) \braidup (1.5,1.1) -- (1.5,1.3);
            \draw[->] (1.8,0.2) -- (1.8,1.3);
            \draw (0.6,-1.5) -- (0.6,-1.3) \braidup (0,-0.4) -- (0,-0.2);
            \draw (0.9,-1.5) -- (0.9,-1.3) \braidup (0.3,-0.4) -- (0.3,-0.2);
            \draw (1.2,-1.5) -- (1.2,-1.3) \braidup (0.6,-0.4) -- (0.6,-0.2);
            \draw (1.5,-1.5) -- (1.5,-1.3) \braidup (0.9,-0.4) -- (0.9,-0.2);
            \draw[->] (0,-1.5) -- (0,-1.3);
            \draw[over,<->] (0,-1.3) \braidup (1.2,-0.4) -- (1.2,-0.2);
            \opendot{0,-1.3};
            \opendot{1.2,-0.4};
            \draw[over] (0.3,-1.5) \braidup (1.8,-0.4) -- (1.8,-0.2);
            \draw[over] (1.8,-1.5) -- (1.8,-1.3) \braidup (1.5,-0.4) -- (1.5,-0.2);
            \genbox{-0.2,-0.2}{2,0.2}{f};
        \end{tikzpicture}
        =
        \begin{tikzpicture}[centerzero]
            \draw[->] (0,-0.5) -- (0,-0.3);
            \draw[<-] (0,-0.3) to[out=up,in=up,looseness=2] (0.3,-0.3) -- (0.3,-0.5);
            \opendot{0,-0.3};
            \draw[<->] (0,0.5) -- (0,0.3) to[out=down,in=down,looseness=2] (0.3,0.3) -- (0.3,0.5);
            \opendot{0,0.3};
            \draw[->] (0.9,-0.5) \braidup (0.6,0.5);
            \draw[over,->] (0.6,-0.5) \braidup (0.9,0.5);
            \draw[->] (1.2,-0.5) -- (1.2,0.5);
            \draw[->] (1.5,-0.5) \braidup (1.8,0.5);
            \draw[over,->] (1.8,-0.5) \braidup (1.5,0.5);
        \end{tikzpicture}
        \in I'.
    \]
    Since the diagrams involving only toggles and crossings are invertible, this implies that $f \in I'$.  Thus, $I \subseteq I'$.  The proof that $I \subseteq I''$ is analogous.
\end{proof}

For $r \in \N$, let $J_r$ denote the two-sided ideal of $\DSalg_r$ defined by any of the equivalent descriptions given in \cref{coffee}.  In particular,
\[
    J_r = \DSalg_r \xi_{r-2,1} \DSalg_r
    \qquad \text{if } r \ge 3.
\]
Let $\Hecke_r = \Hecke_r(q)$ be the Iwahori--Hecke algebra of type $A_{r-1}$.  It follows from the basis theorem \cite[Th.~7.1]{SSS25} that the natural algebra homomorphism
\[
    \Hecke_r \to \DSalg_r
\]
is injective.  We use this homomorphism to view $\Hecke_r$ as a subalgebra of $\DSalg_r$.  This subalgebra is spanned by the elements of $M_\DScat(\upobj^{\otimes r})$ with no cups, caps, or toggles.  This discussion implies the following result.

\begin{lem} \label{claptrap}
    For all $r \in \N$, we have an isomorphism of algebras
    \[
        \Hecke_r \xrightarrow{\cong} \DSalg_r/J_r,\qquad
        f \mapsto f + J_r.
    \]
\end{lem}

Let
\[
    \pi_r \colon \DSalg_r \twoheadrightarrow \DSalg_r/J_r \xrightarrow{\cong} \Hecke_r
\]
be the surjective algebra homomorphism arising from the isomorphism described in \cref{claptrap}.  The primitive idempotents of $\Hecke_r$ are parameterized, up to conjugation, by partitions of $r$.  For $\lambda \vdash r$, let $\te_\lambda$ denote a primitive idempotent of $\Hecke_r$ in the conjugacy class corresponding to $\lambda$.  The idempotent $\te_\lambda$ may no longer be primitive as an element of $\DSalg_r$. 

Let $\te_\lambda = \te_{\lambda,1} + \te_{\lambda,2} + \dotsb + \te_{\lambda,k}$ be a decomposition of $\te_\lambda$ into mutually orthogonal primitive idempotents in $\DSalg_r$.  Then
\[
    \te_\lambda
    = \pi_r(\te_\lambda)
    = \pi_r(\te_{\lambda,1}) + \dotsb + \pi_r(\te_{\lambda,k})
\]
is a decomposition of $\te_\lambda$ as a sum of mutually orthogonal idempotents in $\Hecke_r$.  Since $\te_\lambda$ is primitive in $\Hecke_r$, it follows that there is a unique $i$, $1 \le i \le k$, such that $\pi_r(\te_{\lambda,i}) \ne 0$.  We define
\begin{equation} \label{PVD}
    e_\lambda := \te_{\lambda,i}.
\end{equation}
This definition of $e_\lambda$ depends on the decomposition of $\te_\lambda$ into mutually orthogonal primitive idempotents in $\DSalg_r$.  However, the conjugacy classes of the idempotents appearing in such a decomposition are unique.  Hence, $e_\lambda$ is a primitive idempotent in $\DSalg_r$ that is uniquely defined up to conjugacy.

For $\lambda \vdash r \ge 1$ and $i \in \N$, define
\begin{equation} \label{humps}
    e_\varnothing^{(i)}
    := \frac{1}{\delta^i}
    \begin{tikzpicture}[anchorbase]
        \draw[<->] (-0.6,0.6) -- (-0.6,0.3) to[out=down,in=down,looseness=1.5] (-0.3,0.3) -- (-0.3,0.6);
        \draw[<->] (0.3,0.6) -- (0.3,0.3) to[out=down,in=down,looseness=1.5] (0.6,0.3) -- (0.6,0.6);
        \draw[->] (-0.6,-0.6) -- (-0.6,-0.4);
        \draw[<-] (-0.6,-0.4) -- (-0.6,-0.3) to[out=up,in=up,looseness=1.5] (-0.3,-0.3) -- (-0.3,-0.6);
        \draw[->] (0.3,-0.6) -- (0.3,-0.4);
        \draw[<-] (0.3,-0.4) -- (0.3,-0.3) to[out=up,in=up,looseness=1.5] (0.6,-0.3) -- (0.6,-0.6);
        \opendot{-0.6,0.4};
        \opendot{0.3,0.4};
        \opendot{-0.6,-0.4};
        \opendot{0.3,-0.4};
        \node at (0.01,0.35) {$\cdots$};
        \node at (0.01,-0.35) {$\cdots$};
    \end{tikzpicture}
    \in \DSalg_{2i} \text{ if } \delta \ne 0,
    \qquad
    e_\lambda^{(i)}
    := \gamma_{r,i}(e_\lambda)
    =
    \begin{tikzpicture}[anchorbase]
        \draw[multi,->] (0,0) -- (0,0.7) \toplabel{r};
        \draw[multi,->] (0.2,-1) \botlabel{r-1} -- (0.2,-0.1);
        \genbox{-0.4,-0.1}{0.4,0.3}{e_\lambda};
        \draw[<->] (-0.5,0.7) -- (-0.5,0.4) to[out=down,in=down,looseness=1.5] (-0.8,0.4) -- (-0.8,0.7);
        \draw[<->] (-1.4,0.7) -- (-1.4,0.4) to[out=down,in=down,looseness=1.5] (-1.7,0.4) -- (-1.7,0.7);
        \opendot{-0.5,0.5};
        \opendot{-1.4,0.5};
        \node at (-1.06,0.45) {$\cdots$};
        \draw[<-] (-0.2,-0.1) to[out=down,in=up] (-1.7,-0.8) -- (-1.7,-1);
        \draw[->] (-0.2,-1) -- (-0.2,-0.7) to[out=up,in=up,looseness=1.5] (-0.5,-0.7) -- (-0.5,-0.8);
        \draw[<-] (-0.5,-0.8) -- (-0.5,-1);
        \draw[->] (-1.1,-1) -- (-1.1,-0.7) to[out=up,in=up,looseness=1.5] (-1.4,-0.7) -- (-1.4,-0.8);
        \draw[<-] (-1.4,-0.8) -- (-1.4,-1);
        \node at (-0.8,-0.75) {$\cdots$};
        \opendot{-0.5,-0.8};
        \opendot{-1.4,-0.8};
    \end{tikzpicture}
    \in \DSalg_{r+2i},
\end{equation}
where both diagrams contain $i$ cups and $i$ caps.  Note that we have defined $e_\lambda^{(i)}$ for all partitions $\lambda$ and $i \in \N$, except for when $\lambda = \varnothing$ and $\delta = 0$.  It is straightforward to verify that $e_\lambda^{(i)}$ is an idempotent whenever it is defined.

\begin{prop} \label{primidem}
    Suppose $r \in \N$, $r \ge 2$.
    \begin{enumerate}
        \item If $\delta \ne 0$, then $\{ e_\lambda^{(i)} : 0 \le i \le \frac{r}{2},\ \lambda \vdash r-2i \}$ is a complete set of pairwise-nonconjugate, primitive idempotents in $\DSalg_r$.
            
        \item If $\delta = 0$, then $\{ e_\lambda^{(i)} : 0 \le i < \frac{r}{2},\ \lambda \vdash r-2i \}$ is a complete set of pairwise-nonconjugate, primitive idempotents in $\DSalg_r$.
    \end{enumerate}
\end{prop}

\begin{proof}
    We prove the result by induction on $r$. Suppose $r=2$.  If $\delta \ne 0$, then it is straightforward to verify that
    \[
        \DSalg_2/\DSalg_2 e_\varnothing^{(1)} \DSalg_2 \cong \Hecke_2
        \qquad \text{and} \qquad
        e_\varnothing^{(1)} \DSalg_2 e_\varnothing^{(1)} \cong \kk e_\varnothing^{(1)} \cong \kk,
    \]
    as algebras.  Thus, the result follows from \cref{raft}.  On the other hand, if $\delta = 0$, then
    \[
        f = 
        \begin{tikzpicture}[centerzero]
            \draw[<->] (-0.6,0.55) -- (-0.6,0.25) to[out=down,in=down,looseness=1.5] (-0.3,0.25) -- (-0.3,0.55);
            \draw[->] (-0.6,-0.55) -- (-0.6,-0.35);
            \draw[<-] (-0.6,-0.35) -- (-0.6,-0.25) to[out=up,in=up,looseness=1.5] (-0.3,-0.25) -- (-0.3,-0.55);
            \opendot{-0.6,0.35};
            \opendot{-0.6,-0.35};
        \end{tikzpicture}
    \]
    generates a nilpotent ideal $\DSalg_2 f \DSalg_2$, and $\DSalg_2/ \DSalg_2 f \DSalg_2 \cong \Hecke_2$ as algebras, by \cref{claptrap}.  Thus, the result follows.
    
    Now suppose that $r \ge 3$.  Taking $\xi = \xi_{r-2,1}$ in \cref{raft}, the result follows by induction, using \cref{jack,claptrap}.  (For the case $r=3$, we use the identification $\DSalg_1 \cong \kk$ in the induction step.)
\end{proof}

\begin{lem} \label{accordian}
    Suppose $\lambda$ is a partition and $i \in \N$, such that $e_\lambda^{(i)}$ is defined.  Then $\obj{e_\lambda}$ and $\obj{e_\lambda^{(i)}}$ are isomorphic in $\Kar(\DScat(q,t))$.
\end{lem}

\begin{proof}
    If $\lambda \ne \varnothing$, we have mutually inverse isomorphisms
    \[
        \begin{tikzpicture}[anchorbase]
            \draw[multi,->] (0,0) -- (0,0.7) \toplabel{r};
            \draw[multi,->] (0.2,-1) \botlabel{r-1} -- (0.2,-0.1);
            \genbox{-0.4,-0.1}{0.4,0.3}{e_\lambda};
            \draw[<-] (-0.2,-0.1) to[out=down,in=up] (-1.7,-0.8) -- (-1.7,-1);
            \draw[->] (-0.2,-1) -- (-0.2,-0.7) to[out=up,in=up,looseness=1.5] (-0.5,-0.7) -- (-0.5,-0.8);
            \draw[<-] (-0.5,-0.8) -- (-0.5,-1);
            \draw[->] (-1.1,-1) -- (-1.1,-0.7) to[out=up,in=up,looseness=1.5] (-1.4,-0.7) -- (-1.4,-0.8);
            \draw[<-] (-1.4,-0.8) -- (-1.4,-1);
            \node at (-0.8,-0.75) {$\cdots$};
            \opendot{-0.5,-0.8};
            \opendot{-1.4,-0.8};
        \end{tikzpicture}
        \colon \obj{e_\lambda^{(i)}} \to \obj{e_\lambda}
        \qquad \text{and} \qquad
        \begin{tikzpicture}[anchorbase]
            \draw[multi,->] (0,-0.5) -- (0,0.7) \toplabel{r};
            \genbox{-0.4,-0.1}{0.4,0.3}{e_\lambda};
            \draw[<->] (-0.5,0.7) -- (-0.5,0.4) to[out=down,in=down,looseness=1.5] (-0.8,0.4) -- (-0.8,0.7);
            \draw[<->] (-1.4,0.7) -- (-1.4,0.4) to[out=down,in=down,looseness=1.5] (-1.7,0.4) -- (-1.7,0.7);
            \opendot{-0.5,0.5};
            \opendot{-1.4,0.5};
            \node at (-1.05,0.45) {$\cdots$};
        \end{tikzpicture}
        \colon \obj{e_\lambda} \to \obj{e_\lambda^{(i)}}.
    \]
    If $\lambda = \varnothing$, then
    \[
        \frac{1}{\delta^i}
        \begin{tikzpicture}[anchorbase]
            \draw[->] (-0.6,-0.6) -- (-0.6,-0.4);
            \draw[<-] (-0.6,-0.4) -- (-0.6,-0.3) to[out=up,in=up,looseness=1.5] (-0.3,-0.3) -- (-0.3,-0.6);
            \draw[->] (0.3,-0.6) -- (0.3,-0.4);
            \draw[<-] (0.3,-0.4) -- (0.3,-0.3) to[out=up,in=up,looseness=1.5] (0.6,-0.3) -- (0.6,-0.6);
            \opendot{-0.6,-0.4};
            \opendot{0.3,-0.4};
            \node at (0.01,-0.35) {$\cdots$};
        \end{tikzpicture}
        \colon \obj{e_\varnothing^{(i)}} \to \obj{e_\varnothing}
        \qquad \text{and} \qquad
        \begin{tikzpicture}[anchorbase]
            \draw[<->] (-0.6,0.6) -- (-0.6,0.3) to[out=down,in=down,looseness=1.5] (-0.3,0.3) -- (-0.3,0.6);
            \draw[<->] (0.3,0.6) -- (0.3,0.3) to[out=down,in=down,looseness=1.5] (0.6,0.3) -- (0.6,0.6);
            \opendot{-0.6,0.4};
            \opendot{0.3,0.4};
            \node at (0.01,0.35) {$\cdots$};
        \end{tikzpicture}
        \colon \obj{e_\varnothing} \to \obj{e_\varnothing^{(i)}}
    \]
    are mutually inverse. 
\end{proof}

\begin{theo} \label{31mile}
    A complete set of pairwise-nonisomorphic representatives of the isomorphism classes of indecomposable objects in $\Kar(\DScat(q,t))$ is given by
    \[
        \obj{e_\lambda},\qquad
        \lambda \in \Par,
    \]
    where $\Par$ denotes the set of all partitions.
\end{theo}

\begin{proof}
    It follows from \cref{primidem,accordian} that every indecomposable object is isomorphic to $\obj{e_\lambda}$ for some $\lambda \in \Par$.  Now suppose that $\lambda,\mu \in \Par$, $\lambda \ne \mu$.  If $|\lambda|-|\mu|$ is odd, then there is no morphism between $\uparrow^{\otimes |\lambda|}$ and $\uparrow^{\otimes |\mu|}$.  On the other hand, if $|\lambda|-|\mu|$ is even then, without loss of generality, assume $i = (|\lambda|-|\mu|)/2 > 0$.  Then
    \[
        \im(e_\mu) \cong \im \left( e_\mu^{(i)} \right) \not\cong e_\lambda
    \]
    by \cref{primidem}.  In either case $\im(e_\lambda) \not\cong \im(e_\mu)$, completing the proof of the theorem.
\end{proof}

\section{Submodules of the disoriented skein category\label{sec:DSapp}}

In this section we classify the $\Kar(\OScat(q,t))$-submodules of $\Kar(\DScat(q,t))$.  When $q=1$, this amounts to classifying the monoidal subcategories of the Karoubi envelope of the Brauer category (which is Deligne's interpolating category for the orthogonal groups); see \cref{medusa}.  Since this was already done in \cite[Th.~7.1]{Cou18}, we assume throughout this section that $\kk$ is a field, and we fix elements $q,t \in \kk^\times$ such that $q \ne 1$.  Then $\delta \in \kk$ is uniquely determined by \cref{delta}.  (We will make additional assumptions on $\kk$ below.)

\begin{prop} \label{semisimple}
    The category $\Kar(\DScat(q,t))$ is semisimple if and only if the following conditions are satisfied:
    \begin{itemize}
        \item $q$ is not a root of unity;
        \item for all $a \in \Z$, $t \ne \pm q^a$.
    \end{itemize}
\end{prop}

\begin{proof}
    Since $\Kar(\DScat(q,t))$ is Krull--Schmidt, it is semisimple if and only if all its endomorphism algebras are semisimple.
    \details{
        In a Krull--Schmidt category $\cC$, every indecomposable object $X$ has local endomorphism ring $\End(X)$.  If $\End(X)$ is semisimple and local, then in must be a division ring (a semisimple Artinian ring that is local cannot split as a product of two nonzero matrix blocks).  Hence, $\End(X)$ is a division ring and Schur's lemma implies that $X$ is simple. As every object decomposes into indecomposables, every object is a finite direct sum of simples.  Thus, $\cC$ is semisimple.
    }
    A $\kk$-algebra $A$ is semisimple if and only if $eAe$ is semisimple for all idempotents $e \in A$.  Thus, $\Kar(\DScat(q,t))$ is semisimple if and only if all the endomorphism algebras of $\DScat(q,t)$ are semisimple.  By \cite[Prop.~7.5]{SSS25}, the endomorphism algebras of $\DScat(q,t)$ are isomorphic to the \emph{$q$-Brauer algebras} of \cite{RSS24}.  Thus, if $e$ is the quantum characteristic of $q$, then \cite[Th.~7.2]{RSS24} implies that $\DSalg_n(q,t)$ is semisimple if and only if $e > n$ and $t \ne \pm q^a$ for
    \[
        a \in \{i \in \Z : 4-2n \le i \le n-2\} \setminus \{i \in 2\Z + 1 : 4-2n < i \le 3-n\}.
    \]
    Thus, $\DSalg_n(q,t)$ is semisimple for \emph{all} $n$ when $e = \infty$ and $t \ne \pm q^a$ for 
    \[
        a \in \bigcup_{n \in \N} \big( \{i \in \Z : 4-2n \le i \le n-2\} \setminus \{i \in 2\Z + 1 : 4-2n < i \le 3-n\} \big)
        = \Z.
        \qedhere
    \]
    \details{
    For each $n\in\mathbb N$ set
    \[
    S_n \;:=\; \{\,i\in\mathbb Z \mid 4-2n \le i \le n-2\,\}
    \setminus
    \{\,i\in 2\mathbb Z+1 \mid 4-2n < i \le 3-n\,\}.
    \]
    We show that $\bigcup_{n\in\mathbb N} S_n=\mathbb Z$.
    
    Fix $a\in\mathbb Z$ and choose $n\ge \max\{a+2,\,4-a\}$. Then $a\le n-2$ holds since
    $n\ge a+2$, and $4-2n\le a$ holds since $2n\ge 4-a$ (because $n\ge 4-a$). Hence
    $a\in\{\,i\in\mathbb Z\mid 4-2n\le i\le n-2\,\}$.
    
    Moreover, if $a$ is odd, then $a\notin \{\,i\in 2\mathbb Z+1\mid 4-2n<i\le 3-n\,\}$:
    indeed $n\ge 4-a$ implies $3-n\le a-1$, hence $a>3-n$, so $a$ cannot satisfy $a\le 3-n$.
    Therefore $a\in S_n$ for this choice of $n$, and since $a$ was arbitrary we obtain
    $\mathbb Z\subseteq \bigcup_{n\in\mathbb N} S_n$. Thus $\bigcup_{n\in\mathbb N} S_n=\mathbb Z$.
    }
\end{proof}

It follows from \cref{semisimple,dns} that, if $q$ is not a root of unity and $t \ne \pm q^a$ for all $a \in \Z$, then the only $\Kar(\OScat(q,t))$-submodules of $\Kar(\DScat(q,t))$ are the trivial ones.  We will not treat the root of unity case here, focussing instead on the case of generic $q$ and $\operatorname{char}(\kk)=0$.  Thus, in light of the isomorphism $\DScat(q,t) \xrightarrow{\cong} \DScat(q,-t)$ described in \cref{medusa}, we assume for the remainder of this section that
\begin{center}
    $\kk$ is a field extension of $\Q(q)$
    \quad \text{and} \quad
    $t=q^d$ for some $d \in \Z$,
\end{center}
where $q$ is transcendental over $\Q$.  To simplify notation, we set
\[
    \OScat = \OScat(q,q^d)
    \qquad \text{and} \qquad
    \DScat = \DScat(q,q^d).
\]

\begin{rem}
    We expect that the results of this section hold for $\kk$ an arbitrary field (not necessarily of characteristic zero), $q$ not a root of unity, and $t=q^d$ for some $d \in \Z$.  The reason for the stronger assumptions above is that they are needed in our proof of \cref{brown}.
\end{rem}

\subsection{Pagoda structure}

The goal of this subsection is to endow the disoriented skein category with the structure of a pagoda as in \cref{pagoda}, following the approach of \cite[\S6.1]{Cou18}.  Let $\cC_0 := \OScat$ and $\cM_0 := \DScat$.  Then $(\cM_0,\cC_0)$ satisfies \cref{jetlag}, with the cylinder twist structure from \cref{DStwist}.  Let $\cC = \Kar(\cC_0)$ and $\cM = \Kar(\cM_0)$.  These satisfy \cref{latte}.

Let $\cS_0$ be the full monoidal subcategory of $\OScat$ on the objects $\upobj^{\otimes r}$.  We have $\cS_0(\upobj^{\otimes r}, \upobj^{\otimes s}) = 0$ if $r \ne s$, and
\[
    \cS_0(\upobj^{\otimes r}, \upobj^{\otimes r}) \cong \Hecke_r,
\]
where, as in \cref{sec:DSindec}, $\Hecke_r$ denotes the Iwahori--Hecke algebra of type $A_{r-1}$.  Then $\cS := \Kar(\cS_0)$ is a Krull--Schmidt monoidal category with $\End_\cS(\one) = \kk 1_\one \cong \kk$.  As discussed in \cref{sec:DSindec}, we have a bijection
\[
    \Lambda := \Par \to \indec(\cS),\qquad \lambda \mapsto \left( \uparrow^{\otimes |\lambda|}, \te_\lambda \right).
\]
Defining
\[
    \ell \colon \Lambda \to \N,\qquad \ell(\lambda) = |\lambda|,
\]
we see that $(\cS,\ell)$ is a tiered monoidal category.  We have the inclusion monoidal functor
\[
    \bT \colon \cS \to \cC.
\]

\begin{lem} \label{beats}
    For all objects $X$ in $\cS$, we have $\bT(X)^\vee \cong \bT(X)$ in $\cM$.
\end{lem}

\begin{proof}
    Since $\cS$ is Krull--Schmidt, it suffices to consider the case where $X$ is indecomposable.  Let $f$ be a primitive idempotent element of $\End_\cS(\upobj^{\otimes r}) = \End_\cC(\upobj^{\otimes r})$.  The duality functor on $\cC$ maps $f$ to its $180\degree$ rotation $f^\vee$.  We need to show that
    \[
        (\upobj^{\otimes r}, f) \cong (\downobj^{\otimes r}, f^\vee) \qquad \text{in } \cM.
    \]
    Let $h = \rev_r \circ \togupdown^{\otimes r} \in \DScat(\downobj^{\otimes r}, \upobj^{\otimes r})$, where $\operatorname{rev}_r \in \End_\cS(\upobj^{\otimes r}) = \End_\cC(\upobj^{\otimes r}) \cong \Hecke_r$ is the morphism corresponding to the longest element of the Weyl group.  For example,
    \[
        \rev_3
        =
        \begin{tikzpicture}[centerzero]
            \draw[->] (0.4,-0.4) -- (-0.4,0.4);
            \draw[over,->] (0,-0.4) to[out=135,in=down] (-0.32,0) to[out=up,in=225] (0,0.4);
            \draw[over,->] (-0.4,-0.4) -- (0.4,0.4);
        \end{tikzpicture}
    \]
    Then
    \[
        (\downobj^{\otimes r}, f^\vee) \cong (\upobj^{\otimes r}, h \circ f^\vee \circ h^{-1}) \qquad \text{in } \cM.
    \]
    Recall that, in a cellular algebra, every primitive idempotent is conjugate to its image under the cellular involution.  (See \cite[Theorem 3.7(i)]{GL96} or \cite[Prop.~5.1]{KX98}.)  Thus,
    \[
        (\upobj^{\otimes r},f) \cong (\upobj^{\otimes r}, f^*),
    \]
    where $f^*$ is the image of $f$ under the cellular involution of the Iwahori--Hecke algebra.
    
    To complete the proof, it suffices to show that $f^* = h \circ f^\vee \circ h^{-1}$, where $h^{-1} = \togdownup^{\otimes r} \circ \rev_r^{-1}$.  Since $\End_\cS(\upobj^{\otimes r})$ is generated, as an algebra, by crossings $\uparrow^{\otimes i} \otimes \posupcross \otimes \uparrow^{r-i-2}$, $1 \le i \le r-2$, it is enough to consider the case $f = f_1 \circ f_2 \circ \dotsb \circ f_k$, where each $f_j$, $1 \le j \le k$, is such a crossing.  The cellular involution on the Iwahori--Hecke algebra fixes such crossings, and so we have
    \[
        f^* = f_k \circ \dots \circ f_2 \circ f_1.
    \]
    For $1 \le i \le r-2$,
    \[
        \left( \uparrow^{\otimes i} \otimes \posupcross \otimes \uparrow^{r-i-2} \right)^\vee
        = \left( \downarrow^{\otimes r-i-2} \otimes \posdowncross \otimes \downarrow^i \right)^\vee
    \]
    and
    \[
        \rev_r \circ \left( \uparrow^{\otimes (r-i-2)} \otimes \posupcross \otimes \uparrow^i \right) \circ \rev_r^{-1}
        = \ \uparrow^{\otimes i} \otimes \posupcross \otimes \uparrow^{r-i-2}.
    \]
    By \cite[(2.25)]{SSS25}, in $\DScat$ (using the convention \cref{toggy}) we have
    \[
        \begin{tikzpicture}[centerzero]
            \draw[->] (0.3,-0.5) -- (0.3,-0.3);
            \draw[<-] (0.3,-0.3) \braidup (-0.3,0.3);
            \draw[->] (-0.3,0.3) -- (-0.3,0.5);
            \draw[->] (-0.3,-0.5) -- (-0.3,-0.3);
            \draw[over,<-] (-0.3,-0.3) \braidup (0.3,0.3);
            \draw[->] (0.3,0.3) -- (0.3,0.5);
            \opendot{-0.3,-0.3};
            \opendot{0.3,-0.3};
            \opendot{-0.3,0.3};
            \opendot{0.3,0.3};
        \end{tikzpicture}
        =
        \begin{tikzpicture}[centerzero]
            \draw[->] (-0.3,-0.5) -- (-0.3,-0.3) \braidup (0.3,0.3) -- (0.3,0.5);
            \draw[->] (0.3,-0.5) -- (0.3,-0.3);
            \draw[over,<-] (0.3,-0.3) \braidup (-0.3,0.3);
            \draw[->] (-0.3,0.3) -- (-0.3,0.5);
            \opendot{0.3,-0.3};
            \opendot{-0.3,0.3};
        \end{tikzpicture}
        =
        \begin{tikzpicture}[centerzero]
            \draw[->] (-0.3,-0.5) -- (-0.3,-0.3) \braidup (0.3,0.3) -- (0.3,0.5);
            \draw[over,->] (0.3,-0.5) -- (0.3,-0.3) \braidup (-0.3,0.3) -- (-0.3,0.5);
        \end{tikzpicture}
        \ .
    \]
    It follows that
    \[
        h \circ f^\vee \circ h^{-1}
        = (h \circ f_k^\vee \circ h^{-1}) \circ \dotsb \circ (h \circ f_2^\vee \circ h^{-1}) \circ (h \circ f_1^\vee \circ h^{-1})
        = f_k \circ \dotsb \circ f_2 \circ f_1,
    \]
    as desired.
\end{proof}

The following result gives an explicit characterization of the partial order $\preceq$ defined in \cref{minion}.  While well known, we were not able to find the precise statement in the literature, and so we have included a proof.  For $\lambda,\mu \in \Par$, we write $\lambda \subseteq \mu$ if the Young diagram of $\lambda$ in included in that of $\mu$.  Equivalently,
\[
    \lambda \subseteq \mu
    \iff \lambda_i \le \mu_i \text{ for all } i.
\]

\begin{lem} \label{aha}
  For $\lambda, \mu \in \lambda$, we have $\lambda \preceq \mu$ if and only if $\lambda \subseteq \mu$.
\end{lem}

\begin{proof}
    Suppose $\lambda \preceq \mu$.  Equivalently, the Littlewood--Richardson coefficient $c^{\mu}_{\lambda,\nu}>0$ for some $\nu$.  By the Littlewood--Richardson rule, there exists an Littlewood--Richardson tableau of skew shape $\mu/\lambda$.  In particular, $\mu/\lambda$ must be a valid skew Young diagram, which forces $\lambda \subseteq \mu$.

    Conversely, assume $\lambda \subseteq \mu$.  Then the skew Schur function $s_{\mu/\lambda}$ is nonzero, since there exists at least one semistandard Young tableau of shape $\mu/\lambda$.  For example, one may fill the box in row~$i$ and column~$j$ with the entry~$i$.  The Littlewood--Richardson expansion reads
    \[
        s_{\mu/\lambda}
        = \sum_{\nu} c^{\mu}_{\lambda,\nu}\, s_\nu ,
    \]
    with each coefficient $c^{\mu}_{\lambda,\nu}\ge 0$.   As $s_{\mu/\lambda} \neq 0$, at least one coefficient must be positive; hence $c^{\mu}_{\lambda,\nu}>0$ for some~$\mu$.
\end{proof}

\begin{lem} \label{base-change}
    Let $A \to A'$ be a homomorphism of commutative rings, $R$ be an $A$-algebra, $P$ be a finitely-generated projective left $R$-module, and $M$ be an arbitrary left $R$-module.  Then the map
    \[
        A' \otimes_A \Hom_R(P,M) \to \Hom_{A' \otimes_A R}(A' \otimes_A P, A' \otimes_A M),\quad
        (r \otimes f) \mapsto \Big( s \otimes p \mapsto rs \otimes f(p) \Big),
    \]
    is an isomorphism of $A'$-modules.
\end{lem}

\begin{proof}
    Since $P$ is a finitely-generated projective $R$-module, there is a finitely-generated $R$-module $Q$ such that $P \oplus Q \cong R^{\oplus n}$ for some $n \in \N$.  It is straightforward to verify that the statement holds when $P=R^{\oplus n}$.
    \details{
        If $P = R^{\oplus n}$ for some $n \in \N$, then
        \[
            A' \otimes_A \Hom_R(R^{\oplus n},M)
            \cong A' \otimes_A M^{\oplus n}
            \cong (A' \otimes_A M)^{\oplus n}
        \]
        and
        \[
            \Hom_{A' \otimes_A R} (A' \otimes_A R^{\oplus n}, A' \otimes_A M)
            \cong \Hom_{A' \otimes_A R} \big( (A' \otimes_A R)^{\oplus n}, A' \otimes_A M \big)
            \cong (A' \otimes_A M)^{\oplus n}.
        \]
        The map in the statement of the lemma realizes the isomorphism between the two.
    }
    Then,
    \begin{multline*}
        \big( A' \otimes_A \Hom_R(P,M) \big) \oplus \big( A' \otimes_A \Hom_R(Q,M) \big)
        \\
        \cong A' \otimes_A \Hom_R(R^{\oplus n}, M)
        \overset{*}{\cong} \Hom_{A' \otimes R} (A' \otimes_A R^{\oplus n}, A' \otimes_A M)
        \\
        \cong \Hom_{A' \otimes R} (A' \otimes_A P, A' \otimes_A M) \oplus \Hom_{A' \otimes R} (A' \otimes_A Q, A' \otimes_A M).
    \end{multline*}
    Since the restriction of the isomorphism labeled $*$ clearly sends the summand $A' \otimes_A \Hom_R(P,M)$ to the summand $\Hom_{A' \otimes R}(A' \otimes_A P, A' \otimes_A M)$, the result follows.
\end{proof}

Define
\[
    \DSalg_r= \DSalg_r(q,q^d) := \End_\DScat(\upobj^{\otimes r}),\qquad r \in \N.
\]
By \cref{31mile}, we have a bijection
\[
    R \colon \Lambda = \Par \xrightarrow{\cong} \indec(\cM),\qquad R(\lambda) = \im(e_\lambda).
\]
For $j \in \Z_{>0}$, define
\begin{equation}\label{lime}
    m_j :=
    \begin{cases}
        d + 2j-2 & \text{if } d>0, \\
        2j-2 & \text{if } d \in -2\N, \\
        2j-1 & \text{if } d \in -2\N-1,
    \end{cases}
    \quad
    n_j := \frac{m_j-d}{2} =
    \begin{cases}
        j-1 & \text{if } d>0, \\
        j - \frac{d}{2} - 1 & \text{if } d \in -2\N, \\
        j - \frac{d}{2} - \frac{1}{2} & \text{if } d \in -2\N - 1,
    \end{cases}
\end{equation}
and
\[
    r_j := (m_j + 1)(n_j+1).
\]

\begin{prop}[{cf.\ \cite[Lem.~6.1.4]{Cou18}}] \label{brown}
    Define $\Upsilon := \{ \nu^{(j)} : j \in \N \} \subseteq \Par$, where
    \begin{equation}\label{nujdef}
        \nu^{(0)} = \varnothing
        \quad \text{and} \quad
        \nu^{(j)} = \left( (2n_j+2)^{m_j+1} \right)
        \quad \text{for } j > 0.
    \end{equation}
    For $\lambda \in \Par$,
    \begin{equation} \label{tricky}
        \dim_\kk \cM(\one,R(\lambda))
        \le
        \begin{cases}
            1 & \text{if } \lambda \in \Upsilon, \\
            0 & \text{otherwise}.
        \end{cases}
    \end{equation}
    In particular, $\bB_\cM \subseteq \{ R(\nu) : \nu \in \Upsilon \}$.
\end{prop}

We will show in \cref{bibimbap} that the inequality \cref{tricky} is actually equality.

\begin{proof}
    Let $A := \C \llbracket q-1 \rrbracket$, with field of fractions $\C \Laurent{q-1}$.  Then $\kk \otimes_{\Q(q)} \C\Laurent{q-1}$ is a ring of characteristic zero.  Pick a maximal ideal $\mathfrak{m}$ of $\kk\otimes_{\Q(q)} \C\Laurent{q-1}$ and define
    \[
        \bbK := \left( \kk\otimes_{\Q(q)} \C\Laurent{q-1} \right)/\mathfrak{m},
    \]
    which is a field containing both $\kk$ and $\C\Laurent{q-1}$. We view $\C$ as an $A$-algebra via $q \mapsto 1$.  Note that $A$ is a complete discrete valuation ring with residue field $\C$.  We summarize these rings and the relationships between them:
    \[
        \kk \hookrightarrow \bbK
        \hookleftarrow A = \C \llbracket q-1 \rrbracket
        \xrightarrowdbl{q \mapsto 1} \C.
    \]
    For the remainder of this proof, $\bullet$ denotes an arbitrary element of $\{A,\kk,\bbK,\C\}$ and $\bbF$ denotes an arbitrary element of $\{\kk,\bbK,\C\}$.   Since we will work over multiple ground rings, we will denote the ground ring by a superscript.  For example, we let $\DSalg^\bullet_r = \DSalg^\bullet_r(q,q^d)$ be the algebra defined in \cref{DSalg} over $\bullet$.  It follows from the basis theorem \cite[Th.~7.1]{SSS25} that
    \[
        \DSalg^\bbK_r = \bbK \otimes_A \DSalg^A_r
        \qquad \text{and} \qquad
        \DSalg^\C_r = \C \otimes_A \DSalg^A_r.
    \]
    
    Recall the primitive idempotents $e^\bbF_\lambda$, $\lambda \vdash r$, defined as in \cref{PVD}.  Since the composite
    \[
        \Hecke^\bbF_r \hookrightarrow \DSalg^\bbF_r \xrightarrow[\cong]{\pi_r} \Hecke^\bbF_r
    \]  
    is the identity map of $\Hecke_r^\bbF$, we have $\pi_r(e_\lambda^\bbF) = \te_\lambda^\bbF$.  The algebra homomorphism $\pi_r$ induces a surjective homomorphism of  $\DSalg_r^\bbF$-modules $\DSalg^\bbF_re_\lambda^\bbF \to \Hecke_r^\bbF\te_\lambda^\bbF$.  Since $\Hecke_r^\bbF\te_\lambda^\bbF$ is the simple $\Hecke_r^\bbF$-module associated to $\lambda$, it follows that $\Hecke_r^\bbF\te_\lambda^\bbF$ is absolutely irreducible, and that $R^\bbF(\lambda) = \DSalg^\bbF_r e^\bbF_\lambda$ is its projective cover (as a $\DSalg^\bbF_r$-module).  Thus, the inflation of $\Hecke_r^\bbF \te_\lambda^\bbF$ is precisely the simple head $L^\bbF(\lambda)$ of $R^\bbF(\lambda)$.  Therefore, for any $\DSalg^\bbF_r$-module $M$, we have
    \begin{equation} \label{oyster}
        \dim_\bbF \Hom_{\DSalg^\bbF_r} \big( R^\bbF(\lambda), M \big)
        =
        \left[ M : L^\bbF(\lambda) \right].
    \end{equation}
    (See, for example, \cite[Lem.~1.7.6]{Ben98}.  Absolute irreducibility implies that the $\Delta_i$ appearing there is equal to $\bbF$.)
    
    Since $\DSalg^A_r$ is free of finite rank over $A = \C \llbracket q-1 \rrbracket$, it is $(q-1)$-adically complete.  Thus, by \cite[Th.~21.31, Cor.~21.32]{Lam01}, idempotents in $\DSalg^A_r/(q-1)\DSalg^A_r \cong \C \otimes_A \DSalg^A_r$ lift to idempotents in $\DSalg^A_r$.  In particular, there exists a primitive idempotent $e^A_\lambda \in \DSalg^A_r$ such that
    \begin{equation} \label{rabbit}
        \unit{\C} \otimes_A e^A_\lambda = e^\C_\lambda.
    \end{equation}
    Define
    \begin{equation} \label{pizza}
        R^A(\lambda) := \DSalg^A_r e^A_\lambda,
        \quad \text{so that } R^\C(\lambda) = \C\otimes_A R^A(\lambda).
    \end{equation}
    Viewing $\cM^\bbF_0(\one,\upobj^{\otimes r})$ as a left $\DSalg^\bbF_r$-module, we have
    \begin{multline} \label{tricky2}
        \dim_\bbF \cM^\bbF \big( \one, R^\bbF(\lambda) \big)
        = \dim_\bbF e^\bbF_\lambda \cM^\bbF_0 (\one,\upobj^{\otimes r})
        = \dim_\bbF \Hom_{\DSalg^\bbF_r} \big( \DSalg^\bbF_r e^\bbF_\lambda, \cM^\bbF_0 (\one,\upobj^{\otimes r}) \big)
        \\
        = \dim_\bbF \Hom_{\DSalg^\bbF_r} \big( R^\bbF(\lambda), \cM_0^\bbF(\one, \upobj^{\otimes r}) \big)
        \overset{\cref{oyster}}{=}
        \left[ \cM^\bbF_0(\one, \upobj^{\otimes r}) : L^\bbF(\lambda) \right].
    \end{multline}
    \details{
        The first equality in \cref{tricky2} holds by the definition of morphism spaces in the Karoubi envelope.
    }
    When $r$ is odd, $\cM^\kk_0(\one, \upobj^{\otimes r}) = 0$.  Thus, since $|\lambda|$ is even for all $\lambda \in \Upsilon$, \cref{tricky} holds when $|\lambda|$ is odd.  Therefore, we assume $r$ is even for the remainder of the proof.
	
    Let
    \begin{equation} \label{burger}
        C^\bullet(r) := \cM_0^\bullet(\one, \upobj^{\otimes r}).
    \end{equation}
    By the basis theorem \cite[Th.~7.1]{SSS25}, this is a free $\bullet$-module, and so
    \begin{equation} \label{horse}
        C^\bbK(r) = \bbK \otimes_A C^A(r)
        \qquad \text{and} \qquad
        C^\C(r) = \C \otimes_A C^A(r).
    \end{equation}
    We have
    \begin{equation} \label{rice1}
        \dim_\C \Hom_{\DSalg^\C_r} \big( R^\C(\lambda),C^\C(r) \big)
        \overset{\cref{tricky2}}{\underset{\cref{burger}}{=}}
        \dim_\C \cM^\C \big( \one, R^\C(\lambda) \big)
        =
        \begin{cases}
            1 & \text{if } \lambda \in \Upsilon, \\
            0 & \text{otherwise},
        \end{cases}
    \end{equation}
    where the final equality is \cite[Lem.~6.1.4]{Cou18}.   Using \cref{horse,pizza}, together with \cref{base-change} for the ring homomorphisms $A \hookrightarrow \bbK$ and $A \to \C$, we have
    \begin{gather} \label{eggK}
        \Hom_{\DSalg^{\bbK}_r} \big( \bbK \otimes_A R^A(\lambda),C^\bbK(r) \big)
        \cong \bbK \otimes_A \Hom_{\DSalg^A_r} \big( R^A(\lambda),C^A(r) \big),
        \\ \label{eggC}
        \Hom_{\DSalg^\C_r} \big( R^\C(\lambda),C^\C(r) \big)
        \cong \C \otimes_A \Hom_{\DSalg^A_r} \big( R^A(\lambda),C^A(r) \big).
    \end{gather}
    Since $A$ is a principal ideal domain, $\Hom_{\DSalg^A_r} \left( R^A(\lambda),C^A(r) \right)$ is a free $A$-module, because it is a submodule of the free module $\Hom_A \left( R^A(\lambda),C^A(r) \right)$.  Thus, \cref{eggC} implies
    \begin{equation} \label{lobster}
        \rank_A \Hom_{\DSalg^A_r} \big( R^A(\lambda),C^A(r) \big)
        = \dim_\C \Hom_{\DSalg^\C_r} \big( R^\C(\lambda),C^\C(r) \big).
    \end{equation}
    By \cref{lobster,rice1}, we have
    \begin{equation} \label{rice2}
        \rank_A \Hom_{\DSalg^A_r} \big( R^A(\lambda),C^A(r) \big)
        =
        \begin{cases}
            1 & \text{if } \lambda \in \Upsilon, \\
            0 & \text{otherwise}.
        \end{cases}
    \end{equation}
	
    A priori, $\unit{\bbK} \otimes_A e^A_\lambda$ might not be primitive in $\DSalg^{\bbK}_r$.  However, the projective cover $R^{\bbK}(\lambda)$ appears as a direct summand of the projective module $\bbK \otimes_A R^A(\lambda)$.
    \details{
        Recall the quotient map $\pi_r^\bullet \colon \DSalg_r^{\bullet} \twoheadrightarrow H_r^{\bullet}$ introduced below \cref{claptrap}.  Then,
        \[
            \pi_r^{\C}(e_\lambda^{\C})
            \overset{\cref{rabbit}}{=} \pi_r^{\C}(1_\C \otimes_A e_\lambda^A)
            =1_{\C}\otimes_A \pi_r^A(e_\lambda^A).
        \] 
        By \cref{PVD}, $e_\lambda^{\C}$ is chosen such that $\pi_r^{\C}(e_\lambda^{\C})\neq 0$. Hence $\pi_r^A(e_\lambda^A)\neq 0$ and it is not a $(q-1)$-torsion.  Thus, it survives localization:
        \[
            \pi_r^\bbK(1_\bbK\otimes_A e_\lambda^A)
            =1_\bbK\otimes_A \pi_r^A(e_\lambda^A)\neq 0.
        \]
        In particular, by \cref{PVD} again, this implies $e^\bbK_\lambda$ is a summand of $1_\bbK\otimes_A e_\lambda^A$.
    }
    Thus,
    \begin{multline} \label{rice4}
        \dim_{\bbK} \Hom_{\DSalg^{\bbK}_r} \big( R^{\bbK}(\lambda),C^\bbK(r) \big)
        \le \dim_\bbK \Hom_{\DSalg^{\bbK}_r} \big( \bbK \otimes_A R^A(\lambda), C^\bbK(r) \big)
        \\
        \overset{\cref{eggK}}{=} \rank_A \Hom_{\DSalg^A_r} \big( R^A(\lambda),C^A(r) \big)
        \overset{\cref{rice2}}{=}
        \begin{cases}
            1 & \text{if } \lambda \in \Upsilon, \\
            0 & \text{otherwise}.
        \end{cases}
    \end{multline}
    Thus,
    \begin{equation} \label{rice5}
        \left[ C^\bbK(r) : L^\bbK(\lambda) \right]
        \overset{\cref{oyster}}{\underset{\cref{burger}}{=}}
        \dim_{\bbK} \Hom_{\DSalg^{\bbK}_r} \big( R^{\bbK}(\lambda),C^\bbK(r) \big)
        \overset{\cref{rice4}}{\le}
        \begin{cases}
            1 & \text{if } \lambda \in \Upsilon, \\
            0 & \text{otherwise}.
        \end{cases}
    \end{equation}

    Now suppose, towards a contradiction, that $\left[ C^\kk(r) : L^\kk(\lambda)\right] \geq 2$.  Then we can choose a composition series of $C^\kk(r)$,
	\begin{equation} \label{cs1}
		0=M_0\subset M_1 \subset \cdots \subset M_l = C^\kk(r),
	\end{equation}
    with at least two subquotients $M_i/M_{i-1}$ isomorphic to $L^\kk(\lambda)$. By  extension of scalars from $\kk$ to $\bbK$, we obtain from \cref{cs1} a series of $\DSalg^{\bbK}_r$-modules
    \begin{equation} \label{cs2}
        0=\bbK\otimes_\kk M_0\subset \bbK\otimes_\kk M_1 \subset \cdots \subset \bbK\otimes_\kk M_l
        = \bbK\otimes_\kk C^\kk(r)
        = C^\bbK(r),
    \end{equation}
    with at least two subquotients isomorphic to $\bbK \otimes_\kk L^\kk(\lambda) \cong L^\bbK(\lambda)$. (The latter isomorphism holds because $L^\kk(\lambda)$ is absolutely irreducible.)  But this contradicts~\cref{rice5}.	Thus, \cref{tricky} follows from \cref{tricky2,rice5}.
\end{proof}

For $a,b \in \N$, we say that $\lambda,\mu \in \Par$ are \emph{$(a \times b)$-dual} if
\[
    \lambda_{a+1} = 0 = \mu_{a+1}
    \quad \text{and} \quad
    \lambda_i + \mu_{a+1-i} = b,\quad
    \text{for } 1 \le i \le a.
\]
Each partition $\lambda \subseteq (b^a)$ has a unique $(a \times b)$-dual.  Note that, if $\lambda$ is $(a \times b)$-self-dual, then $ab \in 2\N$ and $\lambda \vdash \frac{ab}{2}$.

\begin{lem}[{cf.\ \cite[Lem.~6.1.7]{Cou18}}] \label{cortado}
    Suppose $j \in \Z_{\ge 0}$ satisfies $\dim_\kk \cM(\one, R(\nu^{(j)})) = 1$.  With the notation of \ref{P4}, we have
    \begin{equation} \label{pave}
        \Lambda_j = \{ \lambda \in \Par : \lambda \text{ is $\big( (m_j+1) \times (2n_j+2) \big)$-self-dual} \}.
    \end{equation}
    For $\lambda \in \Lambda_j$, we have $\dim_\kk \cM(\bT(\lambda), R(\lambda)) = 1$.
\end{lem}

\begin{proof}
    Under our assumption that $q$ is not a root of unity, the category of finite-dimensional $\Hecke_r(q)$-modules is split semisimple, and its Grothendieck ring under parabolic induction is canonically isomorphic to the ring of symmetric functions.  The structure constants are the Littlewood–-Richardson coefficients.  Thus, \cref{pave} follows from \cref{brown} and \cite[Lem.~1.2.3(ii)]{Cou18}.
    
    Suppose $\lambda \in \Lambda_j$.  We have $\kk$-module isomorphisms
    \[
        \cM(\bT(\lambda), R(\lambda))
        \cong \cM \left( \one, R(\lambda) \otimes \bT(\lambda)^\vee \right)
        \cong \cM(\one, R(\lambda) \otimes \bT(\lambda)),
    \]
    where the second isomorphism follows from \cref{beats}.  Decompose $R(\lambda) \otimes \bT(\lambda)$ as a sum of indecomposable objects.  By \cref{bulb}\ref{bulb1}, the multiplicity of $R(\nu)$ in this decomposition is the Littlewood--Richardson coefficient $c^\nu_{\lambda,\lambda}$.  Thus, we have a $\kk$-module isomorphisms
    \[
        \cM(\one, R(\lambda) \otimes \bT(\lambda))
        \cong \bigoplus_{\nu \in \Par} \cM(\one, R(\nu))^{\oplus c^\nu_{\lambda,\lambda}}
        \cong \cM \left( \one, R(\nu^{(j)}) \right)^{\oplus c^{\nu^{(j)}}_{\lambda,\lambda}}
        \cong \cM \left(\one, R(\nu^{(j)}) \right),
    \]
    where the second isomorphism follows from \cref{brown}, and the third isomorphism follows from \cite[Lem.~1.2.3(ii)]{Cou18}.  Therefore, $\dim_\kk \cM(\bT(\lambda), R(\lambda)) = 1$ by \cref{brown}.
\end{proof}

\begin{prop}\label{panda}
    The tuple $(\cM,\cC,\cS,\ell,\id_\Par,\bT)$ is a pagoda, where $\id_\Par$ denotes the identity map on the set $\Par$ of partitions.
\end{prop}

\begin{proof}
    We must verify conditions \ref{P1}--\ref{P5} in \cref{pagoda}.
    
    \smallskip
    
    \noindent \emph{Condition} \ref{P1}: Suppose $\lambda \vdash r$.  Recall, from \cref{sec:DSindec}, that $\te_\lambda$ is the primitive idempotent in $\Hecke_r$ corresponding to $\lambda$, and that we have a decomposition
    \[
        \te_\lambda = e_\lambda + \sum_{j=1}^k e_{\mu^{(j)}}^{(i_j)},\qquad
        1 \le i_j \le \tfrac{r}{2},\ \mu^{(j)} \vdash r-2i_j,\quad 1 \le j \le k,
    \]
    into mutually orthogonal primitive idempotents of $\DSalg$.  Since $\bT(\lambda)$ is the object corresponding to the idempotent $\te_\lambda$, this yields a decomposition
    \[
        \bT(\lambda) = R(\lambda) \oplus \bigoplus_{j=1}^k R \big( \mu^{(j)} \big).
    \]
    Fix $j \in \{1,\dotsc,k\}$ and let $i=i_j$, $\mu = \mu^{(j)}$.  Then
    \[
        0 \ne
        e_\mu^{(i)}
        = \te_\lambda e_{\mu}^{(i)}
        \overset{\cref{humps}}{=}
        \begin{tikzpicture}[anchorbase]
            \draw[multi,->] (0,0) -- (0,0.7);
            \draw[multi,->] (0.2,-1) -- (0.2,-0.1);
            \draw[multi,->] (-0.8,1.1) -- (-0.8,1.5);
            \genbox{-0.4,-0.1}{0.4,0.3}{e_\mu};
            \genbox{-2,0.7}{0.4,1.1}{\te_\lambda};
            \draw[<->] (-0.5,0.7) -- (-0.5,0.4) to[out=down,in=down,looseness=1.5] (-0.8,0.4) -- (-0.8,0.7);
            \draw[<->] (-1.4,0.7) -- (-1.4,0.4) to[out=down,in=down,looseness=1.5] (-1.7,0.4) -- (-1.7,0.7);
            \opendot{-0.5,0.5};
            \opendot{-1.4,0.5};
            \node at (-1.06,0.45) {$\cdots$};
            \draw[<-] (-0.2,-0.1) to[out=down,in=up] (-1.7,-0.8) -- (-1.7,-1);
            \draw[->] (-0.2,-1) -- (-0.2,-0.7) to[out=up,in=up,looseness=1.5] (-0.5,-0.7) -- (-0.5,-0.8);
            \draw[<-] (-0.5,-0.8) -- (-0.5,-1);
            \draw[->] (-1.1,-1) -- (-1.1,-0.7) to[out=up,in=up,looseness=1.5] (-1.4,-0.7) -- (-1.4,-0.8);
            \draw[<-] (-1.4,-0.8) -- (-1.4,-1);
            \node at (-0.8,-0.75) {$\cdots$};
            \opendot{-0.5,-0.8};
            \opendot{-1.4,-0.8};
        \end{tikzpicture}
        \ .
    \]
    It follows that $\te_\lambda (1_{2i} \otimes e_\mu) \ne 0$ and hence $\te_\lambda (1_{2i} \otimes \te_\mu) \ne 0$, since $\te_\mu e_\mu = e_\mu$.  This implies that $\lambda \inplus \nu \otimes \mu$ for some $\nu \vdash 2i$.
    \details{
        The category of finite-dimensional $\Hecke_r$ modules is semisimple.  Since
        \[
            0 \ne \te_\lambda (1_{2i} \otimes \te_\mu)
            \in \te_{\lambda} \Hecke_r (1_{2i} \otimes \te_\mu)
            = \Hom_{\Hecke_r} \left( \Ind_{\Hecke_{2i} \otimes \Hecke_{r-2i}}^{\Hecke_r} (\Hecke_{2i} \otimes \mu), \lambda \right),
        \]
        the claim follows.
    }
    Hence, $\mu \prec \lambda$ by \cref{minion}.  Therefore, \ref{P1} is satisfied.
    
    \smallskip
    
    \noindent \emph{Condition} \ref{P2}: This follows from \cref{beats}.
    
    \smallskip
    
    \noindent \emph{Condition} \ref{P3}: This follows from \cref{brown}.
       
    \smallskip
    
    \noindent \emph{Condition} \ref{P4}: This follows from \cref{cortado}.
    
    \smallskip
    
    \noindent \emph{Condition} \ref{P5}: Suppose $j,j' \in \N$ with $j' < j$, and $\lambda \in \Lambda_j$.  By \cref{cortado}, $\lambda$ is $\big( (m_j+1) \times (2n_j+2) \big)$-self-dual.  Let $\lambda'$ be the partition whose Young diagram is obtained from the Young diagram of $\lambda$ by deleting the first $2(j-j')$ rows and columns.  Then $\lambda'$ is $\big( (m_{j'}+1) \times (2n_{j'}+2) \big)$-self-dual and $\lambda' \preceq \lambda$ by \cref{aha}.  Thus, appealing again to \cref{cortado}, we see that condition \ref{P5} holds.
    
    \smallskip
    
    \noindent \emph{Condition} \ref{P6}: Now suppose that $\nu^{(j)} \preceq \kappa \inplus \lambda \otimes \lambda$ for some $\lambda, \kappa \in \Lambda$ and $j \in \N$.  Applying \cite[Lem.~1.2.3(i)]{Cou18} with $\lambda = \mu$, $\nu = \kappa$, $b = 2n_j+2$, and $a = m_j+1$, we have
    \begin{equation} \label{pain1}
        \lambda_i + \lambda_{a+1-i} \ge b \quad \text{for all } 1 \le i \le a.
    \end{equation}
    For $1 \le i \le \lfloor a/2 \rfloor$, set
    \[
        \mu_i = \min (b, \lambda_i)
        \qquad \text{and} \qquad
        \mu_{a+1-i} = b - \mu_i.
    \]
    If $a$ is odd, we also set $\mu_{(a+1)/2} = b/2$.  Then $\mu \in \Par$, $\mu \subseteq \lambda$, and $\mu$ is $(a \times b)$-self-dual.
    \details{
        For $1 \le i \le \lfloor a/2 \rfloor$, it is clear that $\mu_i \le \lambda_i$.  Also,
        \[
            \mu_{a+1-i} = b - \mu_i
            =
            \begin{cases}
                b - \lambda_i & \text{if } \lambda_i < b, \\
                0 & \text{if } \lambda_i \ge b
            \end{cases}
            \overset{\cref{pain1}}{\le} 
            \lambda_{a+1-i}
        \]
        If $a$ is odd, we also have $\mu_{(a+1)/2} = b/2 \le \lambda_{(a+1)/2}$ by \cref{pain1}.  Thus, $\mu \subseteq \lambda$.
    }
    Hence \ref{P6} is satisfied.
\end{proof}

\subsection{Classification of submodules}

For $m,n \in \N$, let $\Uis(m|2n)$ be the iquantum enveloping algebra denoted $\Uis$ in \cite[Def.~5.1]{SSS25}, and let $\Uis(m|2n)\tmod$ be its category of tensor modules.  By definition, this is the full subcategory of $\Uis(m|2n)$-modules whose objects are finite direct sums of summands of restrictions of tensor products of the natural $U_q(\fgl(m|2n))$-module and its dual.  Fix $d \in \Z \subseteq \kk$.  For $m,n \in \N$ with $d=m-2n$, let
\[
    \bR_\DScat^{m,2n} \colon \DScat(q,q^d) \to \Uis(m|2n)\tmod,
\]
be the strict morphism of $\OScat(q,q^d)$-modules from \cite[Th.~6.4]{SSS25}.  This induces a functor
\[
    \Kar(\bR_\DScat^{m,2n}) \colon \Kar(\DScat(q,q^d)) \to \Uis(m|2n)\tmod.
\]
These functors are full by \cite[Th.~8.1]{SSS25}.  Recall that, to simplify notation, we set
\[
    \OScat = \OScat(q,q^d),\qquad
    \DScat = \DScat(q,q^d).
\]

\begin{theo}[{cf.\ \cite[Th.~7.1.1]{Cou18}}] \label{champagne}
    The $\Kar(\OScat)$-submodules of $\Kar(\DScat)$ form a set $\{\cM_j : j \in \N\}$ with
    \begin{equation} \label{prosecco}
        \Kar(\DScat) = \cM_0 \supsetneq \cM_1 \supsetneq \cM_2 \supsetneq \cM_3 \supsetneq \dotsb
    \end{equation}
    and $\Ob \colon \Submod_{\Kar(\OScat)}(\Kar(\DScat)) \to \Thick(K_\oplus(\Kar(\DScat)))$ is an isomorphism.  For $j \in \Z_{>0}$, we have the following descriptions of $\cM_j$, with $m_j$, $n_j$ as in \cref{lime}:
    \begin{enumerate}
        \item \label{champagne1} For $X,Y \in \Ob(\DScat)$, the $\kk$-module $\cM_j(X,Y)$ consists of all morphisms that factor as $X \to Z \to Y$ for some $Z$ equal to a direct sum of objects $R(\mu)$, with $\mu \in \Par$ satisfying
            \begin{equation} \label{jays}
                \mu^t_i + \mu^t_{2n_j+3-i} > m_j
                \qquad \text{for all } 1 \le i \le n_j+1.
            \end{equation}
            
        \item \label{champagne2} We have $\cM_j = \ker \Kar (\bR_\DScat^{m_j,2n_j})$ for all $j > 0$.
        
        \item \label{champagne3} The submodule $M_j := \cM_j(\one,-) \in \Submod_{P(\Kar(\DScat))}(\Kar(\DScat)_\one)$ is determined by
            \[
                M_j \big( R(\nu^{(k)}) \big) = 0 \text{ if } k<j
                \quad \text{and} \quad
                M_j \big( R(\nu^{(k)}) \big) = \DScat \big( \one,R(\nu^{(k)}) \big)
                \quad \text{if } k \ge j,
            \]
            with $\nu^{(j)}$ as in \cref{nujdef}.
    \end{enumerate}
\end{theo}

\begin{proof}
    By \cref{panda}, we can apply \cref{tree}.  The first sentence of \cref{champagne} follows from parts \ref{tree1} and \ref{tree2} of \cref{tree}.  Part~\ref{champagne3} of \cref{champagne} follows from part~\ref{tree3} of \cref{tree}.
    
    Next we prove Part~\ref{champagne1}.  By \cref{BBQ}\ref{BBQ2}, \cref{tree}\ref{tree2}, \cref{aha}, and \cref{cortado}, $\cM_j(X,Y)$ consists of all morphisms that factor as $X \to Z \to Y$ for some $Z$ equal to a direct sum of objects $R(\mu)$ such that $\lambda \subseteq \mu$ for some $\lambda$ that is $\big( (m_j+1) \times (2n_j+2) \big)$-self-dual.  Thus, it suffices to prove that this condition on $\mu$ is equivalent to \cref{jays}.
    
    If $\lambda \subseteq \mu$ and $\lambda$ is $\big( (m_j+1) \times (2n_j+2) \big)$-self-dual, then
    \[
        \mu^t_i + \mu^t_{2n_j+3-i}
        \ge \lambda^t_i + \lambda^t_{2n_j+3-i}
        = m_j+1,
    \]
    and so $\mu$ satisfies \cref{jays}.  Conversely, suppose $\mu$ satisfies \cref{jays}.  Define $\lambda \in \Par$ by
    \[
        \lambda^t_i = \min(\mu^t_i,m_j+1),\qquad
        \lambda^t_{2n_j+3-i} = m_j+1 - \lambda^t_i \ \left( \le \mu^t_{2n_j+3-i} \right), \qquad
        1 \le i \le n_j+1.
    \]
    Then $\lambda \subseteq \mu$, and $\lambda$ is $\big( (m_j+1) \times (2n_j+2) \big)$-self-dual by construction.
    \details{
        For $1 \le i \le n_j+1$
        \[
            m_j + 1 - \lambda_i^t
            = m_j+1 - \min(\mu_i^t,m_j+1)
            \overset{\cref{jays}}{\le} \mu^t_{2n_j+3-i}. 
        \]
        Thus, $\lambda \subseteq \mu$, as desired.
    }
    
    It remains to prove part~\cref{champagne2}.  We have a commutative diagram of functors
    \begin{equation} \label{muscat}
        \begin{tikzcd}
            \cS \arrow[rr, "\Kar(\bR_\OScat^{m,2n})"] \arrow[d, "\Kar(\bT)"'] & & U_q(\fgl(m|2n))\tmod \arrow[d] \\
            \Kar(\DScat) \arrow[rr, "\Kar(\bR_\DScat^{m,2n})"] & & \Uis(m|2n)\tmod
        \end{tikzcd}
    \end{equation}
    where the right vertical map is restriction, and the left vertical map uses the fact that $\OScat$ is a subcategory of $\DScat$.  When restricted to $\cS$, the functor $\bR_\OScat^{m,2n}$ corresponds to the quantum analogue of Schur--Weyl duality for the $U_q(\mathfrak{gl}(m|2n))$.  Its image is contained in the semisimple category of polynomial representations; see \cite[Prop.~3.1]{BKK00}.  Note that  $\Kar(\bR_\OScat^{m_j,2n_j})(\lambda)$ is the image of the idempotent $\te_\lambda$ of the Iwahori--Hecke algebra $\Hecke_{|\lambda|}$ in the $|\lambda|$-fold tensor product of the standard representation of $U_q(\fgl(m_j|2n_j))$. By quantum Schur--Weyl--Sergeev duality (e.g., see \cite[Th.~5.1]{Mit06}) the image of $\te_\lambda$ is nonzero if and only if $\lambda$ is  an $(m_j,2n_j)$-hook partition.  Thus, for $\lambda \in \Par$,
    \[
        \Kar(\bR_\OScat^{m_j,2n_j})(\lambda) = 0 \iff \lambda_{m_j+1} > 2n_j.
    \]
    (See also \cite[Th.~4.14 and Th.~5.1]{BKK00}.  See \cite[Prop.~3.26]{CW12} for an analogous statement in the non-quantum setting.)  Thus, for $j \in \N$, we have
    \[
        \Kar(\bR_\OScat^{m_j,2n_j})(\nu^{(j)}) = 0
        \qquad \text{and} \qquad
        \Kar(\bR_\OScat^{m_j,2n_j})(\nu^{(j')}) \ne 0 \text{ for } j'<j.
    \]
    By commutativity of \cref{muscat}, this implies that
    \begin{equation} \label{oman}
        \Kar(\bR_\DScat^{m_j,2n_j})(\bT(\nu^{(j)})) = 0
        \qquad \text{and} \qquad
        \Kar(\bR_\DScat^{m_j,2n_j})(\bT(\nu^{(j')})) \ne 0 \text{ for } j'<j.
    \end{equation}
    In particular, the $\Ker \Kar(\bR_\DScat^{m_j,2n_j})$, $j \in \N$, are pairwise distinct.  By \ref{P1} and the equality in \cref{oman}, we have $\Kar(\bR_\DScat^{m_j,2n_j})(R(\nu^{(j)})) = 0$.  Since $\cM_j = \Psi^{-1}(\Tr_{R(\nu^{(j)})} \cM_\one)$ by \cref{tree}, it follows that $\cM_j \subseteq \Ker  \Kar(\bR_\DScat^{m_j,2n_j})$.  Thus, \cref{champagne2} follows from \cref{prosecco}.
\end{proof}

It follows from \cref{champagne} that the inequality in \cref{tricky} is actually equality.

\begin{cor} \label{bibimbap}
    For $\lambda \in \Par$,
    \[
        \dim_\kk \cM(\one,R(\lambda))
        =
        \begin{cases}
            1 & \text{if } \lambda \in \Upsilon, \\
            0 & \text{otherwise}.
        \end{cases}
    \]
    In particular, $\bB_\cM = \{ R(\nu) : \nu \in \Upsilon \}$ and $\bbL = \N$.
\end{cor}

\begin{cor} \label{shoe}
    For $m,n \in \N$, we have that $\ker(\Kar(\bR_\DScat^{m,2n}))$ is generated, as a $\Kar \big( \OScat(q,q^{m-2n}) \big)$-module, by the idempotent $e_{(2n+2)^{m+1}}$, defined as in \cref{PVD}.
\end{cor}

\begin{proof}
    Let $d=m-2n$ and set
    \[
        j =
        \begin{cases}
            n+1 & \text{if } d > 0, \\
            \frac{m+2}{2} & \text{if } d \in -2\N, \\
            \frac{m+1}{2} & \text{if } d \in -2\N-1,
        \end{cases}
    \]
    so that
    \[
        m=m_j,\qquad
        n=n_j,\qquad
        \nu^{(j)} = (2n+2)^{m+1}.
    \]
    Let $\langle e_{\nu^{(j)}}\rangle$ be the $\Kar(\OScat)$-submodule of $\Kar(\DScat)$ generated by $e_{\nu^{(j)}} = 1_{R(\nu^{(j)})} \in \End_{\Kar(\DScat)}(R(\nu^{(j)}))$.  By \cref{oman} and \ref{P1}, we have
    \[
        \Kar(\bR_{\DScat}^{m_j,2n_j}) \big(R(\nu^{(j)})\big) = 0.
    \]
    Hence, $e_{\nu^{(j)}} \in \ker \big(\Kar(\bR_\DScat^{m_j,2n_j}) \big) = \cM_j$ (\cref{champagne}\cref{champagne2}).  Therefore, $\langle e_{\nu^{(j)}}\rangle\subseteq \cM_j$.

    Since
    \[
        \cM \big( \one,R(\nu^{(j)}) \big)
        = 1_{R(\nu^{(j)})} \cM \big( \one,R(\nu^{(j)}) \big)
    \]
    is nonzero by \cref{bibimbap}, while $\cM_{j+1} \big( \one,R(\nu^{(j)}) \big) = 0$ by \cref{champagne}\cref{champagne3}, it follows that $e_{\nu^{(j)}} = 1_{R(\nu^{(j)})}\notin \cM_{j+1}$. Thus, $\langle e_{\nu^{(j)}} \rangle \not\subseteq \cM_{j+1}$.  Since the proper $\Kar(\OScat)$-submodules form the strictly descending chain \eqref{prosecco}, it follows that the only $\Kar(\OScat)$-submodule of $\cM_j$ not contained in $\cM_{j+1}$ is $\cM_j$ itself.  Hence $\langle e_{\nu^{(j)}}\rangle = \cM_j$.
\end{proof}

\begin{cor} \label{ginger}
    The proper $\OScat(q,q^d)$-submodules of $\DScat(q,q^d)$ are all of the form $\ker \left( \bR_\DScat^{m_j,2n_j} \right)$, $j \in \N$, and they form a strictly descending chain
    \[
        \DScat \supsetneq \ker(\bR_\DScat^{m_1,2n_1}) \supsetneq \ker(\bR_\DScat^{m_2,2n_2}) \supsetneq \ker(\bR_\DScat^{m_3,2n_3}) \supsetneq \dotsb.
    \]
    Furthermore, the kernel of $\bR_\DScat^{m,2n}$ is generated, as an $\OScat(q,q^{m-2n})$-module, by $e_{(2n+2)^{m+1}}$.
\end{cor}

\begin{proof}
    This follows from \cref{champagne,shoe,pork}, together with the fact that $\ker(\Kar(\bF)) = \Kar(\ker(\bF))$ for a $\kk$-linear functor $\bF$.
\end{proof}

We can now give a presentation of the category $\Uis(m|2n)\tmod$ of tensor modules for $\Uis(m|2n)$.

\begin{cor} \label{presentation}
    We have an equivalence of $\OScat(q,q^d)$-module categories
    \[
        \Uis(m|2n)\tmod \simeq \Kar \left( \DScat(q,q^{m-2n})/\langle e_{(2n+2)^{m+1}} \rangle \right),
    \]
    where $\langle e_{(2n+2)^{m+1}} \rangle$ is the $\OScat(q,q^{m-2n})$-submodule of $\DScat(q,q^{m-2n})$ generated by $e_{(2n+2)^{m+1}}$.
\end{cor}

\begin{proof}
    The functor $\bR_\DScat^{m,2n}$ is full by \cite[Th.~8.1]{SSS25}.  Thus, it induces a full and faithful functor
    \[
        \DScat(q,q^{m-2n})/\langle e_{(2n+2)^{m+1}} \rangle \to \Uis(m|2n)\tmod
    \]
    by \cref{ginger}.  Since the image of this functor contains the natural module and its dual, we have that $\Kar(\bR_\DScat^{m,2n})$ induces the stated equivalence of categories.
\end{proof}


\bibliographystyle{alphaurl}
\bibliography{ModSubcat}

\end{document}

%% file: ModSubcatMacros.tex
\usepackage{
    amsmath,
    amsfonts,
    amssymb,
    amsthm,
    amscd,
    comment,
    enumitem,
    etoolbox,
    textcomp,   
    gensymb,    
    mathtools,
    mathdots,
    stmaryrd    
}
\usepackage[dvipsnames]{xcolor}
\usepackage[cmtip,all]{xy}
\usepackage[normalem]{ulem} 

\usepackage[letterpaper,margin=1in]{geometry}
\usepackage[T1]{fontenc}
\usepackage{lmodern}            
\usepackage{dsfont}             
\usepackage[colorlinks=true, linkcolor=blue, citecolor=blue, urlcolor=blue, breaklinks=true]{hyperref}

\DeclareFontFamily{OT1}{pzc}{}
\DeclareFontShape{OT1}{pzc}{m}{it}{<-> s * [1.10] pzcmi7t}{}
\DeclareMathAlphabet{\mathpzc}{OT1}{pzc}{m}{it}

\usepackage{zref-clever}

\zcsetup{
  cap,                      
  nameinlink=false,         
  lastsep = {, and }        
}

\zcRefTypeSetup{assumption}{
    Name-sg = Assumption ,
    name-sg = assumption ,
    Name-pl = Assumptions ,
    name-pl = assumptions ,
}
\zcRefTypeSetup{cor}{
    Name-sg = Corollary ,
    name-sg = corollary ,
    Name-pl = Corollaries ,
    name-pl = corollaries ,
}
\zcRefTypeSetup{defin}{
    Name-sg = Definition ,
    name-sg = definition ,
    Name-pl = Definitions ,
    name-pl = definitions ,
}
\zcRefTypeSetup{enumi}{
    Name-sg = {} ,
    name-sg = {} ,
    Name-pl = {} ,
    name-pl = {} ,
}
\zcRefTypeSetup{equation}{
    Name-sg = {} ,
    name-sg = {} ,
    Name-pl = {} ,
    name-pl = {} ,
}
\zcRefTypeSetup{eg}{
    Name-sg = Example ,
    name-sg = example ,
    Name-pl = Examples ,
    name-pl = examples ,
}
\zcRefTypeSetup{lem}{
    Name-sg = Lemma ,
    name-sg = lemma ,
    Name-pl = Lemmas ,
    name-pl = lemmas ,
}
\zcRefTypeSetup{prop}{
    Name-sg = Proposition ,
    name-sg = proposition ,
    Name-pl = Propositions ,
    name-pl = propositions ,
}
\zcRefTypeSetup{rem}{
    Name-sg = Remark ,
    name-sg = remark ,
    Name-pl = Remarks ,
    name-pl = remarks ,
}
\zcRefTypeSetup{theo}{
    Name-sg = Theorem ,
    name-sg = theorem ,
    Name-pl = Theorems ,
    name-pl = theorems ,
}

\zcsetup{
  countertype = {
    enumi=item, enumii=item, enumiii=item, enumiv=item
  },
  counterresetby = {
    enumii=enumi, enumiii=enumii, enumiv=enumiii
  }
}
\zcRefTypeSetup{item}{
    Name-sg = {} ,
    name-sg = {} ,
    Name-pl = {} ,
    name-pl = {} ,
}

\AddToHook{env/assumption/begin}{\zcsetup{countertype={theo=assumption}}}
\AddToHook{env/cor/begin}{\zcsetup{countertype={theo=cor}}}
\AddToHook{env/defin/begin}{\zcsetup{countertype={theo=defin}}}
\AddToHook{env/eg/begin}{\zcsetup{countertype={theo=eg}}}
\AddToHook{env/lem/begin}{\zcsetup{countertype={theo=lem}}}
\AddToHook{env/prop/begin}{\zcsetup{countertype={theo=prop}}}
\AddToHook{env/rem/begin}{\zcsetup{countertype={theo=rem}}}

\newcommand{\cref}[1]{\zcref{#1}}
\newcommand{\Cref}[1]{\zcref[S]{#1}}

\makeatletter
\def\namedlabel#1#2{\begingroup
    (#2)%
    \def\@currentlabel{(#2)}%
    \phantomsection\label{#1}\endgroup
}
\makeatother

\newcommand\C{\mathbb{C}}
\newcommand\bbF{\mathbb{F}}
\newcommand\bbL{\mathbb{L}}
\newcommand\N{\mathbb{N}}
\newcommand\Q{\mathbb{Q}}

\newcommand\Z{\mathbb{Z}}
\newcommand\bbK{\mathbb{K}}
\newcommand\kk{\Bbbk}
\newcommand\one{\mathds{1}}

\newcommand\bA{\mathbf{A}}

\newcommand\bB{\mathbf{B}}

\newcommand\bT{\mathbf{T}}

\newcommand\cC{\mathcal{C}}
\newcommand\cD{\mathcal{D}}

\newcommand\cS{\mathcal{S}}

\newcommand\cM{\mathcal{M}}
\newcommand\cN{\mathcal{N}}

\newcommand\tN{\mathtt{N}}

\newcommand\rev{\textup{rev}}

\newcommand{\Ker}{\textup{Ker}}
\newcommand{\te}{\tilde{e}}

\newcommand\tmod{\textup{-tmod}}
\newcommand{\obj}[1]{\operatorname{im}(#1)}             
\newcommand{\Par}{\mathrm{Par}}

\newcommand{\xrightarrowdbl}[2][]{
    \xrightarrow[#1]{#2}\mathrel{\mkern-14mu}\rightarrow
}

\newcommand{\Laurent}[1]{
    (\!( #1 )\!)
}

\newcommand{\fgl}{\mathfrak{gl}}

\newcommand{\Ualg}{\mathrm{U}}
\newcommand{\Ui}{\Ualg^\imath}

\newcommand{\Uis}{\Ui_\sigma}

\newcommand{\Us}{\Ualg_\sigma}

\newcommand\cEnd{\mathpzc{End}}     
\newcommand\OScat{\mathpzc{OS}}     
\newcommand\DScat{\mathpzc{DS}}     
\newcommand{\unit}[1]{{1}_{#1}}

\newcommand{\upobj}{{\mathord{\uparrow}}}
\newcommand{\downobj}{{\mathord{\downarrow}}}

\newcommand\bF{\mathbf{F}}

\newcommand\bR{\mathbf{R}}

\newcommand{\Hecke}{\mathrm{H}}
\newcommand{\DSalg}{\mathrm{DS}}
\newcommand{\word}{\langle \upobj,\downobj \rangle}

\DeclareMathOperator{\add}{add}                 

\DeclareMathOperator{\End}{End}

\DeclareMathOperator{\Hom}{Hom}
\DeclareMathOperator{\id}{id}
\DeclareMathOperator{\im}{im}                   
\DeclareMathOperator{\Ind}{Ind}                 
\DeclareMathOperator{\indec}{indec}             
\DeclareMathOperator{\Mor}{Mor}                 

\DeclareMathOperator{\Ob}{Ob}                   
\DeclareMathOperator{\PI}{PI}                   
\DeclareMathOperator{\Kar}{Kar}
\DeclareMathOperator{\rank}{rank}
\DeclareMathOperator{\Res}{Res}                 

\DeclareMathOperator{\Span}{span}
\DeclareMathOperator{\Submod}{Submod}           
\DeclareMathOperator{\Thick}{Thick}

\DeclareMathOperator{\Tr}{Tr}

\usepackage{tikz}
\usetikzlibrary{arrows.meta,patterns}
\usetikzlibrary{decorations.markings,decorations.pathreplacing}
\usepackage{tikz-cd}

\tikzset{multi/.style={very thick}}
\tikzset{
    anchorbase/.style={
        >=To,
        line cap = round, line join = round,
        baseline={([yshift=-0.5ex]current bounding box.center)},
    }
}
\tikzset{
    centerzero/.style={
        >=To,
        line cap = round, line join = round,
        baseline={([yshift=-0.5ex](#1))},
        },
    centerzero/.default={0,0}
}
\tikzset{   
    wipe/.style={
        white,
        line cap = butt,    
        line width=4pt
    }
} 
\tikzset{
    over/.style={
        preaction={
            draw=white,
            line cap = butt,    
            line width=4pt,
            -{},                
        }
    },
}

\newcommand{\strandlabel}[1]{$\scriptstyle{#1}$}
\newcommand{\botlabel}[1]{node[anchor=north] {\strandlabel{#1}}}
\newcommand{\toplabel}[1]{node[anchor=south] {\strandlabel{#1}}}
\newcommand{\braidup}{to[out=up,in=down]}
\newcommand{\braiddown}{to[out=down,in=up]}
\newcommand{\coupon}[2]{ 
    \draw (#1) node[inner sep=2pt,draw,rounded corners,fill=black!20!white] {$\scriptstyle{#2}$}
}
\newcommand{\genbox}[3]{
    \filldraw[fill=black!20!white,rounded corners] (#1) rectangle (#2) node[midway] {\strandlabel{#3}}
}
\newcommand{\indicate}[2]{
    \draw[rounded corners,violet,densely dotted] (#1) rectangle (#2)
}

\newcommand{\opendot}[1]{
    \node[white] at (#1) {$\scriptstyle{\bullet}$};
    \node at (#1) {$\scriptstyle{\circ}$}
}
\newcommand{\bubright}[1]{
    \draw[->] (#1)+(0.2,0) arc(360:0:0.2)
}

\newcommand{\rightbub}{
    \begin{tikzpicture}[centerzero]
        \bubright{0,0};
    \end{tikzpicture}
}

\newcommand{\posupcross}{
    \begin{tikzpicture}[centerzero]
        \draw[->] (0.2,-0.2) -- (-0.2,0.2);
        \draw[wipe] (-0.2,-0.2) -- (0.2,0.2);
        \draw[->] (-0.2,-0.2) -- (0.2,0.2);
    \end{tikzpicture}
}

\newcommand{\negupcross}{
    \begin{tikzpicture}[centerzero]
        \draw[->] (-0.2,-0.2) -- (0.2,0.2);
        \draw[wipe] (0.2,-0.2) -- (-0.2,0.2);
        \draw[->] (0.2,-0.2) -- (-0.2,0.2);
    \end{tikzpicture}
}

\newcommand{\posdowncross}{
    \begin{tikzpicture}[centerzero]
        \draw[<-] (0.2,-0.2) -- (-0.2,0.2);
        \draw[wipe] (-0.2,-0.2) -- (0.2,0.2);
        \draw[<-] (-0.2,-0.2) -- (0.2,0.2);
    \end{tikzpicture}
}

\newcommand{\togupdown}{
    \begin{tikzpicture}[centerzero]
        \draw[->] (0,0) -- (0,0.2);
        \draw[->] (0,0) -- (0,-0.2);
        \opendot{0,0};
    \end{tikzpicture}
}
\newcommand{\togdownup}{
    \begin{tikzpicture}[centerzero]
        \draw[->] (0,0.2) -- (0,0);
        \draw[->] (0,-.2) -- (0,0);
        \opendot{0,0};
    \end{tikzpicture}
}

\newcommand{\rightcup}{
    \begin{tikzpicture}[anchorbase]
        \draw[->] (-0.15,0.15) -- (-0.15,0) arc(180:360:0.15) -- (0.15,0.15);
    \end{tikzpicture}
}
\newcommand{\leftcup}{
    \begin{tikzpicture}[anchorbase]
        \draw[<-] (-0.15,0.15) -- (-0.15,0) arc(180:360:0.15) -- (0.15,0.15);
    \end{tikzpicture}
}
\newcommand{\rightcap}{
    \begin{tikzpicture}[anchorbase]
        \draw[->] (-0.15,-0.15) -- (-0.15,0) arc(180:0:0.15) -- (0.15,-0.15);
    \end{tikzpicture}
}
\newcommand{\leftcap}{
    \begin{tikzpicture}[anchorbase]
        \draw[<-] (-0.15,-0.15) -- (-0.15,0) arc(180:0:0.15) -- (0.15,-0.15);
    \end{tikzpicture}
}

\newtheorem{theo}{Theorem}[section]

\newtheorem{prop}[theo]{Proposition}
\newtheorem{lem}[theo]{Lemma}
\newtheorem{cor}[theo]{Corollary}

\theoremstyle{definition}
\newtheorem{assumption}[theo]{Assumption}
\newtheorem{defin}[theo]{Definition}
\newtheorem{rem}[theo]{Remark}
\newtheorem{eg}[theo]{Example}

\numberwithin{equation}{section}
\allowdisplaybreaks

\setenumerate[1]{label=(\alph*)}          

\newtoggle{comments}
\newtoggle{details}
\newtoggle{detailsnote}

\iftoggle{comments}{%
    \usepackage[notref,notcite]{showkeys}   
    \newcommand{\acomments}[1]{
        \ \\
        {\color{red}
            \textbf{AS:} #1
        }
        \ \\
    }
    \newcommand{\hcomments}[1]{
        \ \\
        {\color{cyan}
            \textbf{HS:} #1
        }
        \ \\
    }

    \newcommand{\ycomments}[1]{
        \ \\
        {\color{blue}
            \textbf{YS:} #1
        }
        \ \\
    }
    \newcommand{\question}[1]{
        \ \\
        {\color{blue}
            \textbf{Question:} #1
        }
        \ \\
    }
    }{%
        \newcommand{\acomments}[1]{}
        \newcommand{\hcomments}[1]{}
        \newcommand{\ycomments}[1]{}
        \newcommand{\question}[1]{}
    }

\iftoggle{details}{%
    \newcommand{\details}[1]{
        \ \\
        {\color{OliveGreen}
            \textbf{Details:} #1
        }
        \\
    }
}{%
    \newcommand{\details}[1]{\ignorespaces}
}

%% file: ModSubcat.bbl
\begin{thebibliography}{EGNO15}

\bibitem[AK02]{AK02}
Y.~Andr\'{e} and B.~Kahn.
\newblock Nilpotence, radicaux et structures mono\"{\i}dales.
\newblock {\em Rend. Sem. Mat. Univ. Padova}, 108:107--291, 2002.
\newblock With an appendix by Peter O'Sullivan.
\newblock \href {https://arxiv.org/abs/math/0203273}
  {\path{arXiv:math/0203273}}.

\bibitem[BCNR17]{BCNR17}
J.~Brundan, J.~Comes, D.~Nash, and A.~Reynolds.
\newblock A basis theorem for the affine oriented {B}rauer category and its
  cyclotomic quotients.
\newblock {\em Quantum Topol.}, 8(1):75--112, 2017.
\newblock \href {https://arxiv.org/abs/1404.6574} {\path{arXiv:1404.6574}},
  \href {https://doi.org/10.4171/QT/87} {\path{doi:10.4171/QT/87}}.

\bibitem[Ben98]{Ben98}
D.~J. Benson.
\newblock {\em Representations and cohomology. {I}}, volume~30 of {\em
  Cambridge Studies in Advanced Mathematics}.
\newblock Cambridge University Press, Cambridge, second edition, 1998.
\newblock Basic representation theory of finite groups and associative
  algebras.

\bibitem[BK19]{BK19}
M.~Balagovi\'c and S.~Kolb.
\newblock Universal {K}-matrix for quantum symmetric pairs.
\newblock {\em J. Reine Angew. Math.}, 747:299--353, 2019.
\newblock \href {https://arxiv.org/abs/1507.06276} {\path{arXiv:1507.06276}},
  \href {https://doi.org/10.1515/crelle-2016-0012}
  {\path{doi:10.1515/crelle-2016-0012}}.

\bibitem[BKK00]{BKK00}
G.~Benkart, S.-J. Kang, and M.~Kashiwara.
\newblock Crystal bases for the quantum superalgebra
  {$U_q(\mathfrak{gl}(m,n))$}.
\newblock {\em J. Amer. Math. Soc.}, 13(2):295--331, 2000.
\newblock \href {https://arxiv.org/abs/math/9810092}
  {\path{arXiv:math/9810092}}, \href
  {https://doi.org/10.1090/S0894-0347-00-00321-0}
  {\path{doi:10.1090/S0894-0347-00-00321-0}}.

\bibitem[Bro13]{Bro13}
A.~Brochier.
\newblock Cyclotomic associators and finite type invariants for tangles in the
  solid torus.
\newblock {\em Algebr. Geom. Topol.}, 13(6):3365--3409, 2013.
\newblock \href {https://doi.org/10.2140/agt.2013.13.3365}
  {\path{doi:10.2140/agt.2013.13.3365}}.

\bibitem[Bru17]{Bru17}
J.~Brundan.
\newblock Representations of the oriented skein category.
\newblock 2017.
\newblock \href {https://arxiv.org/abs/1712.08953} {\path{arXiv:1712.08953}}.

\bibitem[CH17]{CH17}
J.~Comes and T.~Heidersdorf.
\newblock Thick ideals in {D}eligne's category {$\underline{\rm Re}{\rm
  p}(O_\delta)$}.
\newblock {\em J. Algebra}, 480:237--265, 2017.
\newblock \href {https://arxiv.org/abs/1507.06728} {\path{arXiv:1507.06728}},
  \href {https://doi.org/10.1016/j.jalgebra.2017.01.050}
  {\path{doi:10.1016/j.jalgebra.2017.01.050}}.

\bibitem[CO11]{CO11}
J.~Comes and V.~Ostrik.
\newblock On blocks of {D}eligne's category {$\underline{\rm Re}{\rm p}(S_t)$}.
\newblock {\em Adv. Math.}, 226(2):1331--1377, 2011.
\newblock \href {https://arxiv.org/abs/0910.5695} {\path{arXiv:0910.5695}},
  \href {https://doi.org/10.1016/j.aim.2010.08.010}
  {\path{doi:10.1016/j.aim.2010.08.010}}.

\bibitem[Cou18]{Cou18}
K.~Coulembier.
\newblock Tensor ideals, {D}eligne categories and invariant theory.
\newblock {\em Selecta Math. (N.S.)}, 24(5):4659--4710, 2018.
\newblock \href {https://arxiv.org/abs/1712.06248} {\path{arXiv:1712.06248}},
  \href {https://doi.org/10.1007/s00029-018-0433-z}
  {\path{doi:10.1007/s00029-018-0433-z}}.

\bibitem[CW12]{CW12}
S.-J. Cheng and W.~Wang.
\newblock {\em Dualities and representations of {L}ie superalgebras}, volume
  144 of {\em Graduate Studies in Mathematics}.
\newblock American Mathematical Society, Providence, RI, 2012.
\newblock \href {https://doi.org/10.1090/gsm/144} {\path{doi:10.1090/gsm/144}}.

\bibitem[EGNO15]{ENGO15}
P.~Etingof, S.~Gelaki, D.~Nikshych, and V.~Ostrik.
\newblock {\em Tensor categories}, volume 205 of {\em Mathematical Surveys and
  Monographs}.
\newblock American Mathematical Society, Providence, RI, 2015.
\newblock \href {https://doi.org/10.1090/surv/205}
  {\path{doi:10.1090/surv/205}}.

\bibitem[Enr07]{Enr07}
B.~Enriquez.
\newblock Quasi-reflection algebras and cyclotomic associators.
\newblock {\em Selecta Math. (N.S.)}, 13(3):391--463, 2007.
\newblock \href {https://doi.org/10.1007/s00029-007-0048-2}
  {\path{doi:10.1007/s00029-007-0048-2}}.

\bibitem[GL96]{GL96}
J.~J. Graham and G.~I. Lehrer.
\newblock Cellular algebras.
\newblock {\em Invent. Math.}, 123(1):1--34, 1996.
\newblock \href {https://doi.org/10.1007/BF01232365}
  {\path{doi:10.1007/BF01232365}}.

\bibitem[HO01]{Har01}
R.~H\"aring-Oldenburg.
\newblock Actions of tensor categories, cylinder braids and their {K}auffman
  polynomial.
\newblock {\em Topology Appl.}, 112(3):297--314, 2001.
\newblock \href {https://doi.org/10.1016/S0166-8641(00)00006-7}
  {\path{doi:10.1016/S0166-8641(00)00006-7}}.

\bibitem[Kra15]{Kra15}
H.~Krause.
\newblock Krull-{S}chmidt categories and projective covers.
\newblock {\em Expo. Math.}, 33(4):535--549, 2015.
\newblock \href {https://arxiv.org/abs/1410.2822} {\path{arXiv:1410.2822}},
  \href {https://doi.org/10.1016/j.exmath.2015.10.001}
  {\path{doi:10.1016/j.exmath.2015.10.001}}.

\bibitem[KX98]{KX98}
S.~K\"onig and C.~Xi.
\newblock On the structure of cellular algebras.
\newblock In {\em Algebras and modules, {II} ({G}eiranger, 1996)}, volume~24 of
  {\em CMS Conf. Proc.}, pages 365--386. Amer. Math. Soc., Providence, RI,
  1998.

\bibitem[Lam01]{Lam01}
T.~Y. Lam.
\newblock {\em A first course in noncommutative rings}, volume 131 of {\em
  Graduate Texts in Mathematics}.
\newblock Springer-Verlag, New York, second edition, 2001.
\newblock \href {https://doi.org/10.1007/978-1-4419-8616-0}
  {\path{doi:10.1007/978-1-4419-8616-0}}.

\bibitem[Mit06]{Mit06}
H.~Mitsuhashi.
\newblock Schur-{W}eyl reciprocity between the quantum superalgebra and the
  {I}wahori-{H}ecke algebra.
\newblock {\em Algebr. Represent. Theory}, 9:309--322, 2006.
\newblock \href {https://arxiv.org/abs/math/0506156}
  {\path{arXiv:math/0506156}}, \href
  {https://doi.org/10.1007/s10468-006-9014-5}
  {\path{doi:10.1007/s10468-006-9014-5}}.

\bibitem[QS19]{QS19}
H.~Queffelec and A.~Sartori.
\newblock Mixed quantum skew {H}owe duality and link invariants of type {$A$}.
\newblock {\em J. Pure Appl. Algebra}, 223(7):2733--2779, 2019.
\newblock \href {https://arxiv.org/abs/1504.01225} {\path{arXiv:1504.01225}},
  \href {https://doi.org/10.1016/j.jpaa.2018.09.014}
  {\path{doi:10.1016/j.jpaa.2018.09.014}}.

\bibitem[RSS24]{RSS24}
H.~Rui, M.~Si, and L.~Song.
\newblock The {J}ucys-{M}urphy basis and semisimplicity criteria for the
  {$q$}-{B}rauer algebra.
\newblock {\em Lett. Math. Phys.}, 114(4):Paper No. 104, 36, 2024.
\newblock \href {https://arxiv.org/abs/2211.14756} {\path{arXiv:2211.14756}},
  \href {https://doi.org/10.1007/s11005-024-01850-8}
  {\path{doi:10.1007/s11005-024-01850-8}}.

\bibitem[Sel11]{Sel11}
P.~Selinger.
\newblock A survey of graphical languages for monoidal categories.
\newblock In {\em New structures for physics}, volume 813 of {\em Lecture Notes
  in Phys.}, pages 289--355. Springer, Heidelberg, 2011.
\newblock \href {https://arxiv.org/abs/0908.3347} {\path{arXiv:0908.3347}},
  \href {https://doi.org/10.1007/978-3-642-12821-9_4}
  {\path{doi:10.1007/978-3-642-12821-9_4}}.

\bibitem[SSS25]{SSS25}
H.~Salmasian, A.~Savage, and Y.~Shen.
\newblock The disoriented skein and iquantum {B}rauer categories.
\newblock {\em Forum Math. Sigma}, 2025.
\newblock To appear.
\newblock \href {https://arxiv.org/abs/2507.12328} {\path{arXiv:2507.12328}}.

\bibitem[tD98]{tom98}
T.~tom Dieck.
\newblock Categories of rooted cylinder ribbons and their representations.
\newblock volume 494, pages 35--63. 1998.
\newblock Dedicated to Martin Kneser on the occasion of his 70th birthday.
\newblock \href {https://doi.org/10.1515/crll.1998.010}
  {\path{doi:10.1515/crll.1998.010}}.

\bibitem[tDHO98]{tH98}
T.~tom Dieck and R.~H\"aring-Oldenburg.
\newblock Quantum groups and cylinder braiding.
\newblock {\em Forum Math.}, 10(5):619--639, 1998.
\newblock \href {https://doi.org/10.1515/form.10.5.619}
  {\path{doi:10.1515/form.10.5.619}}.

\bibitem[Tur89]{Tur89}
V.~G. Turaev.
\newblock Operator invariants of tangles, and {$R$}-matrices.
\newblock {\em Izv. Akad. Nauk SSSR Ser. Mat.}, 53(5):1073--1107, 1135, 1989.
\newblock \href {https://doi.org/10.1070/IM1990v035n02ABEnH000711}
  {\path{doi:10.1070/IM1990v035n02ABEnH000711}}.

\end{thebibliography}
